\documentclass[final,leqno,onefignum,onetabnum]{siamltex1213}

\usepackage{amsmath,amsfonts,amssymb}
\usepackage{graphicx,color,subfigure}
\usepackage[usenames,dvipsnames]{xcolor}
\usepackage{mathdots}
\usepackage{array}
\usepackage{soul}
\usepackage{extarrows}
\usepackage{wrapfig}

\newcommand{\rtext}[1]{#1}
\newcommand{\rrtext}[1]{{#1}}
\newcommand{\sout}[1]{}

\newcommand{\abstractContent}{
Convolving the output of Discontinuous Galerkin (\DG) computations
using spline filters can improve both smoothness and accuracy
of the output.
At domain boundaries, these filters have to be one-sided 
\rtext{for non-periodic boundary conditions}.
Recently, 
position-dependent smoothness-increasing
accuracy-preserving (\bsiac) filters were shown to be 
a superset of the top-of-the-line
one-sided \rlkv\ and \srv\ filters.
Since PSIAC filters can be formulated symbolically, 
convolution with PSIAC filters reduces to a sequences of small inner products
with local \DG\ output and hence provides a more 
stable and efficient implementation.

The paper focuses on the remarkable fact that,  
for the canonical 
hyperbolic test equation,
new piecewise constant \bsiac\ filters of small support
outperform the 
top-of-the-line boundary filters in the literature.
\rtext{Numerical experiments show that} these least-degree filters 
reduce the error at the boundaries
to less than the error even of the symmetric filters of the same support
\ccchg{that are applied in the interior}.
Due to their simplicity, and since this least degree filter 
has an exact symbolic form, convolution is stable as well as efficient;
and derivatives of the convolved output are easy to compute.
}

\newcommand{\UFaffiliation}{Department of Computer \& Information
Science \& Engineering, University of Florida}
\newcommand{\NSFsupport}{This work was supported in part by NSF grant
 CCF-1117695 and NIH grant R01 LM011300-01}

\newcommand{\idf}[1]{\chi_{#1}}
\newcommand{\idfunc}[2]{\idf{[#1,#2]}}
\newcommand{\Kcal}{\mathcal{K}}
\newcommand{\ftex}[1]{\cfno{#1}} 
\newcommand{\eqtext}[1]{\xlongequal{\text{#1}}} 
\newcommand{\puttext}[3]{\put(#1,#2){\makebox(0,0)[c]{#3}}}

\newcommand{\putfnsize}[3]{\put(#1,#2){\makebox(0,0)[c]{{\footnotesize #3}}}}
\newcommand{\putvector}[5]{\put(#1,#2){\vector(#3,#4){#5}}}
\newcommand{\crate}{\rho}

\newcommand{\hk}{h_1}
\newcommand{\bbfun}[2]{B_{#1}^{#2}}

\newcommand{\hh}{h}

\newcommand{\dDG}{d} 
\newcommand{\ddg}{\dDG} 
\newcommand{\df}{k} 

\newcommand{\nc}{r} 
\newcommand{\drepro}{\nc} 
\newcommand{\idxfst}{0}
\newcommand{\idxlst}{{j_\nc}}

\newcommand{\nkft}{n} 
\newcommand{\lst}{\nkft} %
\newcommand{\sft}{j}

\newcommand{\nkdg}{m} 
\newcommand{\ffn}{f} 
\newcommand{\flt}[1]{\ffn_{#1}}
\newcommand{\cf}[2]{\ffn_{#1;#2}}

\newcommand{\cfno}[1]{\ffn_{#1}}

\newcommand{\lft}{a} 
\newcommand{\rgt}{b} 
\newcommand{\idxft }{\scalebox{.7}{\ensuremath{\mathcal{J}}}} 

\newcommand{\filter}{kernel}
\newcommand{\orderft}{reproduction degree}

\newcommand{\refx}{\lambda}
\newcommand{\brefx}{\pmb{\lambda}}
\newcommand{\ipmat}{G} 
\newcommand{\ipmatl}[1]{\ipmat_{\refx_{#1}}} 

\newcommand{\kft}{t}
\newcommand{\xx}{x}
\newcommand{\bkft}{\mathbf{\kft}}
\newcommand{\knotsftv}{\boldsymbol \kft}   

\newcommand{\tr}{t}

\definecolor{dblue}{rgb}{0,.3,.6}
\definecolor{dgreen}{rgb}{0,.4,0}

\newcommand{\finaltime}{T}
\newcommand{\ftime}{final time}

\newcommand{\ix}{\omega}
\newcommand{\midx}{{\boldsymbol\omega}}
\newcommand{\tm}{\tau} 
\newcommand{\tmend}{\finaltime} 
\newcommand{\tb}{\mathbf{t}}

\newcommand{\sk}{f} 
\newcommand{\blend}[1]{\alpha_{#1}}
\newcommand{\drm}{\mathrm{d}} 

\newcommand{\Bsp}[2]{B(#1\,|\,#2)}
\newcommand{\Nbfi}[1]{\phi_{#1}}
\newcommand{\Nbf}[3]{\Nbfi{#1}(#2\,;\,#3)}

\newcommand{\ub}{\mathbf{u}}
\newcommand{\Ical}{\mathcal{I}}
\newcommand{\adia}{A}
\newcommand{\kernel}{kernel}

\newcommand{\cchg}[1]{{#1}}
\newcommand{\ccchg}[1]{{#1}}
\newcommand{\figref}[1]{Fig.~\ref{#1}}
\newcommand{\testref}[1]{Test~\ref{#1}}
\newcommand{\secref}[1]{Section~\ref{#1}}
\newcommand{\thmref}[1]{Theorem~\ref{#1}}

\newcommand{\equaref}[1]{Eq.~\eqref{#1}}
\newcommand{\defref}[1]{Definition~\ref{#1}}
\newcommand{\exref}[1]{Example~\ref{#1}}

\newcommand{\mysecondproof}[1]{}
\newcommand{\jorg}[1]{\textcolor{blue}{#1}}

\newcommand{\dg}{\delta}

\newcommand{\DG}{DG}

\newcommand{\convol}{\ast}
\newtheorem{testproblem}{\bf Test}[section]

\newcommand{\dgout}{u}

\newcommand{\gs}{s} 
\newcommand{\knts}{\sft:\sft+\df+1}  
\newcommand{\tsft}{\kft_{\knts}}  
\newcommand{\divdif}{\mathord
   {\kern.43em\vrule width.6pt height5.6pt depth-.28pt\kern-.43em\Delta}}
\newcommand{\mc}{\mathbf{\cfno{}}}
\newcommand{\R}{\mathbb{R}}

\newcommand{\siac}{SIAC}
\newcommand{\bsiac}{PSIAC}
\newcommand{\proto}{prototype}

\newcommand{\Nbs}[2]{N(#1\,|\,#2)}

\newcommand{\ds}{d_s} 

\newcommand{\vh}{\mathbb{P}^d_h}
\newcommand{\vecSymb}{\boldsymbol{\mathcal{V}}}
\newcommand{\srv}{{SRV}}
\newcommand{\rlkv}{{RLKV}}
\newcommand{\newft}{{NP$_0$}}

\newcommand{\newftk}{{NP$_k$}}
\newcommand{\bbname}{Bernstein-B\'ezier}

\newcommand{\RR}{\mathbb{R}}

\newcommand{\uex}{\dgout}

\newcommand{\probR}[3]{
Test~\ref{#1} (#2 wave-speed, #3 boundary condition)}
\newcommand{\probRef}[3]{\textbf{Convergence rates}\probR{#1}{#2}{#3}.}
\newcommand{\probErr}[3]{\textbf{Error}\probR{#1}{#2}{#3}.}

\newcommand{\dft}{\df}
\newcommand{\bdy}{boundary region}

\newtheorem{prop}{Proposition}[section]
\newtheorem{example}{Example}[section]
 
\title{Explicit Least-degree Boundary Filters for Discontinuous Galerkin\thanks{\NSFsupport}}

 \author{
  Dang-Manh Nguyen\thanks{\UFaffiliation. 
  (\email{nguyendm@ufl.edu, jorg@cise.ufl.edu}).}  
  \, \and 
   J\"org Peters$^\dagger$   
 }

\begin{document}         
\maketitle   
 
 \begin{abstract}
 \abstractContent 
 \end{abstract} 
 
 
 
 \pagestyle{myheadings}
 \thispagestyle{plain}
 \markboth{ D-M.~Nguyen, J.~Peters}
 {Explicit Least-degree Boundary Filters for Discontinuous Galerkin}
  
\newcommand{\IL}{{\it left}}
\newcommand{\IM}{{\it middle}}
\newcommand{\IR}{{\it right}}

\section{Introduction}\label{sec:intro}
The output of Discontinuous Galerkin (\DG) computations often captures higher
order moments of the true solution \cite{moment78}.
Therefore post-processing \DG\ output
by convolution with splines can improve both smoothness and accuracy 
\cite{Bramble77,siac02hyperbolic,siac14estimate}. 
In the interior of the domain of computation, symmetric
smoothness increasing accuracy conserving (SIAC) spline filters 
have been demonstrated to provide optimal accuracy
\cite{siac02hyperbolic}.
Near boundaries of the computational domain such symmetric filters need to be 
complemented by one-sided filters 
\rtext{to accomodate non-periodic boundary data}.
The point-wise error and stability, of the pioneering Ryan-Shu boundary filters 
\cite{siac2003} have been noticeably improved upon, during the last decade
\cite{siac2011,siac2012,siac2014,siac16der}.  

\begin{figure}[ht!]     
	\def\wid{0.45\linewidth} 
	\centering
	\subfigure[Error filtered by \rlkv\ in the \bdy]
	{\includegraphics[width=\wid]{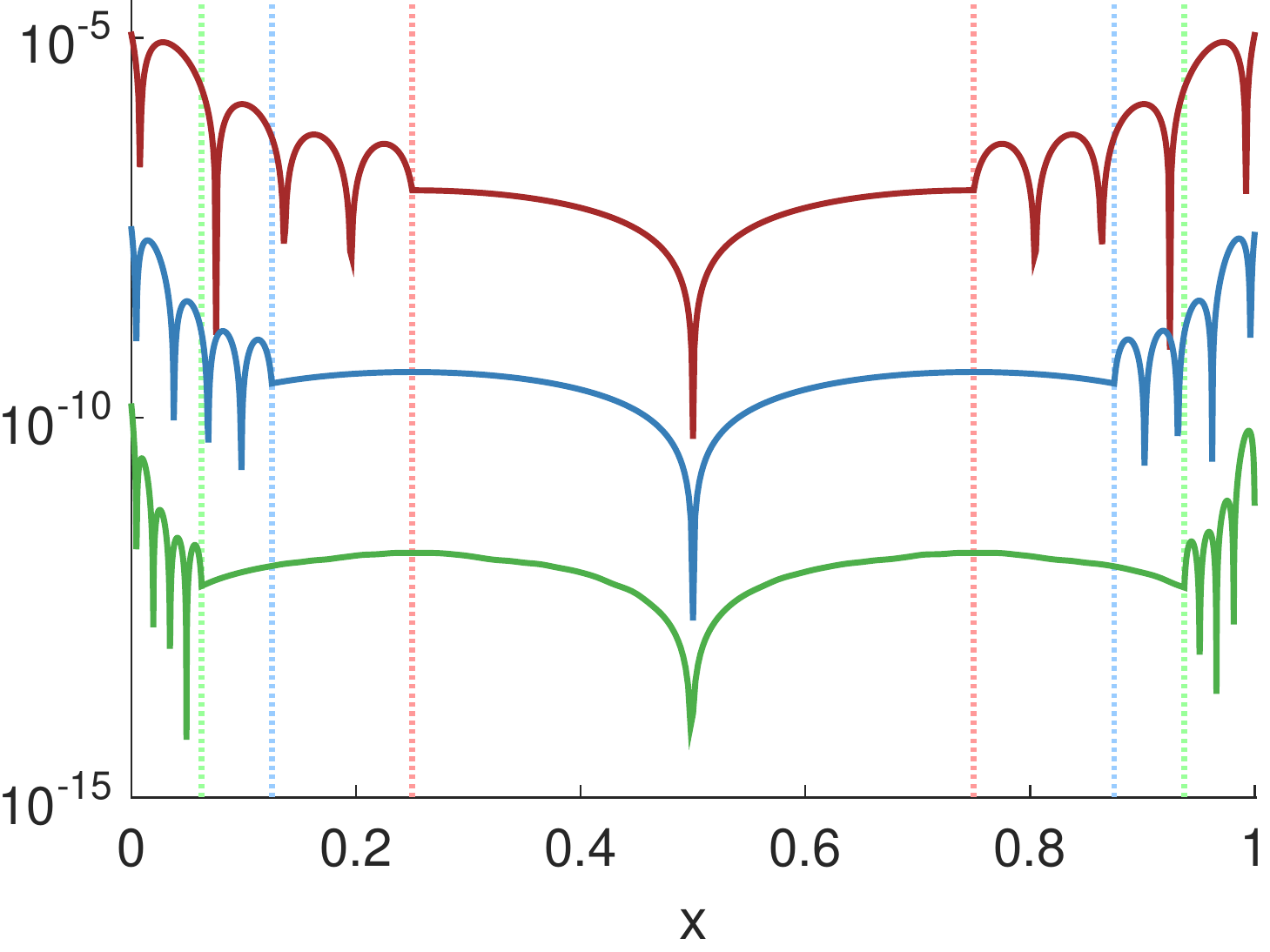}}   
	\hfill  
	\subfigure[Error filtered by \newft\ in the \bdy]
        {\includegraphics[width=\wid]
		{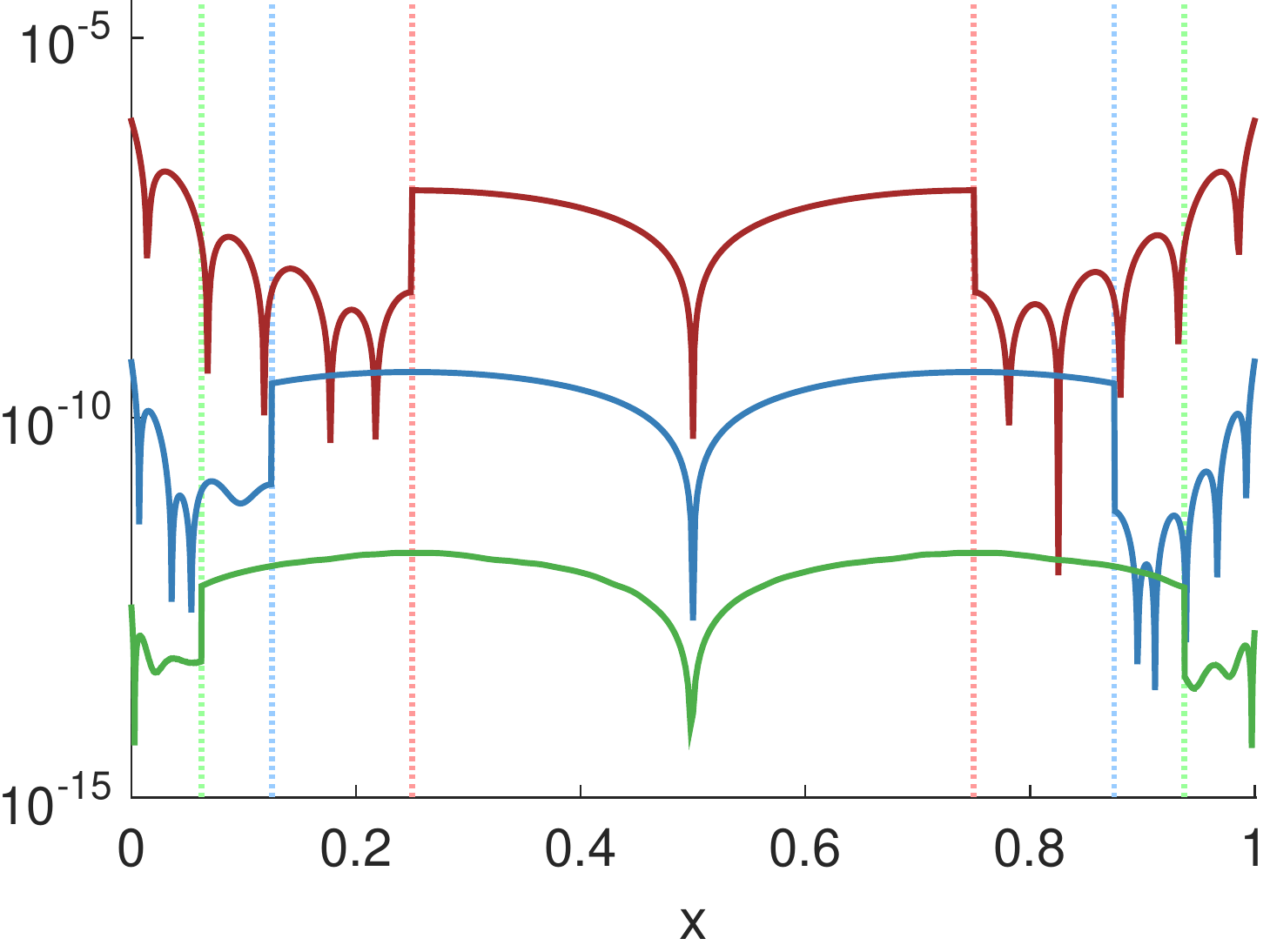}} 
	\subfigure[\DG\ error at \ftime\ $\tmend=1$]
           {\includegraphics[width=\wid]{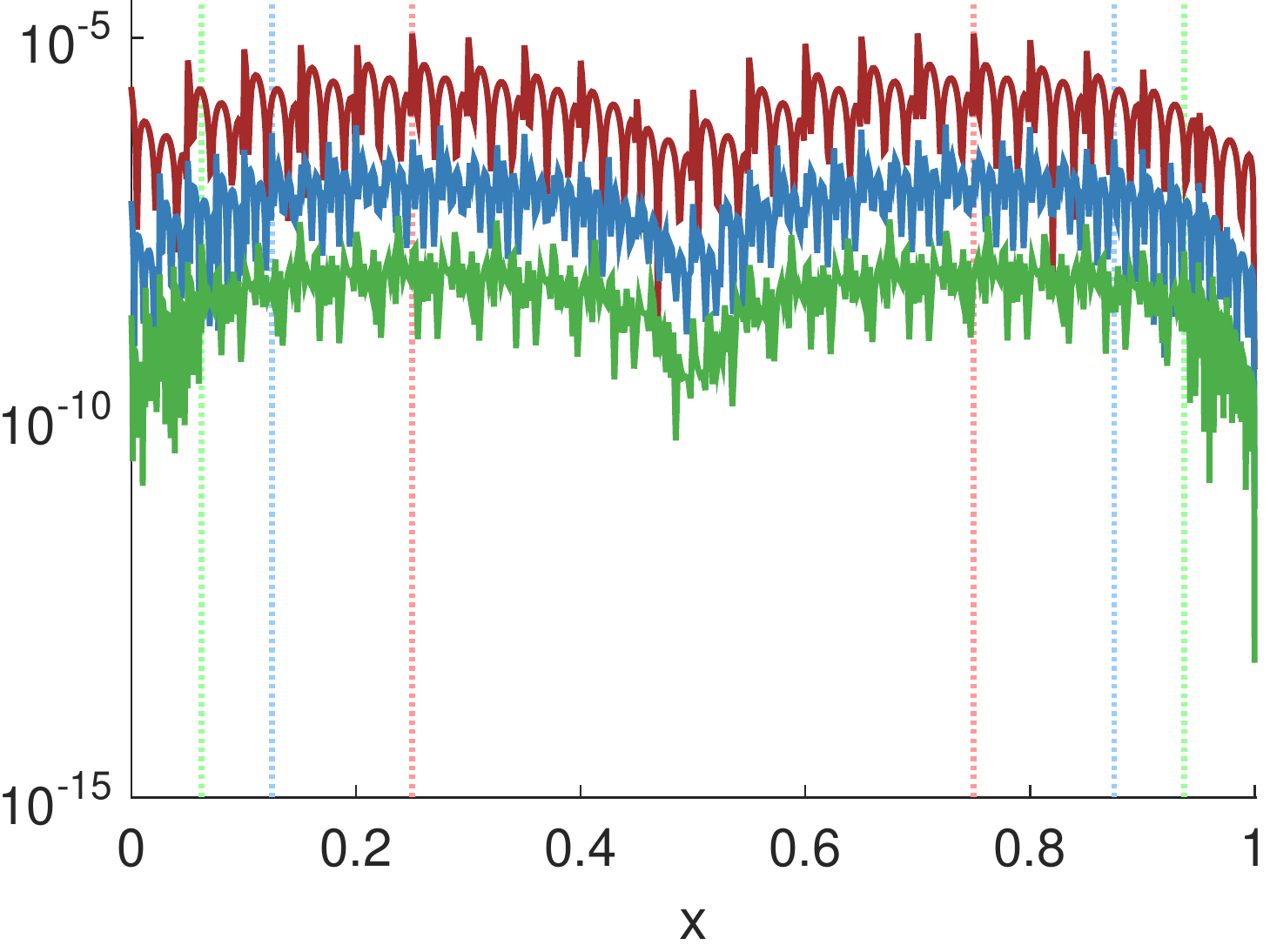}}
	\hfill 
	\subfigure[Maximal errors \ccchg{restricted to}
        the left $\rhd$ and right $\lhd$ \ccchg{\bdy}]
           {\includegraphics[width=\wid]{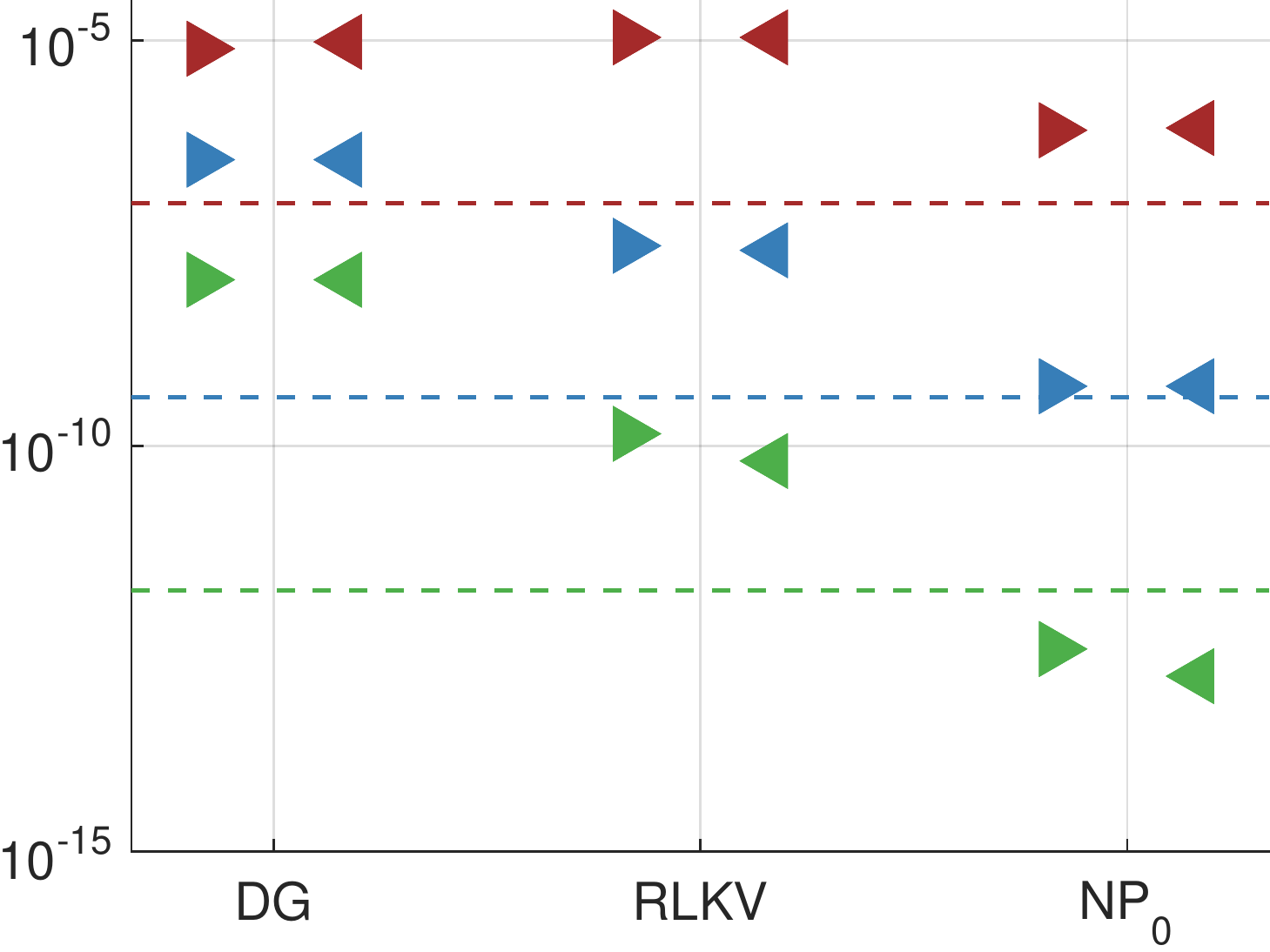}}  
	\unitlength = \textwidth 
	\divide \unitlength by 40
	\cchg{ 
	\begin{picture}(40,0)(0,0)	
	\putfnsize{4}{32.5}{RLKV}
	\putfnsize{16}{32.5}{RLKV}
	\putfnsize{10}{31.7}{symmetric filter}
	\puttext{4}{22.5}{$\refx$}
	\putvector{4.5}{22.5}{1}{0}{1.4}
	\putvector{3.5}{22.5}{-1}{0}{1.5}
	\putfnsize{26}{30.5}{\newft}
	\putfnsize{37.5}{30.5}{\newft}	
	\putfnsize{31.5}{31.5}{symmetric filter}
	\puttext{25.9}{22.5}{$\refx$}
	\putvector{26.4}{22.5}{1}{0}{1.4}
	\putvector{25.4}{22.5}{-1}{0}{1.5}	
	\end{picture}
}
\caption{
       Pointwise errors of the canonical partial differential equation 
       $u_\tau+u_x=0$ \eqref{eq:hypEqs} at time $\tau=1$
       based on $\dDG=3$ \DG\ output. Each subfigure shows top-most
       the mesh-size $\color{red}{h=20^{-1}}$  graph
       (red), $\color{blue}{h = 40^{-1}}$ (blue) and 
       $\color{OliveGreen}{h = 80^{-1}}$ (green) at the bottom.
       \cchg{ (a,b) The interior, bordered by vertical dotted lines of the 
         mesh-size color, is post-processed by the symmetric SIAC filter.
         The graph for the interior is therefore identical in (a) and (b).
         To the left of the left line is the \emph{left \bdy} $[0..\refx]$
         and to the right of right line is the right \bdy.
       }
       (a) Errors after convolving \cchg{the data of the \bdy s}
       with the one-sided \rlkv\ filters \cite{siac2014}
       and the interior with the symmetric SIAC filter.
       This graph is identical to \cite[Figure 5,top-right]{siac2014} 
       (but is computed with the more new stabler symbolic formulation).
       (b) Errors after convolving \cchg{the data of the \bdy s}
       with the new least-degree \bsiac\ filters \newft.
       (c) Errors of the degree $\dDG=3$ \DG\ output.
       (d) 
       The error in the left \bdy\ of (a) and (b) 
       is indicated by $\rhd$ and in the right by $\lhd$.
       The error of the symmetric SIAC filter applied in the interior
       is indicated by $--$.
       The error of the new \bsiac\ filter \newft\ is lower or on par
       with the optimally superconvergent symmetric SIAC filter in the interior.
}       
\label{fig:compare-d2} 
\end{figure}   

Most recently these boundary filters have been simplified
and improved by replacing numerical approximation 
with symbolic formulas, both in the uniform
symmetric case \cite{siac15unifiedView} and in the general case
\cite{JPetersFcoefs}. For general knot sequences, \cite{psiac} introduced 
a factored symbolic characterization of spline filters that facilitates
their knots being shifted or scaled. This allowed
characterizing the existing boundary filters as 
position-dependent SIAC spline filters (\bsiac\ filters).
\bsiac\ coefficients are polynomial expressions in the position 
and the coefficients of these polynomial expressions are rational numbers
for rational knot sequences. 
\cchg{
In the \emph{\bdy}, where \bsiac\ filtering is deployed, \bsiac\ filtering
converts the \DG\ output to a single polynomial \cite[Theorem 4.2]{psiac}.}
The \bsiac\  filters
can be symbolically precomputed for prototype filters and
these prototype filters are easily scaled and shifted for a specific data set.
The \bsiac\ characterization therefore replaces
Gauss quadrature that is otherwise required to
repeatedly derive, at each point near the boundary, 
a position-dependent filter to apply the filter to the \DG\ output.
\bsiac\ filtering then reduces to a sequence of single dot products between
the filter vector and the short vectors of local \DG\ output.
Finally, instead of approximating derivatives of the filtered DG output 
\cite{thomee77der,siac05der,siac09der,siac16der},
derivatives of the \bsiac-filtered output have an explicit expression.

This paper introduces a new piecewise constant \bsiac\ filter,
that we will refer to as the \newft\ filter.
The simple \newft\ \filter\ 
outperforms both the superconvergent, but large-support and
numerically unstable \srv\ boundary filter 
\cite{siac2011} 
and the small-support and stable but suboptimal \rlkv\ filter
\cite{siac2014}.
\rtext{Here unstable means that \srv\ requires, for reliable results
for cubic or higher \DG\ order, quadruple precision.
\cite{siac2014}).}
To illustrate this point, \figref{fig:compare-d2} shows
results for the canonical hyperbolic equation \rtext{$u_\tau+u_x=0$}
(c.f.~\equaref{eq:hypEqs}).
For three different mesh spacings the \rlkv\ error near the boundaries
exceeds
the error of the symmetric SIAC filter that applies (only) in the interior.
%
For large mesh spacing, the \rlkv\ error in \figref{fig:compare-d2}a 
even exceeds that of the unfiltered \DG\ error in \figref{fig:compare-d2}c.
By contrast,  \figref{fig:compare-d2}b shows that the new \bsiac\ filter
\newft\ to always reduce the \DG\ error on the boundary,
even below the error of the symmetric \siac\ filter.
This message is condensed in \figref{fig:compare-d2}d.
Here the maximal error of the symmetric SIAC filter of degree $\dDG$
in the interior is displayed as dashed lines and the error near the 
left and right boundary by $\rhd$, respectively $\lhd$. 
\rtext{In fact, with $\ddg$ the polynomial degree of the \DG\ output
and $\dft$ the degree of the spline filters, 
the theory of \cite{siac14estimate} 
only guarantees a convergence order of $\ddg+1 + \dft$.
Yet, extensive numerical experiments, presented in detail in the Appendix,
for the canonical hyperbolic equation
\rtext{$u_\tau+u_x=0$} \ccchg{for increasing final times $T$
of the \DG\ computation} show the \newft\ filter with $\dft=0$ yielding 
optimal convergence of order $2\ddg+1$.
(The present paper does not aim to provide a formal investigation of
optimal superconvergence of the \newft\ filter.)}

Not only does the new \newft\ filter reduce the error, but,
being piecewise constant, computing the convolution with the data is 
simple and stable.
Leveraging the symbolic formula provided by \cite{psiac},
filtering the \DG\ output at an endpoint, say $x=0$,
amounts to the scalar product of the local \DG\ coefficient vector 
with a vector of the form (\equaref{eq:symFilteredDG} of
 \thmref{thm:symFilteredDG} in this paper):
\begin{equation}
\vecSymb = 
   Q_{\refx} \brefx
   \in \mathbb{Q}^{(3\ddg+1)(\ddg+1)},\quad 
   \brefx := 
   \Big[ \refx^{0} \cdots \refx^{\nc} \Big]^\tr
   \label{eq:endVecV}
 \end{equation}
\cchg{
where $[0..\refx]$ is the left \bdy}
and $Q_{\refx}$ is a matrix with rational entries.
$Q_{\refx}$ depends only on the space of polynomials
used for the \DG\ approximation 
and on the space of filter kernels, see \equaref{eq:symFilteredDG}.
In practice,  the local \DG\ coefficient vector 
is multiplied with the matrix $Q_{\refx}$ in advance
yielding a vector of size  $3\ddg+1$ 
to be multiplied with $\brefx$.

To demonstrate the simplicity of the new \newft\ filter, we contrast the
\rtext{entries with the largest absolute values}, the
central four entries of $\vecSymb$, 
both for the \newft\ filter and the \srv\ filter
(rederived in its more stable \bsiac\ form) {when $\ddg=3$}:
\begin{align}
 \vecSymb_4(\text{\newft}) &= 
 \begin{matrix}
 {10080}^{-1}
 \end{matrix}
 \left[\begin{matrix}
70381 & 70381 & -56627 & -56627
 \end{matrix}\right]
 \label{eq:vecSymb}
  \\
 \vecSymb_4(\text{\srv}) &= 
 {15256200960000}^{-1}
 \big[\,3549982809648204 \quad 9809076669570393 
 \notag 
 \\
 &\hskip2cm
 11473452075703833 \quad 6592494198365004 \big].
 \notag
\end{align}
While the entries of $\vecSymb$ for the new filters 
are fractions of integers with 5 digits at most, those alternating 
numbers for the \srv\ filters 
have up to 17 digits. Crucially, as shown in \figref{fig:vecSymb},
the fractions for the new \newft\ filter are also much smaller.
For example, \figref{fig:vecSymb} shows that for $\dDG = 3$
the alternating coefficients of the \srv\ filter are two orders of 
magnitude larger than those of the \newft\ filter.
\def\wtwo{.45\textwidth}
\begin{figure}[ht!]
\centering 
\begin{tabular}{cc}
   \subfigure[Coefficients of new \newft\ filter]{\includegraphics[width=\wtwo]{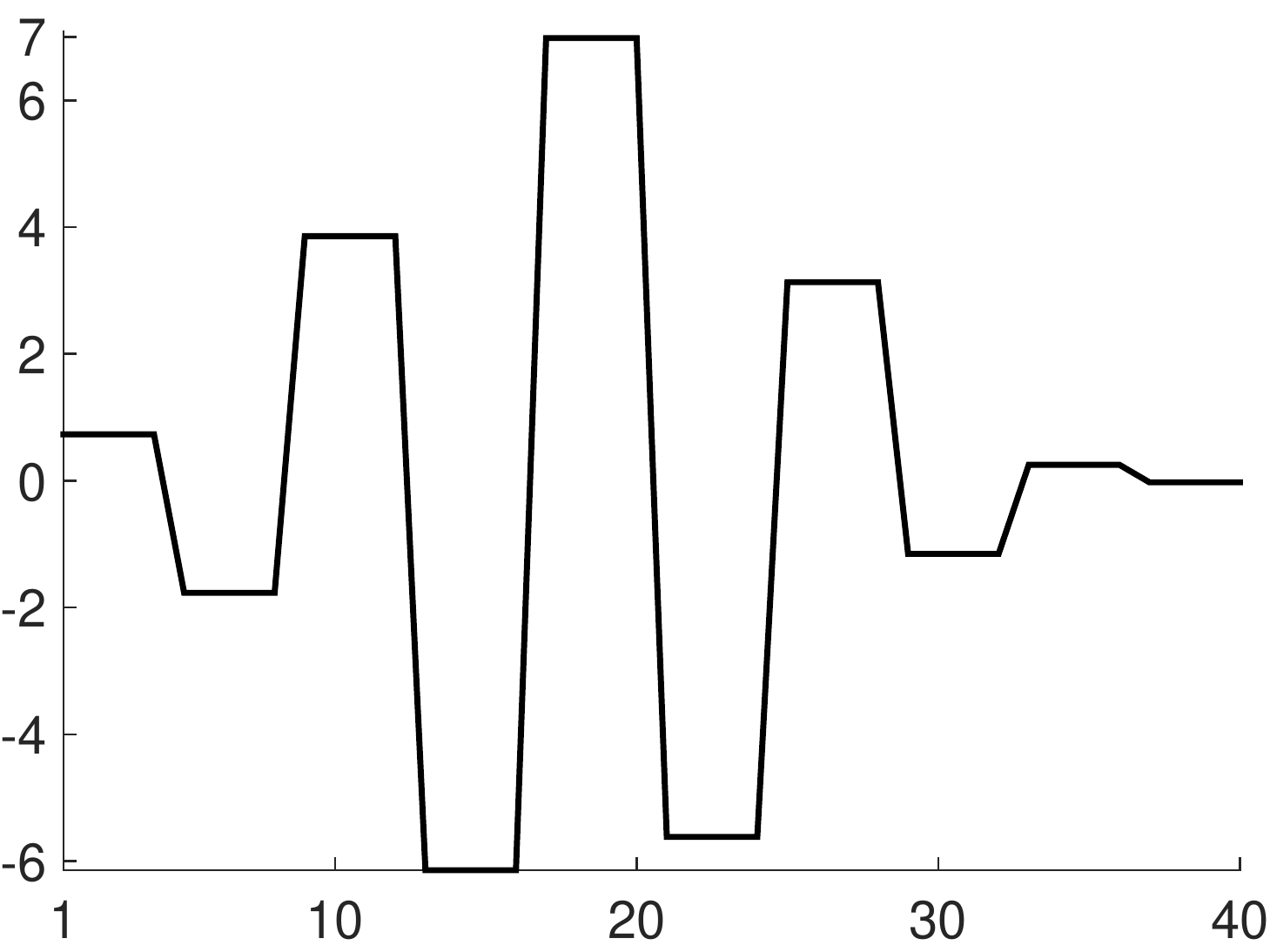}}
   \subfigure[Coefficients of \srv\ filter vs.~of \newft\ filter]{\includegraphics[width=\wtwo]{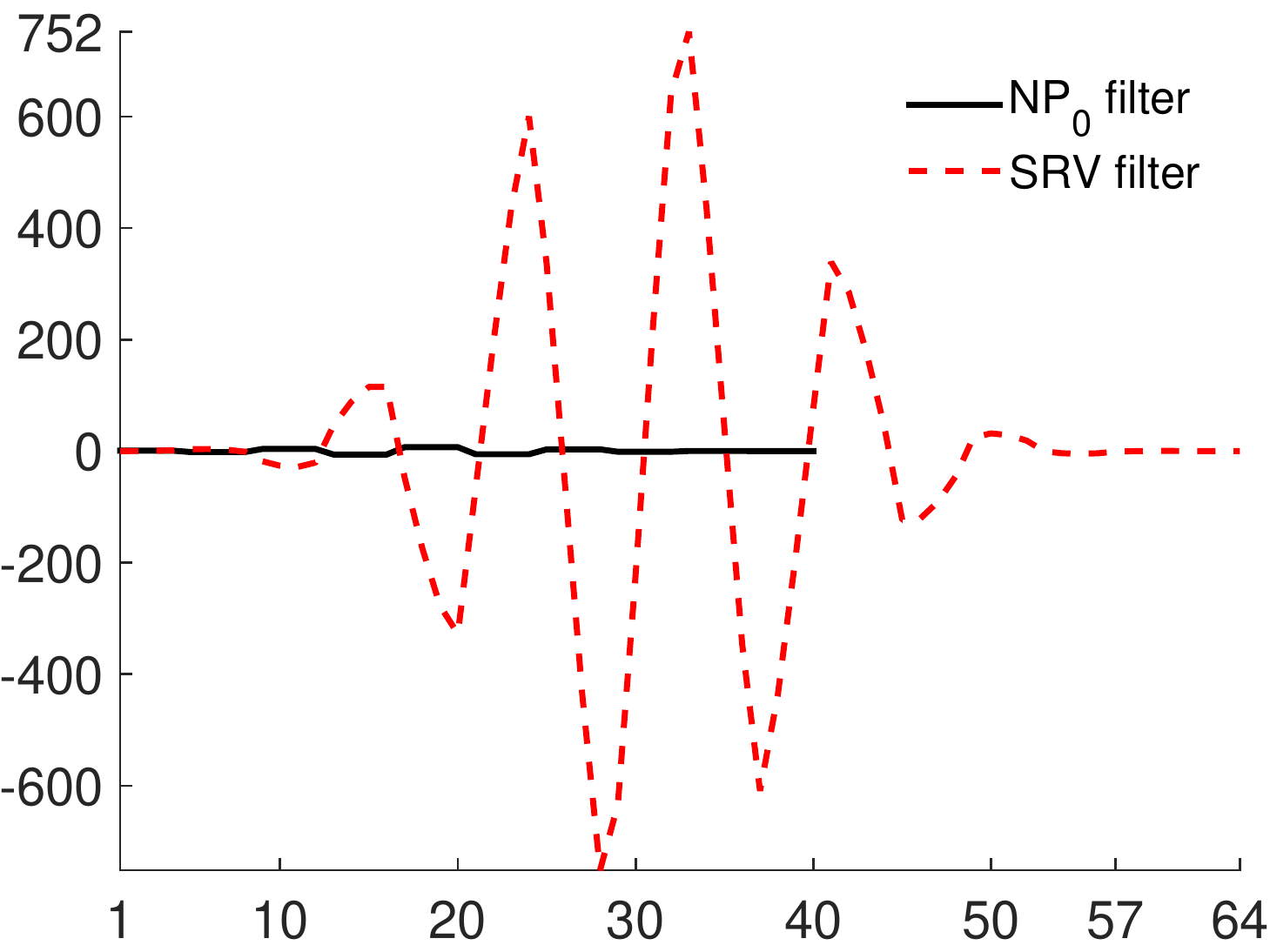}}
\end{tabular}
\caption{Size of the entries of 
$
\vecSymb = 
   Q_{\refx} \brefx
$
for \ccchg{filtering $\dDG=3$} \DG\ output at a boundary point $x=0$.
Note the 100-fold difference in scale between (a) and (b).}
\label{fig:vecSymb}
\end{figure}

In summary, 
\begin{itemize}
\item[$\rhd$]
the \bsiac\ boundary filter \newft\ is simple and explicit; and
\item[$\rhd$]
numerical experiments in double floating point precision
on the canonical wave equation
\rtext{
show \newft\ to have a convergence rate of $2\ddg+1$
where $\ddg$ is the piecewise degree of the \DG\ output.}
\end{itemize}

 
\paragraph{Organization}  
\secref{sec:notation} introduces notation, spline convolution
and the canonical test equations.
followed by a review of the literature in more detail.
\secref{sec:least} focuses on  the new \newft\ filter and compares it
to the existing \srv\ and \rlkv\ filters in their improved symbolic 
\bsiac\ form.
\secref{sec:errnumer} compares the \srv, \rlkv, \newft\ and the 
symmetric \siac\ filters numerically on three variants of the canonical
hyperbolic equation.



\newcommand{\ddfun}{g}
\newcommand{\kp}{\kappa}
\section{Notation and Definitions} 
\label{sec:notation}
This section establishes the notation for filters and \DG\ output,
exhibits the canonical test problem, the \DG\ method, 
and reproducing filters and reviews one-sided and position-dependent
\siac\ filters in the literature.

We denote by $f\convol g$ the convolution of a function $f$ with
a function $g$, i.e.\ 
\begin{align*}
   (f\convol g)(x) := \int_\R f(t) \, g (x-t) \, \drm t
   = (g\convol f)(x),
\end{align*}
for every $x$ where the integral exists.
Filtering means convolving a function $f$ with a kernel $g$.

\subsection{Sequences, Splines, and Reproduction}
The goal of \emph{\siac\ filtering} is to spatially smooth out
the \DG\ output $\dgout(\xx,\tm)$ 
by convolution in $\xx$ with a linear combination of B-splines. 
Typically filtering is applied after the last time step when $\tm = \tmend$.
Specifically, we will focus on piecewise polynomial 
\siac\ spline \kernel s $\sk: \R\to\R$
such that convolution of $\sk$ with monomials $(\cdot)^\dg$
reproduces the monomials up to degree $\nc$.
Let $\idxft := (\idxfst,\ldots,\idxlst)$ be a sequence of 
strictly increasing integers between $0$ and $\idxlst$,
abbreviate the sequences of consecutive integers as
\begin{align*}
i:j  \,\, &:=\, \, 
\begin{cases}
   (i,i+1,\ldots,j-1,j), \text{if } i\le j, \\
   (i,i-1,\ldots,j+1,j), \text{if } i > j,
\end{cases}
\quad
   s_{i:j} := (s_i,\ldots,s_j).
\end{align*}
Let $B(x|\tsft)$ denote the unit integral B-spline with (non-decreasing) 
knot sequence $\tsft$ (see \cite{Boor:2002:BB}) related
to the recursively defined B-spline $\Nbs{t}{t_{i:i+\df+1}}$ 
by 
$
  \Nbs{t}{t_{i:i+\df+1}} = \frac{t_{i+\df+1}-t_i}{\df+1}\, 
  \Bsp{t}{t_{i:i+\df+1}}.
$
Then a \emph{\siac\ spline \filter\ of degree $\df$ and \orderft{ }$\nc$ 
with index sequence $\idxft$ and knot sequence $\kft_{0:\nkft}$}
is a spline
\begin{equation*}
   \sk(x) := \sum_{\sft \in \idxft}  \cfno{\sft} B(x|\tsft),
\end{equation*}
of degree $\df$ with coefficients $\cfno{\sft}$ chosen so that
\begin{align}
   &\Bigl( \sum_{\sft\in\idxft}  \cfno{\sft} B(\cdot|\tsft)
   \convol (-\cdot)^\dg \Bigr) (x)
   =
   (-x)^\dg, \qquad \dg = 0:\nc.
   \label{eq:convol}
\end{align}
We reserve the following symbols:
\begin{align*}
\begin{tabular}{l l}
$\dDG$ & degree of the \DG\ output; \\
$\nkdg$ & number of intervals of the \DG\ output; 
\\
$s_{0:\nkdg}$ & \emph{prototype} increasing break point sequence, 
typically integers;
\\
& the break sequence of the \DG\ output is $\hh  s_{0:\nkdg}$;\\
$\df$ & degree of the filter kernel;  \\
$\nc+1$ & number of filter coefficients  \\
& for reproduction of polynomials up to degree $\nc$;
\\
$\idxft := (\idxfst,\ldots,\idxlst)$ & index sequence;\\
&
if the B-splines of the filter are consecutive, then $\idxlst = \nc$;
\\
$\nkft$& number of knot intervals spanned by the filter;\\
&
$\nkft = \idxlst+\df+1$; 
\\
$\bkft := \kft_{0:\nkft}$ & prototype (integer) knot 
sequence of the filter;\\
        & the input knot sequence of the filter is $\hh \kft_{0:\nkft} + \xi$
        \\
        &where $\xi$ is the shift and $\hh$ scales.
\end{tabular}
\end{align*}
This notation is illustrated by the following example.
\begin{example} \label{ex:symbols}
A linear \DG\ output sequence on 200 uniform segments
of the interval $[-1..1]$ implies $d=1$,
$\nkdg=200$, $h = \frac{1}{100}$ and $s_{0:\nkdg} = -100:100$.
A degree-one spline filter defined over 
the knot sequence $\bkft := 0:6$ and associated with the index set 
$\idxft := \{0,3,4\}$ corresponds to $k=1$, $\nkft=6$, $r=2$ and $\idxlst=4$.
The two B-splines defined over the knot sequences $1:3$ and $2:4$ are 
skipped.
\end{example}

\subsection{The canonical test problem and the Discontinuous Galerkin method}
To demonstrate the performance of the filters on a concrete example,
\cite{siac2003} used the following univariate hyperbolic 
partial differential wave equation:
 \begin{align} 
  \frac{du}{d\tm} +
  \frac{d}{dx}\Big( \kappa(x,\tm) \, u\Big) &= \rho(x,\tm), &x \in (a..b), 
  \tm \in (0..\tmend)
  \label{eq:hypEqs}
  \\
  u(x,0) &= u_0(x), &x\in
[a..b] \notag
 \end{align} 
subject to periodic boundary conditions, $u(a,\tm) = u(b,\tm)$, or Dirichlet
boundary conditions $u(e,\tm) = u_0(\tm)$ where, depending on the sign of
$\kappa(x,\tm)$, $e$ is either $a$ or $b$.
Subsequent work 
\cite{siac2003,siac2011,siac2014}
adopted the same differential equation to test their new one-sided filters 
and to compare to the earlier work. Eq.~\eqref{eq:hypEqs} is therefore
considered the \emph{canonical test problem}. 
We note, however, that \siac\ filters apply more widely, for example
to FEM and elliptic equations \cite{Bramble77}.
\par 
In the \DG\ method,
the domain $[a..b]$ is partitioned into intervals by a sequence
$\hh s_{0:\nkdg}$ of break points $a=:\hh s_0,\ldots,\hh s_\nkdg:=b$.
Assuming that the sequence is rational, scaling by $\hh$ will
later allow us to consider a \proto\ sequence $s_{0:\nkdg}$ of integers.
{
Let $\vh$ be the linear space of all piecewise polynomials 
with break points $\hh s_{0:\nkdg}$ and of degree 
less than or equal to $\ddg$. We use modal or nodal}
scalar-valued basis functions $\Nbf{i}{.}{\hh s_{0:\nkdg}}$ 
$0 \leq i \leq m$ of $\vh$ that  are linearly independent 
and satisfy the scaling relations
\begin{equation} \label{eq:DGbasisfuns}
 \Nbf{i}{hx}{hs_{0:\nkdg}} = \Nbf{i}{x}{s_{0:\nkdg}}.
\end{equation}
Relation \eqref{eq:DGbasisfuns} is typically used 
for refinement in FEM, DG or Iso-parametric PDE solvers.
Examples of  basis functions $\Nbfi{i}$ are
\bbname\ basis functions \cite{deboor}, 
Lagrange polynomials dependent on Legendre-Gauss-Lobatto 
quadrature points \cite{DGbook},
and Legendre polynomials. 

The \DG\ method approximates the time-dependent solution
of Eq.~\eqref{eq:hypEqs} by
\begin{equation} \label{eq:DGoutput} 
   \dgout(x,\tm)  
   :=
   \sum_{i=0}^m \, \dgout_i(\tm) \, \Nbf{i}{x}{\hh s_{0:\nkdg}},
   \quad \phi_i \in \vh.
\end{equation}

Multiplying the two sides of Eq.~\eqref{eq:DGoutput}
with a test function $v$ and integrating by parts
yields the weak form of Eq.~\eqref{eq:hypEqs}:
\begin{equation}
 \int_a^b  \big( \frac{du}{d\tm} \, v - \kappa(x,\tm) u \, \frac{dv}{dx} \big)
 \, \drm x = 
 \int_a^b \rho(x,\tm)\,v\,\drm x -
 \Big(\kappa(x,\tm)u(x,\tm)v(x)\Big)\Big|_{x=a}^{x=b}.
 \label{eq:weakForm}
\end{equation} 
{
Substituting $\dgout$ on the left of Eq.~\eqref{eq:weakForm}
by \eqref{eq:DGoutput},
treating the rightmost, non-integral term
of Eq.~\eqref{eq:weakForm} as a numerical flux,  and choosing
$v(x) := \Nbf{j}{x}{\hh s_{0:\nkdg}}$, 
}
yields a system of ordinary differential equations {in $\tm$} with 
the coefficients $u_i(\tm)$, $0\leq i \leq m$, as unknowns.
{This system can be solved by, e.g., a standard fourth-order
four stage explicit Runge-Kutta method (ERK) \cite[Section 3.4]{DGbook}}.

\subsection{A synopsis of \DG\ filtering}

Since convolution with a symmetric SIAC \kernel\ of a function $g$ at $x$ 
requires $g$ to be defined in a two-sided neighborhood of $x$, 
near boundaries, Ryan and Shu \cite{siac2003} proposed 
convolving the \DG\ output with a \kernel\
whose support is shifted to one side of the origin:
for $x$ near the left domain endpoint $\lft$,
the one-sided SIAC \kernel\ is defined over
$ (x-\lft) + h \big( -(3\ddg+1),-3\ddg,\ldots,0 \big)$
where $\ddg$ is the degree of the \DG\ output.
The Ryan-Shu $\xx$-position-dependent 
one-sided kernel yields optimal $L^2$-convergence,
but its point-wise error near $\lft$ can be larger than that of the
\DG\ output.
\par 
In \cite{siac2011}, {Slingerland-Ryan-Vuik}
improved the one-sided \filter\ by increasing its monomial reproduction
from degree $\drepro = 2\ddg$ to degree $\drepro = 4\ddg$.
This one-sided \kernel\ reduces the boundary error when $\dDG=1$
but the \kernel\ support is increased by $2\ddg$ additional knot intervals
and numerical roundoff requires quadruple precision calculations to determine
the \kernel's coefficients.
\rtext{(\cite{siac2011} additionally required quadruple precision for 
computing the DG output.)}
Indeed, the coefficients of the boundary filters 
\cite{siac2003,siac2011,siac2012,siac2014,siac15unifiedView} 
are computed by inverting a matrix whose entries are determined by
Gaussian quadrature; and, as pointed out in \cite{siac2014}, 
\srv\ filter matrices are close to singular. 
\par
Ryan-Li-Kirby-Vuik \cite{siac2014} therefore suggested 
an alternative one-sided position-dependent 
 \filter\ that has the same support size
as the symmetric kernel and has reproduction degree higher by one, 
enriching the spline space by one B-spline. This \rlkv\ \filter\ is
stably computed, as has been verified numerically, in double precision,
up to input data degree $\dDG=4$ and
joins the symmetric SIAC filter, applied in the interior, without a jump
in error.
However, the error of the \rlkv\ \filter\ at the boundaries can be higher than
that of the symmetric \filter\ 
and the $L^2$ and $L^\infty$ superconvergence rates are sub-optimal \cite{siac2014}
\rtext{(c.f.\ \figref{fig:pbc},\ref{fig:dbc},\ref{fig:xt-speed})}.
\cite{siac16der} additionally states that \rlkv\ has a poorer
 derivative approximation than \srv\ filters.

\cite{psiac} reinterprets the published one-sided
filters in an explicit, symbolic form as position-dependent
PSIAC spline filters.
Symbolic expression of coefficients for spline filters have recently been
developed in \cite{siac15unifiedView} for uniform knot sequences 
and in \cite{JPetersFcoefs} for general knot sequences.
Reinterpretation of the published filters in symbolic form 
improves their numerical stability.

\subsection{Symbolic formulation of the filters}
\label{subsec:defM0}
We split the \DG\ data at any known discontinuities
and treat the domains separately.
Then convolution can be applied throughout a given closed
interval $[\lft..\rgt]$. 
A SIAC spline \kernel\ with knot sequence $t_{0:\nc+\df+1}$ 
is \emph{symmetric} (about the origin in $\R$)  if
$t_\ell + t_{\nc+\df+1-\ell}=0$ for $\ell=0: \lceil (\nc+\df+1)/2 \rceil$.
Note that, unlike the (position-independent)
classical symmetric SIAC filter, the position-dependent boundary
kernel coefficients have to be determined afresh for each point $\xx$.



\begin{lemma}  [\siac\ coefficients \cite{psiac}]
The vector $\mc := [\cfno{0},\ldots,\cfno{\nc}]^\tr \in \R^{\nc+1}$ 
of B-spline coefficients of the \siac\ filter with
index sequence $\idxft := (\idxfst,\ldots,\idxlst)$  and
knot sequence $\kft_{0:\nkft}$ is 
\begin{align}
    \mc := \text{first column of } M^{-1},
    &
    \quad  
   M := M_{\kft_{0:\nkft},\idxft}
   = \left[ \begin{matrix}
    \sum\limits_{|\midx|=\dg} 
    \kft_{\sft:\sft+\dft+1}^{\midx}   
   \end{matrix} \right]_{\dg=0:\nc,\,\sft\in\idxft}
   \label{eq:M0powerForm}         
   \\ \text{where} \quad
   \kft^\midx_{0:p} 
   &:=
   \kft_0^{\ix_0} \ldots \kft_{p}^{\ix_p}
   \quad \text{and}\quad
   |\kft_{0:p}| := \sum^{p}_{j=0} |\kft_j|
   \notag 
\end{align}
\label{thm:siac}
\end{lemma}
This characterization yields the following formula for 
the filter coefficients.

\begin{theorem} [Scaled and shifted \siac\ coefficients are
polynomial \cite{psiac}]
\label{thm:sblFtCoef}
The \siac\ filter coefficients $\cfno{\xi;\ell}$
associated with the knot sequence
$
   h\kft_{0:\nkft}+ \xi 
$
are polynomials of degree $\nc$ in $\xi$:
\begin{equation} 
   \mc_\xi    
   := [\cfno{\xi;\ell}]_{\ell=0:\nc}    
   = \ M_{\kft_{0:\nkft},\idxft}^{-1} \,
   \diag\big(\left[\begin{matrix}(-1)^\ell \, \binom{\ell+\df+1}{\ell}
   \end{matrix}\right]_{\ell=0:\nc}\big)
   \Big[ \big(\frac{\xi}{h} \big)^{0:\nc} \Big]^{\tr}.
\label{eq:cxih}
\end{equation}
\label{thm:oneSidedCoefs}
\end{theorem}

The following corollary
implies that the \kernel\ coefficients 
$\cfno{\xi;\ell}$, can be pre-computed stably, as scaled integers.
\begin{corollary} [%
Coefficient polynomials $\cfno{\xi,\ell}$ have rational coefficients \cite{psiac}]
If the knots $\kft_{0:\nkft}$ are rational, then the 
filter coefficients $\cfno{\xi,\ell}$ 
are polynomials in $\xi$ and $h$ with rational coefficients.
\label{lem:rational}
\end{corollary} 

We can now define the \bsiac\ \kernel.

\begin{definition} [\bsiac\ kernel] \label{def:filter}
A \bsiac\ \kernel\
with index sequence $\idxft = (\idxfst,\ldots,\idxlst)$
and knot sequence $h\kft_{0:\nkft}+x$ has the form 
\begin{equation} \label{eq:filterOverSK}
   \flt{x}(s)
      := \sum_{j\in\idxft} \, \cf{\xx}{j} \Bsp{s}{\hh t_{j:j+\df+1}+\xx},    
      \quad 
      s \in \hh[\kft_0,\kft_\nkft] + \xx.
\end{equation}
\end{definition}
The \DG\ output is convolved with a \bsiac\ \filter\ $\flt{x-\hh \refx}(s)$ 
of \orderft{ }$\nc$, associated with an index sequence $\idxft$
and defined over shifted  knots $\hh \kft_{0:\nkft} + x -\hh \refx$ -- 
where the constant $\hh \refx$ adjusts the filter kernel to the left or right boundary.
\rrtext{
\exref{ex:FtCoef} illustrates the simple explicit form 
of the \bsiac\ coefficients according to \thmref{thm:sblFtCoef}
and verifies that the corresponding \bsiac\ filter reproduces as predicted 
by the derivation.
}
	
\begin{example} [Reproduction by \bsiac\ filtering]
\label{ex:FtCoef}
Let $h=1$, $\dft=0$ and $\ftex{x}$ be the least degree \bsiac\ filter 
with $\kft_{0:\nkft} = \{-2,-1,0\}$, $\nc=1$, $\idxft = (0,1)$.
According to \eqref{eq:filterOverSK} of \defref{def:filter}:
\begin{align}
   \ftex{x}(s) :&= \ftex{x;0} \Bsp{s}{\{-2,-1\}+x} 
                 + \ftex{x;1} \Bsp{s}{\{-1,0\}+x} 
   \label{eq:FtTest}
   \\
   &= \ftex{x;0} \, \idf{[-2,-1]+x} + \ftex{x;1} \, \idf{[-1,0]+x},
   \qquad
   \idfunc{\alpha}{\beta}(s) := \begin{cases}
   1, & \text{ if } s\in [\alpha,\beta];\\
   0, & \text{ else}
   \end{cases}
   \notag 
\end{align}
where $\idfunc{\alpha}{\beta}$ denotes the indicator function
of the domain $[\alpha,\beta]$.
Equation \eqref{eq:cxih} of \thmref{thm:sblFtCoef} provides the formula
\begin{align}
\left[ \begin{matrix}
\ftex{x;0}\\
\ftex{x;1}
\end{matrix} \right]
:=
\left[ \begin{matrix}
1 & 1 \\
-2-1 & -1+0
\end{matrix} \right]^{-1}
\,
\left[ \begin{matrix}
 1 & 0 \\
 0 & -2
\end{matrix} \right]
\,
\left[ \begin{matrix}
1\\
x
\end{matrix} \right]
=
\dfrac{1}{2}\left[ \begin{matrix}
2x-1\\
3-2x
\end{matrix} \right].
\label{eq:LfilterCoef}
\end{align}
This choice of filter coefficients $\ftex{x;0}$ and $\ftex{x;1}$ 
satisfies \equaref{eq:convol}.
{\em For $\dg=0$:} 
\[ (\ftex{x} \convol 1)(x) =
\int_{\RR} \ftex{x}(s) \drm s = \ftex{x;0} \int_{\RR} \idf{[-2,-1]+x}(s) \,\drm s + 
\ftex{x;1} \int_{\RR} \idf{[-1,0]+x}(s) \, \drm s 
= 1.
\]   
\begin{align}
\text{{\em For $\dg=1$:}} \quad (\ftex{x} &\convol (-\cdot) )(x) 
= \int_{\RR} \ftex{x}(s) \, (s-x) \, \drm s
\notag 
 \\
 &= \ftex{x;0} \int_{\RR} \idf{[-2,-1]+x}(s) \, (s-x) \,\drm s + 
\ftex{x;1} \int_{\RR} \idf{[-1,0]+x}(s) \, (s-x) \,\drm s
\notag 
\\
&= \ftex{x;0} \int_{-2+x}^{-1+x} (s-x)\drm s + 
\ftex{x;1} \int_{-1+x}^{0+x} (s-x)\drm s
\notag 
\\
&= \ftex{x;0} \left( \frac{s^2}{2}\big|_{s=-2+x}^{s=-1+x} -x \right)
\notag 
+ 
\ftex{x;1} \left( \frac{s^2}{2}\big|_{s=-1+x}^{s=0+x} -x \right) 
\notag 
\\
&=  \ftex{x;0} ( -\ftex{x;1}-x ) + \ftex{x;1} (\ftex{x;0} - x ) = -x.
\notag 
\end{align}
\end{example}

\noindent
Leveraging \thmref{thm:oneSidedCoefs},
we can efficiently compute the convolution as follows.

\begin{theorem} [Efficient \bsiac\ filtering of \DG\ output \cite{psiac}]
Let $\flt{x}(s)$ be a \bsiac\ \filter\ of \orderft{ } $\nc$
with index sequence $\idxft = (\idxfst,\ldots,\idxlst)$
and knot sequence $h\kft_{0:\nkft}+x-h\refx$.
Let
$
   \dgout(x,\tm) 
   :=
   \sum_{i=0}^m \, \dgout_i(\tm) \, \Nbf{i}{x}{\hh s_{0:\nkdg}},
$
$x \in [a,b]$ and $\tm \geq 0$,
be the \DG\ output.
Let $\Ical$ be the set of indices of basis functions
$\Nbf{i}{.}{\hh \gs_{0:\nkdg}}$
with support overlapping
$\hh[\refx-\kft_{\nkft},\refx-\kft_0]$.
Then the \emph{filtered \DG\ approximation}
is a polynomial in $x$ of degree $\nc$:
\begin{align} 
   \label{eq:symFilteredDG}
   \big(u*\flt{x}\big)(x) &= \ub_\Ical 
   \, 
   Q_{\refx}
   \,
   \Big[ \big(\frac{x}{h} - \refx\big)^{0:\nc} \Big]^\tr.
   \\
   \ub_\Ical &:= [u_i(\tm)\big]_{i\in\Ical}, 
   \notag
   \\
   Q_{\refx} &:= 
   \ipmatl{}
   \,
        \adia
   \,
   M^{-1}_{0, \bkft,\idxft}
   \,
   \diag(\left[\begin{matrix}(-1)^\ell \, \binom{\ell+\df+1}{\ell} \end{matrix}
   \right]_{\ell=0:\nc}),
   \notag
   \\
   \ipmatl{} &:= \left[ \int_{\refx-\kft_{\lst}}^{\refx-\kft_0} \,
   \Nbf{i}{s}{\gs_{0:\nkdg}}  
   \, \Bsp{s}{\refx-\kft_{\nkft-j:\idxlst-j}} \,
   \drm s \right]_{i\in\Ical,\, j\in\idxft}.
   \label{eq:Tcinvariant}   
\end{align} 
$\adia$ is the reversal matrix with 1 on the antidiagonal and zero else.
\label{thm:symFilteredDG}
\end{theorem}

The factored representation implies that instead of recomputing the
filter coefficients afresh for each point $\xx$ of the convolved output as
was the practice prior to \cite{psiac}, 
we simply pre-compute the coefficients corresponding to 
one \proto\ knot sequence $\bkft$ and, at runtime, pre-multiply with the data
and post-multiply with the vector of shifted monomials 
scaled by $\hh$ according to Eq.\ \eqref{eq:symFilteredDG}.

Increased multiplicity of an inner knot of the symmetric, position-independent
\siac\ kernel reduces its
smoothness, and this, in turn, reduces the smoothness of the filtered output.
By contrast, \thmref{thm:symFilteredDG} shows that
when the \bsiac\ knots are shifted along evaluation points $\xx$ then
\bsiac\ convolution yields a \emph{polynomial}, i.e.\ 
\rtext{the representation
near the boundary is infinitely smooth regardless of the knot multiplicity.}
\emph{
That is, we may view position-dependent filtering as a form of polynomial
approximation.}
\cchg{For example, the \rlkv-filtered output is a single polynomial over the 
\bdy\ where it applies.} 

\newcommand{\lambdaL}[1]{\lambda_{L,#1}}
\newcommand{\lambdaR}[1]{\lambda_{R,#1}}
\newcommand{\acoef}[2]{a_{#1,#2}}
\newcommand{\smallequal}{=}
\begin{example}[Coefficients of the \rlkv-filtered \DG\ output polynomial] \label{ex:RLKVpolynomials} 
Let $\ddg=\dft = 3$. Consider the canonical partial differential equation 
$u_\tau+u_x=0$ of \eqref{eq:hypEqs} for $x\in[0,1]$ at \ftime\ $\tmend=1$
for mesh-sizes $h_i := 2^{-i}/10$, $i=1,2,3$.
The analytical solution of this equation is $u_{e}(x) = \sin 2\pi(x-\tmend)$.
Let $\lambdaL{i} := 5h_i$, $\lambdaR{i} := 1-5h_i$ and
$f_{L,i,x}$ be the \rlkv\ filters with respect to the left 
\bdy s.
Applying the \rlkv-filter to the \DG\ output $\dgout_{h_i}$ computed
for mesh size $h_i$ yields 
a polynomial in $x$ as predicted by \thmref{thm:symFilteredDG}:
\begin{align}
P_{L,i}(x) := (f_{L,i,x}\convol \dgout_{h_i})(x) &= \sum_{k=0}^7 \acoef{k}{i} (x-\lambdaL{i})^k,
\quad &x \in [0,\lambdaL{i}],
\end{align}
\label{eq:RLKVpolynomials}
where 
\thickmuskip=0mu
\begin{align*} 
 \left[ 
\begin{smallmatrix}
\acoef{0}{1} & \cdots & \acoef{3}{1}
\\
\acoef{4}{1} & \cdots & \acoef{7}{1}
\end{smallmatrix}
\right]
& \smallequal
 \left[ 
\begin{smallmatrix}
0.999999901374753 & 0.000021494468508 & -19.738996791744032 & -0.008053345774016
\\64.88630724285224 & 0.623484670536888 & -82.011046997856752 & -11.875898409510508
\end{smallmatrix}
\right] 
\\
 \left[ 
\begin{smallmatrix}
\acoef{0}{2} & \cdots & \acoef{3}{2}
\\
\acoef{4}{2} & \cdots & \acoef{7}{2}
\end{smallmatrix}
\right]
& \smallequal
\left[ 
\begin{smallmatrix}
0.707106780904271  & 4.442883051171333 & -13.95772600673645 & -29.233165736223864 
\\
 45.9166940244231 & 57.754323004960845 &-59.8089096658096 &-57.790782092166637
\end{smallmatrix}
\right] 
\\
 \left[ 
\begin{smallmatrix}
\acoef{0}{3} & \cdots & \acoef{3}{3}
\\
\acoef{4}{3} & \cdots & \acoef{7}{3}
\end{smallmatrix}
\right]
&\smallequal
\left[ 
\begin{smallmatrix}
0.382683432364482 &  5.804906304724222 & -7.553868156289703 &-38.194755175826380 
\\
 24.85114895244633 & 75.39658368210462 & -32.618972790201198 & -71.729391408995241
\end{smallmatrix}
\right]. 
\end{align*}
Analogously, by symmetry, the filtered data of the right \bdy\ is
\[
   P_{R,i}(x) := (f_{R,i,x}\convol \dgout_{h_i})(x) 
   = \sum_{k=0}^7 (-1)^{k+1} \,\acoef{k}{i} (x-\lambdaR{i})^k, 
   \quad x \in [\lambdaR{i},1].
\]
\figref{fig:RLKVpolynomials}a plots
the polynomials $P_{L,i}$ and $P_{R,i}$,
the diffference between the polynomials and the exact solution
(note the scale $10^{-5}$ in \figref{fig:RLKVpolynomials}b) and
the error in log scale  \figref{fig:RLKVpolynomials}c.
\figref{fig:RLKVpolynomials}c matches the
error graphs of \cite[Figure 5,top-right]{siac2014} that were 
pointwise computed numerically.
\def\wid{.3\linewidth}
\begin{figure}[h!t]
  \centering
  \subfigure[\rlkv-filtered polynomials]{
  \includegraphics[width=\wid]{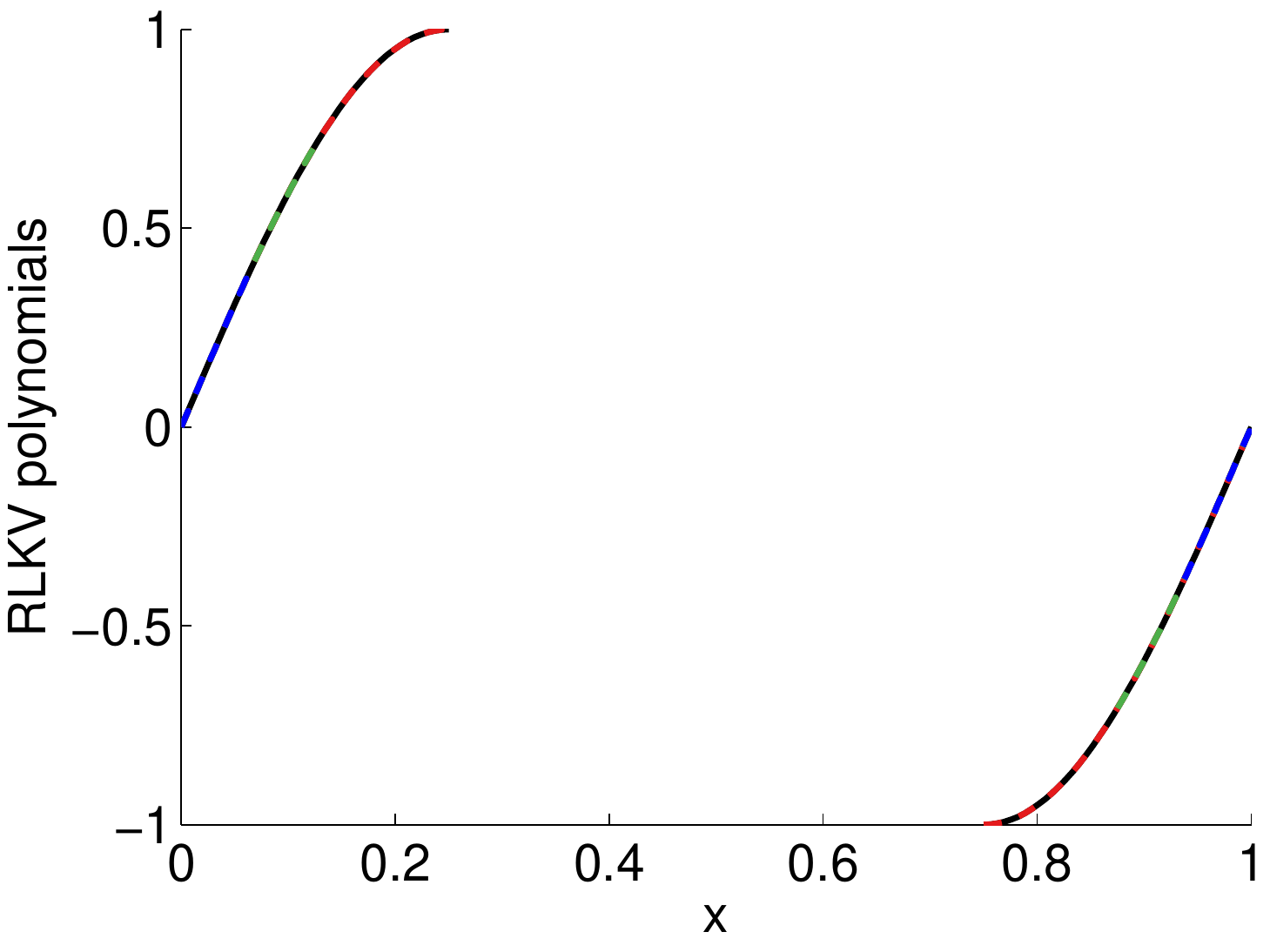}}
  \subfigure[filtered - exact solution]{
  \includegraphics[width=\wid]{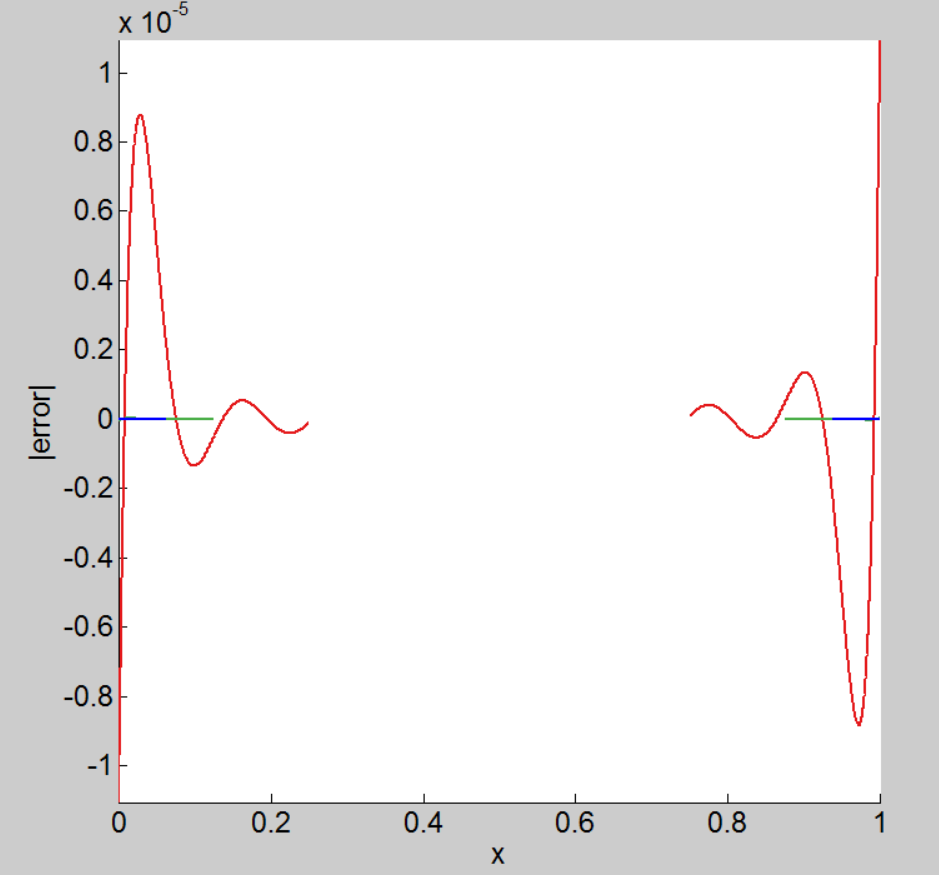}}
  \subfigure[Absolute error, log scale]{
  \includegraphics[width=\wid]{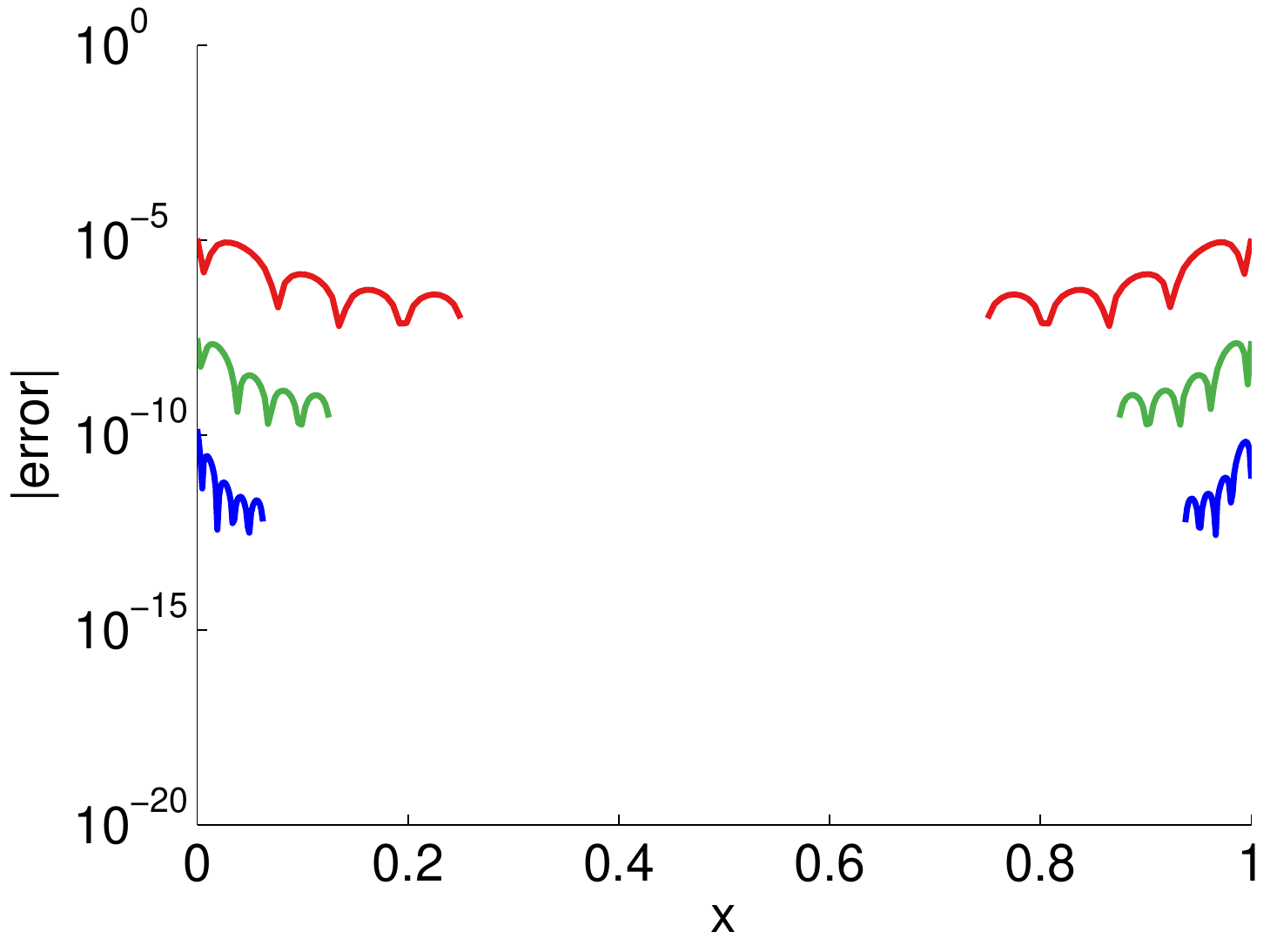}}
  \caption{ \rlkv-filtered \bdy\ of the \DG\ output for $\dDG=3$
  of \exref{ex:RLKVpolynomials}.
  The graphs in red, green, and blue correspond to $i=1,2,3$ respectively,
  i.e.\ to $N=20,40,80$ \DG\ segments.
  (a) The three polynomials of degree 7 of the \rlkv-filtered output
  with the exact solution (black dashed) superimposed.
  (b) The difference between the filtered output and the exact solution
  at very fine resolution.
  (c) Log scale of the absolute error \IL:
     $|P_{L,i}(x)-u_e(x)|$, $x \in [0,\lambdaL{i}]$ and \IR:
     $|P_{R,i}(x)-u_e(x)|$, $x \in [\lambdaR{i},1]$
     at the computational resolution $6N$
     (As (b) indicates, for higher resolution the arches would reach down 
     to zero.)
   }
   \label{fig:RLKVpolynomials}
\end{figure}
\end{example}	

The polynomial characterization directly provides a symbolic expression for
the derivatives of the convolved \DG\ output.
\begin{corollary}[Derivatives of \bsiac-filtered \DG\ output \cite{psiac}]
\begin{equation} \label{eq:derivative}
   \frac{\drm^\ell}{\drm x^\ell}\big(u*\flt{x}\big)(x) = \ub_\Ical 
   \,    
   Q_\refx
   \,
   \diag( h^{-(0:\nc)} )
   \,
   \big( 
   \frac{\drm^\ell}{\drm x^\ell}
   (x-\hh\refx)^{0:\nc} \big)^\tr.
\end{equation}
\end{corollary}

\newcommand{\nkSR}{\mu}
\newcommand{\knotsize}{\mu}
\newcommand{\lambSR}[1]{\refx_{#1,d}}
\newcommand{\kftv}[1]{{\boldsymbol \kft}_{#1,d}}

\subsection{Boundary filters as \bsiac\ filters}
The symmetric knot sequence of the symmetric kernel of degree $\dDG$ is
\begin{align} 
   \tb := h (-\nkSR, -\nkSR +1, \ldots, \nkSR ),\quad
   \nkSR := \frac{\nc+\ddg+1}{2},  \ \nc := 2\dDG.
   \label{eq:muDef}
\end{align} 
On $[a..b]$, this symmetric kernel can only be applied 
at evaluation points $x$ where 
\begin{align}
   \lambSR{L}:= a+\nkSR,
   \leq x \leq 
   b-\nkSR =: \lambSR{R}.
\end{align}

The boundary SIAC \filter s RS \cite{siac2003} 
and \srv\ \cite{siac2011} of \orderft\ $\nc + 1$ 
are of degree $\df=\ddg$, the degree of the \DG\ output.
Their index sequence $\idxft$ is consecutive, 
and they are defined over the shifted knots
\begin{align} 
 &\kftv{*}(\xi) := \Big( -\nkSR,-\nkSR+1,\ldots,\nkSR \Big) + \xi-\lambSR{*},
 \quad *\in \{L,R\}
 \label{eq:RSknots}
\end{align}
that form a symmetric support about the origin {when $\xi = \lambSR{*}$}.
The two \filter s differ in their degree:
$\nc(\text{RS})=2d$ and $\nc(\text{RV})=4d$. 
Explicit forms of the matrix
$\ipmat_{\refx}$, that is defined in \thmref{thm:symFilteredDG} to efficiently
construct the filter, are presented in \cite{psiac}.

The index sequence $\idxft$ of the boundary \filter\
\rlkv\ \cite{siac2014} is non-consecutive. 
The left and 
right \filter s are of degree $2\ddg + 1$ 
and are defined over the shifted knots, symmetric about the origin:
\begin{align} 
 \kftv{L}(\xi) &:= \Big( -\nkSR,\ldots,\nkSR-1,
 \underbrace{\nkSR,\ldots,\nkSR}_{\text{$\ddg+1$ times}} \Big) + \xi-\lambSR{L},
 \label{eq:RLKVknotsL}
 \\
 \idxft_L &:= \{1:(2\ddg+1),3\ddg+1\};
 \notag
 \\
\kftv{R}(\xi) &:= \Big(
 \underbrace{-\nkSR,\ldots,-\nkSR}_{\text{$\ddg+1$ times}},
  -\nkSR+1,\ldots,\nkSR,
 \Big) + \xi-\lambSR{R}, \label{eq:RLKVknotsR} 
  \\
 \idxft_R &:= \{1,\ddg:(3\ddg+1)\}.
 \notag
\end{align}
Explicit forms of the matrix
$\ipmat_{\refx}$, that is defined in \thmref{thm:symFilteredDG} to efficiently
construct the filter, are presented in \cite{psiac}.

Theorem~\ref{thm:symFilteredDG}  shows that a \bsiac\ filter need not have 
the same degree as the symmetric filter and it shows that \bsiac\ filters
may have multiple knots without reducing the continuity of the filtered 
\DG\ output. To illustrate this,
\cite{psiac} introduced filters with multiple interior knots.
This class of filters is denoted \newftk\ 
and has knot sequences
\begin{align} 
   \knotsftv_L &:= x-\refx_L + \big(-\knotsize,\ldots,\knotsize-3,
   \knotsize-2,\underbrace{\knotsize-1,{\cdots},\knotsize-1}_{\text{$\df+1$ {times}}},
   \underbrace{\knotsize,{\cdots},\knotsize}_{\text{$\df+1$ times}} \big), \
   \notag
   \\
   \knotsftv_R &:= x-\refx_R + \big(\underbrace{-\knotsize,{\cdots},-\knotsize}_{\text{$\df+1$ {times}}},
   \underbrace{-\knotsize+1,{\cdots},-\knotsize+1}_{\text{$\df+1$ {times}}},
   -\knotsize+2,-\knotsize+3\ldots,\knotsize\big).
   \label{eq:LRknots}
\end{align}

\section{New least-degree \DG\ filters}
\label{sec:least}
The symbolic formulation \eqref{eq:symFilteredDG} applies to kernels
of degree  different from the degree $\ddg$ of the \DG\ output.
We may therefore consider a piecewise-constant ($\df=0$) \bsiac\ filter.
Notably, the \newft\ filter has
a consecutive index sequence $\idxft$ and is defined over the shifted knots
\begin{align} 
 &\kftv{*}(\xi) := \Big( -\nkSR,-\nkSR+1,\ldots,\nkSR \Big) + \xi-\lambSR{*},
 \quad \nkSR = \frac{3\dDG+1}{2},
 \quad *\in \{L,R\}.
 \label{eq:newKnots} 
\end{align} 
The piecewise constant \newft\ filter 
has \orderft\ $\nc+1 = 3\dDG+1$  
and the same support size as the symmetric kernel.
\rtext{(Numerical experiments show that a filter with $2\ddg+1$
	constant pieces still achieves optimal superconvergence --
	albeit with a larger error than using $3\ddg+1$ pieces.
        $3\ddg+1$ is the number of pieces that the symmetric interior
        \siac\ filter uses).}

{
While the smoothness of the filtered output of 
position-\emph{independent} filters, e.g.\
symmetric \siac\ filters, depends on the filter degree,
position-\emph{dependent} \bsiac\ filters yield maximally smooth output regardless of their degree:
by \thmref{thm:symFilteredDG},
the \DG\ output filtered by a \bsiac\ filter
is a polynomial over the respective \bdy\
independent of the degree or smoothness of the \bsiac\ filter.
That is, even our piecewise constant \newft\  \bsiac\ filter 
increases the smoothness to infinity in the \bdy.
}
\rrtext{\exref{ex:poly} illustrates the remarkable fact that 
\bsiac\ filtering yields a single polynomial.
This is in contrast to the finite smoothness at break points of 
data filtered with position-independent SIAC filters.}

\newcommand{\utest}{\uex_{h_0}}
\newcommand{\domtest}{\Omega}
\begin{example} [\bsiac-filtering yields polynomial output] \label{ex:poly}
Let $h=1$, $\idfunc{\alpha}{\beta}$ be the indicator function of
$[\alpha,\beta]$ and 
\[
\utest(x) := \idfunc{0}{1}(x) + \idfunc{3}{4}(x), \, x\in [0,7] =: \domtest.
\] 
the discontinuous \DG\ output.
Convolving $\utest(x)$ at $x$ in the interior region $[2,5]$ of $\domtest$
with the symmetric (position-{\em independent}) \siac\
filter of reproduction degree $\nc=1$, 
$$\Kcal(s) := \frac{1}{2} \Bsp{s}{-1:0} + \frac{1}{2} \Bsp{s}{0:1} = 
\frac{1}{2}\idfunc{-1}{1}(s),$$
yields
\begin{align}
(\Kcal \convol \utest)(x) = 
\frac{1}{2} \Bsp{s}{2:4} + \frac{1}{2} \Bsp{s}{3:5}, \quad x \in [2,5].
\label{eq:ex2:convolvK}
\end{align}
That is, 
convolving with $\Kcal$ yields a $C^0$ output, 
as predicted by the \siac\ theory developed
by Ryan et al. \cite{siac2003,siac2011}.

By contrast, at $x$ in the left \bdy\ $[0,2]$ of $\domtest$,
convolving $\utest(x)$ with the left-sided least-degree 
position-dependent \bsiac\ filter $\ftex{x}$ defined in \exref{ex:FtCoef} 
yields
\begin{align}
\text{\em for $x\in [0,2]$:} \quad  &(\ftex{x} \convol \utest)(x) = \int_{\RR} \ftex{x}(s) \utest(x-s) \drm s
\notag 
\\
& \eqtext{by \equaref{eq:FtTest}}
\ftex{x;0} \int_{-2+x}^{-1+x} \utest(x-s) \drm s
+ \ftex{x;1} \int_{-1+x}^{0+x} \utest(x-s) \drm s
\notag 
\\
& \eqtext{$0\leq x-s \leq 2$}
\ftex{x;0} \int_{-2+x}^{-1+x} \idfunc{0}{1}(x-s) \drm s
+ \ftex{x;1} \int_{-1+x}^{0+x} \idfunc{0}{1}(x-s) \drm s
\label{eq:ex2:convolvL}
\\
& \eqtext{change: $t=x-s$}
\ftex{x;0} \int_{1}^{2} \idfunc{0}{1}(t) \drm t
+ \ftex{x;1} \int_{0}^{1} \idfunc{0}{1}(t) \drm t
\notag 
\\
&= \ftex{x;1} \eqtext{by \equaref{eq:LfilterCoef}} \frac{1}{2}(3-2x). 
\notag 
\end{align}
The equality \equaref{eq:ex2:convolvL} holds because
when $0\leq x \leq 2$ and $-2+x \leq s \leq 0+x$ then
$0\leq x-s \leq 2$. Hence $\utest(x-s) = \idfunc{0}{1}(x-s)$.
As \thmref{thm:symFilteredDG} predicts
$(\ftex{x} \convol \utest)(x) = \frac{1}{2}(3-2x), \, x\in[0,2]$, 
is a polynomial over the \bdy\ $[0,2]$. 
\end{example}

\figref{fig:filters} graphs instances of the \srv, \rlkv\
and \newft\ kernels. 
\rtext{Note that the \newft\ filter remains piecewise constant,
while the degree of the other two filters increases with the degree $\ddg$
of the \DG\ data.}

\begin{figure}[ht!]
\def\figw{.45\textwidth}
 \centering
 \begin{tabular}{cc}
  \includegraphics[width=\figw]{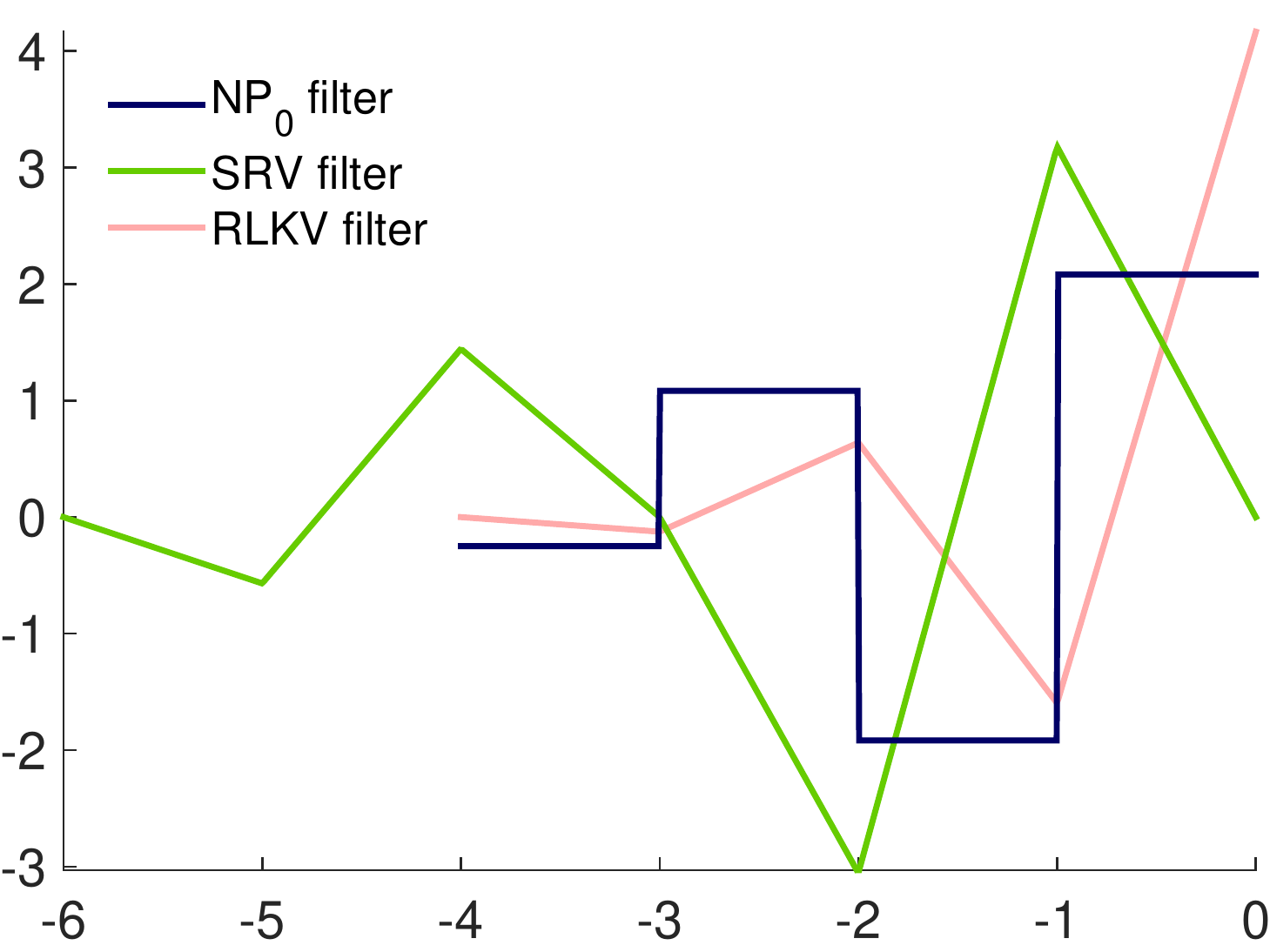}
  &
  \includegraphics[width=\figw]{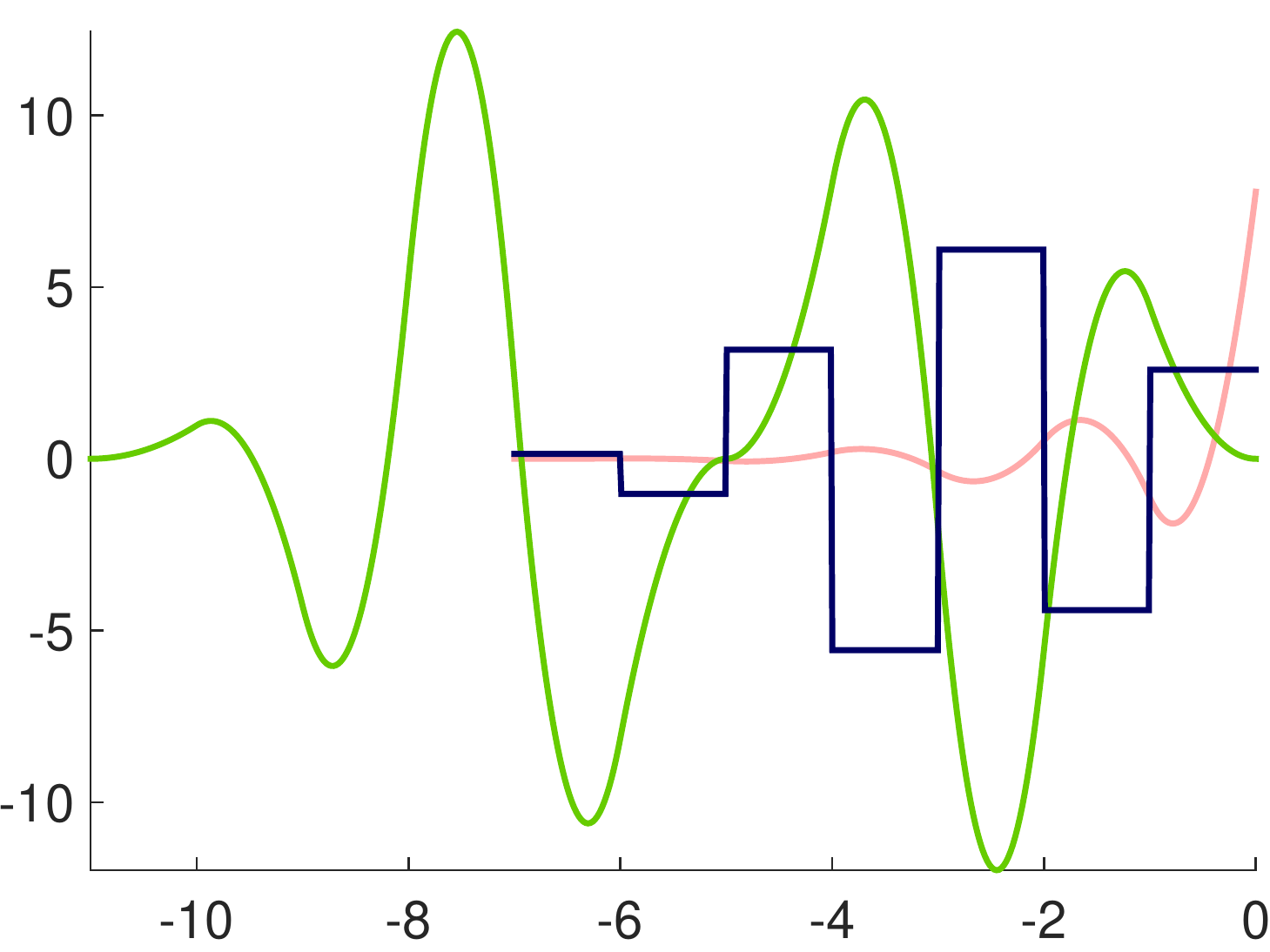}
  \\
  (a) \filter s for $\ddg=1$
  &
  (b) \filter s for $\ddg=2$
  \\
  \includegraphics[width=\figw]{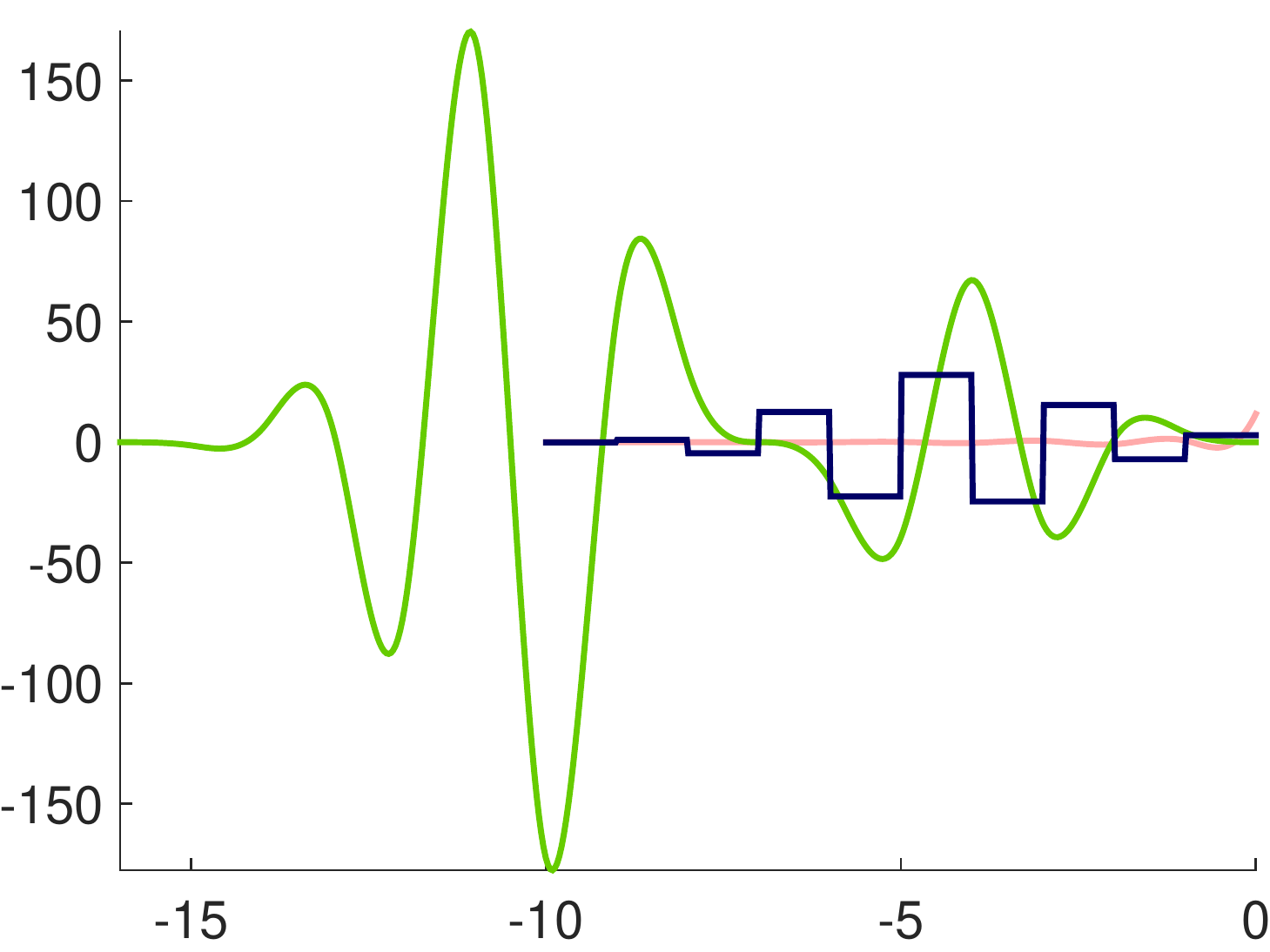}
  &
  \includegraphics[width=\figw]{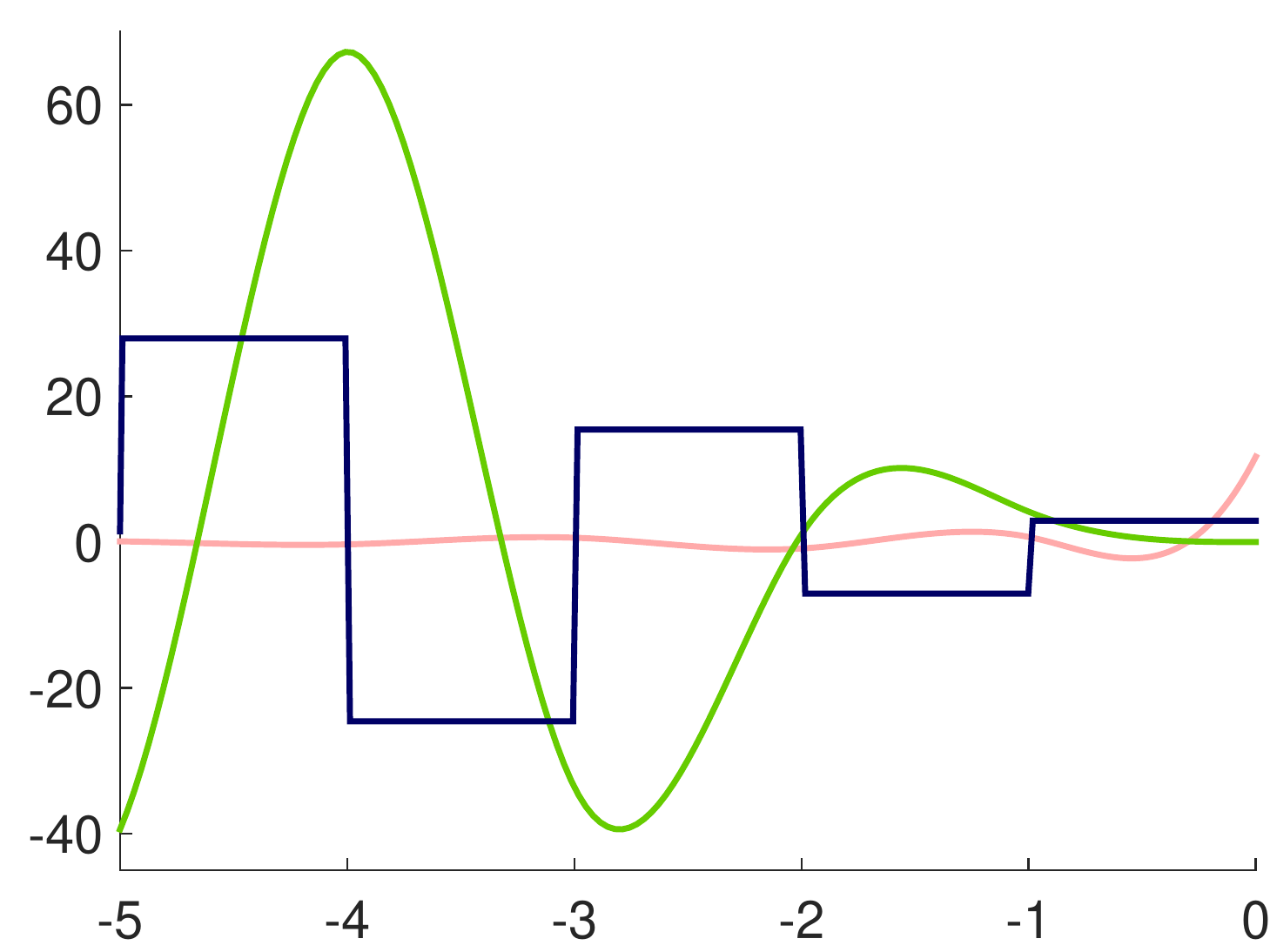}  
  \\
  (c) \filter s for $\ddg=3$
  &
  (c') right zoom of (c)
 \end{tabular} 
\caption{Graphs of the three \filter s defined at the left boundary $x=a$.
Note that the degree of \cite{siac2011}  and \cite{siac2014} increases with
$\ddg$ while the degree of the new \newft\ kernel remains
piecewise constant.
The \newft\ kernel (blue) has the same support as the \rlkv\ kernel (red), 
smaller than the \srv\ kernel (green).
}
\label{fig:filters}
\end{figure} 

%

\subsection{Symbolic form}
\label{sec:lsiacSymbolic}
We reduce the convolution of the \DG\ data with the \newft\ kernels
to an inner product of two short vectors. 
The inverse of the \siac\ reproduction matrix $M_{ \bkft,\idxft}$ 
and the matrix $\ipmatl{*}$,
of the formulation  \eqref{eq:symFilteredDG}
in Theorem \ref{thm:symFilteredDG}, 
are explicitly derived for \newft\ based on the following two propositions. 

\begin{prop}(\siac\ reproduction matrix for degree $\dft=0$)
The \siac\ reproduction matrix for least degree filters
 with index sequence $\idxft := (\idxfst,\ldots,\idxlst)$  and
knot sequence $\kft_{0:\nkft}$ is 
\begin{align}  
   M_{\kft_{0:\nkft},\idxft}
   = \left[ \begin{matrix}        
    \frac{\kft_{\sft+1}^{\dg+1} - \kft_{\sft}^{\dg+1} }{\kft_{\sft+1} - \kft_{\sft}}
   \end{matrix} \right]_{\dg=0:\nc,\,\sft\in\idxft}.
   \label{eq:M0powerForm:lsiac}         
\end{align}
If the knot sequence $\kft_{0:\nkft}$ is uniform, i.e.~$\hk := \kft_{\sft+1}-\kft_{\sft}$, then
\begin{align} 
   M_{\kft_{0:\nkft},\idxft}
   = \hk^{-1}
\left[ \begin{matrix}        
    (\kft_{\sft}+\hk)^{\dg+1} - \kft_{\sft}^{\dg+1} 
   \end{matrix} \right]_{\dg=0:\nc,\,\sft\in\idxft}.
   \label{eq:M0powerForm:lsiac:uniform}         
\end{align}
\end{prop}
\begin{proof}
Since $k=0$, each entry of  $M_{\kft_{0:\nkft},\idxft}$ given 
by Eq.~\eqref{eq:M0powerForm} is the sum of all monomials 
in $\kft_{j+1}$ and $\kft_{j}$ of total degree $\dg$,
 and hence of the form \eqref{eq:M0powerForm:lsiac}. 
Eq.~\eqref{eq:M0powerForm:lsiac:uniform}  is a direct consequence of
Eq.~\eqref{eq:M0powerForm:lsiac}.  
\end{proof}
 


\begin{prop} [$\ipmatl{*}$ for degree $\df=0$ and uniform \DG\ intervals]
Assume that the \DG\ break point sequence $\gs_{0:\nkdg}$ is uniform, 
hence after scaling consists of consecutive integers.
Without loss of generality, the \DG\ output on each interval $[s_i,s_{i+1}]$ 
is defined in terms of \bbname\ polynomials $\bbfun{\ell}{i}$ of degree $\ddg$,
where the superscript $i$ indicates
the interval and
$\ell=0:d$,  i.e.\ 
 \begin{equation}
 \bbfun{\ell}{i}(x) :=
 \begin{cases}
 \binom{\ddg}{\ell} (x-s_i)^\ell (s_{i+1}-x)^{\ddg-\ell} \quad & \text{if $x \in [s_i,s_{i+1}]$}
 \\
 0 \quad &\text{otherwise.}
 \end{cases}
 \end{equation}

Let $I$ be the $3\dDG+1$ identity matrix
and $\mathbf{1}$ the $(\dDG+1)$ column vector of ones.
The matrix $\ipmatl{L}$ for the left-sided kernel
and the matrix $\ipmatl{R}$ for the right-sided kernel
defined by Eq.~\eqref{eq:symFilteredDG} are
\begin{equation}   
   \ipmatl{L} = \ipmatl{R} = I \otimes \mathbf{1} =
   \frac{1}{\ddg+1} \,    
   \left[
   \begin{smallmatrix}
   1 & \ldots & 1 & 0 &\ldots & 0 & \ldots & 0 & 0 & \ldots & 0 \\
   0 & \ldots & 0 & 1 &\ldots & 1 & \ldots & 0 & 0 & \ldots & 0 \\
   \vdots & \ddots & \vdots & \vdots &\ddots & \vdots & \ddots & \vdots & \vdots & \ddots & \vdots \\
   0 & \ldots & 0 & 0 &\ldots & 0 & \ldots & 0 & 1 & \ldots & 1 \\
   \end{smallmatrix}
   \right]^{\tr}
   \in \mathbb{N}^{(\ddg+1) (3\ddg+1)} 
   	\times
   	\mathbb{N}^{(3\ddg+1)}.
   \label{eq:matrixG:k0}
\end{equation}
\end{prop} 
\begin{proof}
\newcommand{\kfti}{\widehat{\kft}} 
\newcommand{\nkfti}{3\dDG+1}
\newcommand{\Tlam}[1]{ \ipmatl{#1} }
First we derive Eq.~\eqref{eq:matrixG:k0} for $\ipmatl{L}$. 
In Eq.~\eqref{eq:Tcinvariant}, we change to the variable 
$t = s - \refx_L + \kft_\lst$. 
Since the B-splines are translation invariant and $\dft=0$ (hence $\nkft = \idxlst + 1$) 
\begin{equation}
\Bsp{s}{\refx-\kft_{\nkft-j:\idxlst-j}} = 
\Bsp{t}{\kft_\lst-\kft_{\nkft-j:\nkft-j-1}} = 
\Bsp{t}{j:j+1}.
\end{equation} 

Consequently, Eq.~\eqref{eq:Tcinvariant} can be rewritten as
\begin{equation}
    \Tlam{L}(i,j) = \int_{0}^{\nkfti} \, 
    \Nbf{i}{t}{\gs_{0:\nkdg}-\refx_L+\kft_\lst} 
   \, \Bsp{t}{j:j+1} \, \drm t.
   \label{eq:Tcinvariant1}   
\end{equation}

Since $\gs_0 = \frac{a}{h}$, the one-sided condition $\lambda_L = \kft_\lst + \frac{a}{h}$
implies that
the first point of the sequence of translated \DG\ break points
$\gs_{0:\nkdg}-\refx_L+\kft_\lst$ equals $0$, i.e.,
$\gs_0 - \refx_L + \kft_\lst = 0$. 
Since the break points are consecutive integers starting from 
$0$ and the B-splines $\Bsp{t}{j:j+1}$ are supported over $[0,\nkfti]$,
the relevant \DG\ break points are 
$0:\nkfti$.

We re-write the basis functions 
$\Nbf{j}{s}{\gs_{0:\nkdg}-\refx_L+\kft_\lst}$, that are supported on 
an interval $[i..i+1]$, in terms of \DG\ output
Bernstein-B\'{e}zier basis functions $\bbfun{\ell}{i}$, $\ell=0..\ddg$.
Since each B-spline $\Bsp{t}{j:j+1}$ is supported on $[j,j+1]$
and each $\bbfun{\ell}{i}$ is supported on $[i,i+1]$, the 
entries of $\ipmatl{L}$ are non-trivial only if $i=j$. Therefore
\begin{equation}
  \Tlam{L}\big( (\ddg+1)i + \ell, i\big) =
    \int\limits_{i}^{i+1} \bbfun{\ell}{i}(t) \,  \Bsp{t}{ i:i+1} \,  \mathrm{d} t
    = \frac{1}{h} \times \frac{h}{\dDG+1} = \frac{1}{\dDG+1}.
    \label{eq:Glam:result}
\end{equation}
The second last equality in Eq.~\eqref{eq:Glam:result} holds since 
$\Bsp{t}{i:i+1}|_{[i,i+1]} \equiv \frac{1}{h} $ 
and the integral of $\bbfun{\ell}{i}(t)$ over $[i,i+1]$ equals 
$\frac{h}{\dDG+1}$ \cite{Boor:2002:BB}.
Eq.~\eqref{eq:Glam:result} shares the entries of Eq.~\eqref{eq:matrixG:k0}
for $\ipmatl{L}$. 

A similar argument for deriving $\ipmatl{R}$ but starting with the substitution
$t :=-(s-\refx_R+\kft_0)$
proves Eq.~\eqref{eq:matrixG:k0}
for $\ipmatl{R}$. 
\end{proof}

\subsection{Filter transition}
\newcommand{\ftsym}{\mathcal{K}}
\newcommand{\ftleft}{\mathcal{L}}
Let $\ftleft_{x} \convol \dgout_h$ denote the \DG\ output
$\dgout_h$ filtered  
by the left boundary PSIAC filter $\ftleft_x$ and
$\ftsym \convol \dgout_h$ the \DG\ output
filtered by the symmetric interior filter $\ftsym$.
These two filtered outputs overlap on the interval
$[a_1,a_2] = a+c_L + [0..2h]$ where $c_L$ separates 
the interior and left boundary.
Without loss of generality, after substituting $z := (x-a_1)/(a_2-a_1)$, 
we may assume that $a_1 := 0, \, a_2 := 1$.
\cite{siac2011} suggests a smoothness-preserving
transition filtering scheme that we state more succinctly as 
\begin{align}
   u^\star_h(x) &:= (1-\blend{}(x))\, \big( \ftleft_{x} \convol \dgout_h \big)(x) 
    + \blend{}(x)\, \big( \ftsym \convol \dgout_h\big)(x), 
   \label{eq:blendFun}
   \\
   &\blend{}(x) := \sum^{2\rho}_{i=0}  \blend{i} \bbfun{i}{}(x),
   \quad 
   \blend{i} := 
   \begin{cases}
      0 & \text{ if } i\le \rho;\\
      1 & \text{ else. }
   \end{cases}
   \notag
\end{align}
where $\bbfun{i}{}(x) = \binom{2\rho}{i} (1-x)^{2\rho-i}x^i$ 
are \bbname\ polynomials of degree $2\rho$
defined over the unit interval $[0..1]$.
The 
\bbname\ representation
guarantees that $\dgout^\star_h$ 
Hermite-interpolates both filtered \DG\ output up to order $\rho$. 
The degree $2\rho$ need not be twice 
the minimum degree of the boundary filter and the symmetric filter
but can be chosen to further smooth out the transition.
That is, one may choose $\rho=2$ even though $d=\df = \ds = 1$. 
The transition for $\rho=2$ is less abrupt at $0$ and $1$ 
than for $\rho=1$.


\section{Numerical Comparison of the \srv, \rlkv, \newft\ and the 
symmetric \siac\ filter}
\label{sec:errnumer}
The goal of this section is to compare the new \newft\ filters with the 
state-of-the-art filters, \srv\ \cite{siac2011} and \rlkv\ \cite{siac2014}
\rtext{in their stable symbolic form \cite{psiac}.}
In particular, 
\ccchg{this section explains and discusses the graphs in the Appendix
that represent ca. 15000 convergence rate measurements.
All errors are computed on the regions where the filters apply.
That is, the error of the boundary filters is measured on the \bdy s only,
the error of the symmetric SIAC filter is measured only in the interior,
whereas the error of the \DG\ output is measured over the whole domain.}

The comparison will be based on three test problems that are 
special instances of the canonical problem \eqref{eq:hypEqs}, 
$ 
   \frac{du}{d\tm} +
   \frac{d}{dx}\Big( \kappa(x,\tm) \, u\Big) = \rho(x,\tm)
$
for $x \in (a..b), \,\tm \in (0..\tmend)$.
\ccchg{That is, we measure the convergence not just
for the final time $\tmend=2\pi$ but also for many other
final times such as $\tmend = 0.7*2\pi$.}

\begin{testproblem}[Constant wave-speed, {\em periodic} boundary conditions]
\label{prob:cs-p}
 Consider the specializations of Eq.~\eqref{eq:hypEqs}:
\begin{align}
   \kappa(x,\tm) \equiv 1,\quad
   \rho(x,\tm) \equiv 0,\quad
   0 \leq \tm \leq \tmend,
   \label{eq:kr}
\end{align}
$a:=0$, $b:=1$,
 {periodic} boundary conditions, $u_0(x) := \sin (2\pi x)$.
The exact solution 
\cchg{ at the \ftime\ $\tmend$ is  
$\dgout(x,\tmend) = \sin(2\pi (x-\tmend))$.}
\end{testproblem}

\begin{testproblem}[Constant wave-speed, {\em Dirichlet} boundary conditions]
\label{prob:cs-d}
Consider Eq.~\eqref{eq:hypEqs} with specializations \eqref{eq:kr}, but
$a:=0$, $b:=2 \pi$,
 {Dirichlet} boundary conditions, $u_0(x) := \sin (x)$, 
 $u(0,\tm) := -\sin (\tm)$.
The exact solution is  
\cchg{ at the \ftime\ $\tmend$ is  
$\dgout(x,\tmend) = \sin(x-\tmend)$}.
\end{testproblem}

\begin{testproblem}[{\em{Variable}} wave-speed, periodic boundary conditions]
\label{prob:vs-p}
Consider Eq.~\eqref{eq:hypEqs} with the specializations
\begin{align}
   \kappa(x,\tm) := 2 + \sin(x+\tm),\quad
   \rho(x,\tm) := \cos(x-\tm) + \sin(2x),\quad
   0 \leq \tm \leq \tmend,
   \label{eq:kr2}
\end{align}
$a:=0$, $b:=2 \pi$, and
 {periodic} boundary conditions, $u_0(x) := \sin (x)$.
The exact solution   
\cchg{ at the \ftime\ $\tmend$ is  
$\dgout(x,\tmend) = \sin(x-\tmend)$.}
\end{testproblem}

The comparison of the filters yields broadly the same 
qualitative results for all three test scenarios.
We display the $L^\infty$ 
and the $L^2$ convergence rates
in the Appendix but discuss the results in this section.

\figref{fig:pbc:pw-errors}, \ref{fig:dbc:pw-errors} and
\ref{fig:xt-speed:pw-errors} juxtapose
the \emph{maximal point-wise errors for each \ftime\ $\tmend$} 
of the \DG\ output (over the whole domain) with that of the 
boundary filters \srv, \rlkv, \newft\ (restricted to the \bdy s)
and the symmetric SIAC filter (restricted to the interior).
\rtext{The symmetric SIAC filter is graphed as a light grey dotted
line and applies and always yields the optimal convergence rate $2\ddg+1$.}
This \emph{{\bf time series}} of errors differs from the commonly-used error graphs 
such as in \figref{fig:compare-d2}a,b,c,
where the abscissa represents the one-dimensional domain of computation
for a fixed $\tmend$.
\ccchg{
By contrast in \figref{fig:pbc:pw-errors}, \ref{fig:dbc:pw-errors} and
\ref{fig:xt-speed:pw-errors} the abscissa represents
the final time $T$ for computing the Test equation.
That is, $T=0.4$ means that the equations has been computed by \DG\ up to time
$T=0.4$ (rather than $T=1$) and then the filters have been applied.}
The graphs 
therefore represent a 
large number of measurements: for each \ftime\ $\tmend$ on the abscissa, for each
boundary filter,
the ordinates of \figref{fig:pbc:pw-errors}, \ref{fig:dbc:pw-errors} and
\ref{fig:xt-speed:pw-errors} show the maximal point-wise error over the region
where they apply.

\figref{fig:pbc}, \ref{fig:dbc} and \figref{fig:xt-speed} display a
final-time series of the convergence rates
\begin{equation}
   \crate = \crate(\tm,h) := 
   \ln\big( \frac{\text{error at time $\tau$ for mesh size }  2h}
                 {\text{error at  time $\tau$ for mesh size }  h}\big) /
   \ln 2.
\end{equation}
for the $L^2$ and the $L^\infty$ norm, and for the left \bdy\ and the right \bdy.
\figref{fig:pbc} and \figref{fig:dbc} each show 
\#(degrees$\times$norms$\times$\bdy s$\times$\ftime s$\times$filter types \& \DG\ output
$\times$refinements)
= 
(3$\times$2$\times$2$\times$50$\times$5$\times$2)
= 6000
convergence rates. 
\figref{fig:xt-speed}, with 30 \ftime s, shows 3600 convergence rates.
\cchg{
For example, in \figref{fig:pbc:detail}a, at \ftime\ $\tmend=0.3$,
the convergence rate for symmetric SIAC filter (restricted to the interior of the domain)
for degree $\dft=1$ is $\crate=3$;
the rate for the \rlkv-filter (restricted to the \bdy s) is $\crate=2$;
and the rates are $\crate=3$  for the \srv-filter and the \newft-filter.
}

\def\wtwo{.48\textwidth}
\begin{figure}[h]
\centering 
   \subfigure[$\ddg =\df = 1$ ($\df = 0$ for the \newft\ filter)]{
      \includegraphics[width=\wtwo]{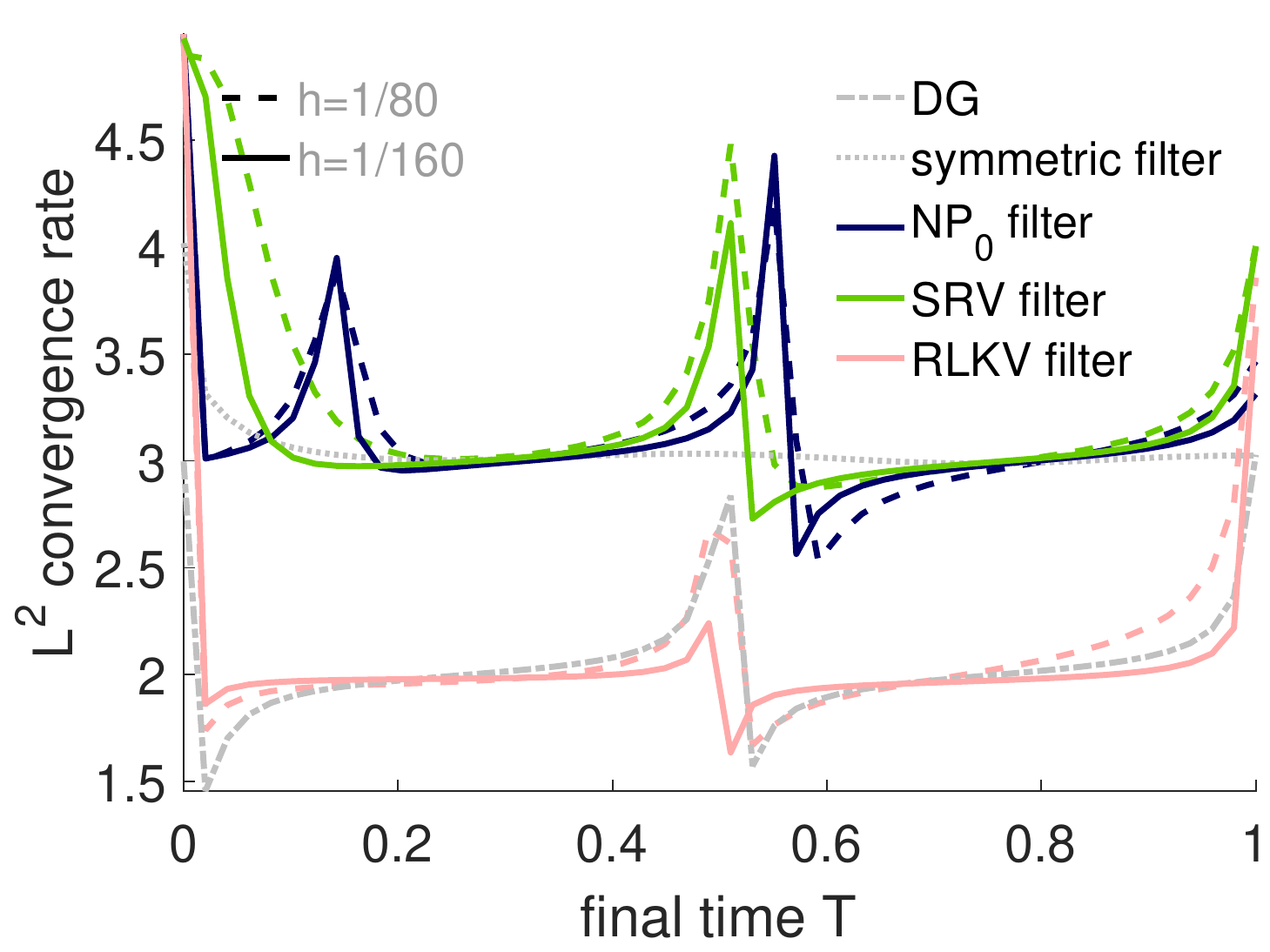}}
   \hfill
   \subfigure[$\ddg = \df = 2$ ($\df = 0$ for the \newft\ filter)]{
      \includegraphics[width=\wtwo]{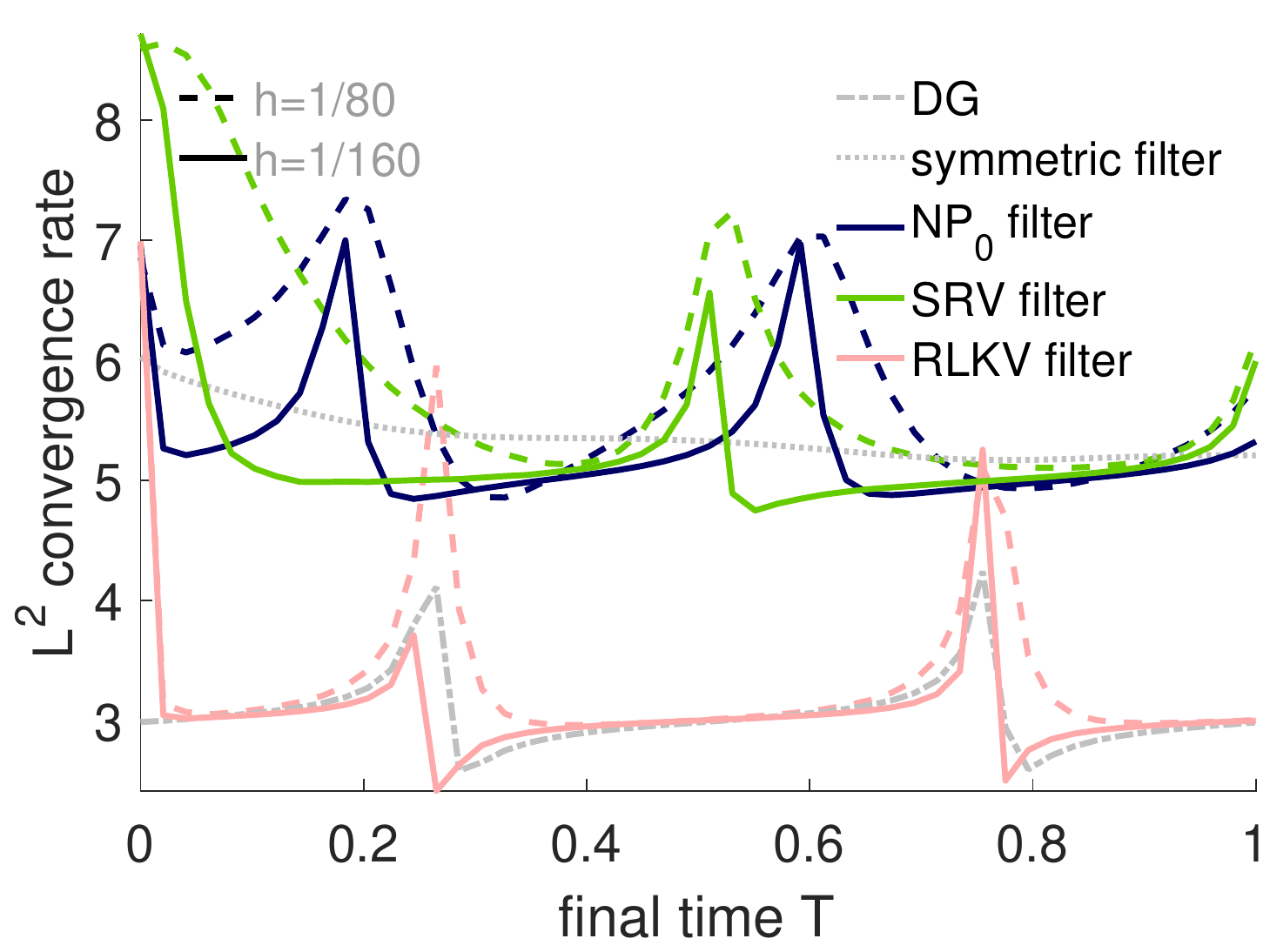}}
\caption{
   {\bf Convergence rate time series}	
   \probR{prob:cs-p}{constant}{periodic} for 
   $L^2$-error convergence rates of \bsiac\ filters.
   Each point on the abscissa represents one \ftime\ $0\leq \tmend \leq 1$
   and each ordinate represents a $L^2$-convergence rate.		
   A complete set of graphs is shown in \figref{fig:pbc}.
}
\label{fig:pbc:detail} 
\end{figure} 
The time-series yields additional insights: The convergence rate
fluctuates with $\tmend$ and can be particularly high for {some specific}
\ftime s $\tmend$.
\jorg{
And the time series in the Dirichlet scenario
in \figref{fig:jumpsDirichlet} shows spikes in the error,
for all filters including the symmetric SIAC filters.}

\def\figw{.45\textwidth}
\begin{figure}[ht!]
\subfigure[$\tmend = 0 \times 2\pi$]{ \includegraphics[width=\figw]{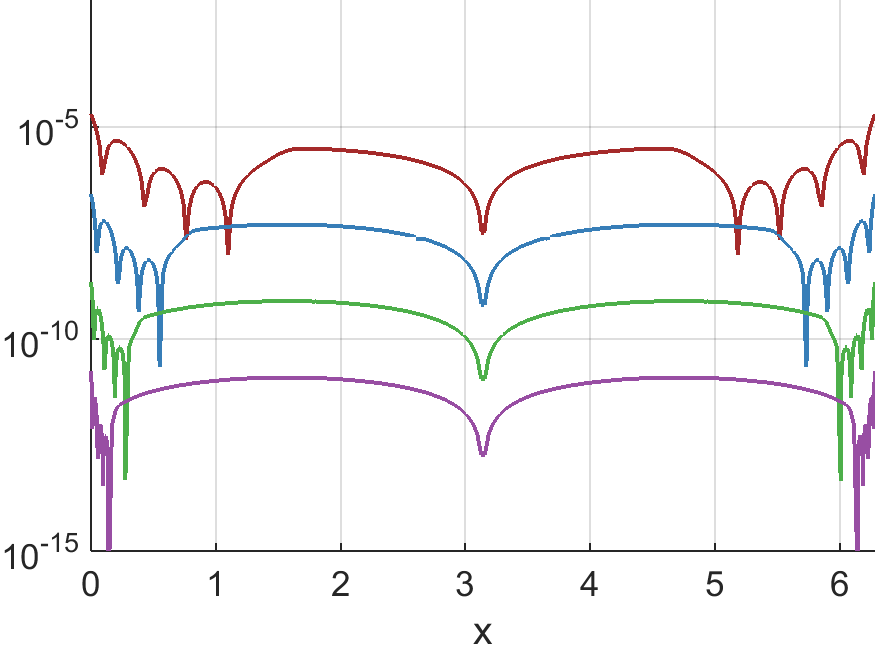}}	
\hfill 	
\subfigure[$\tmend = \frac{1}{3} \times 2\pi$]{ \includegraphics[width=\figw]{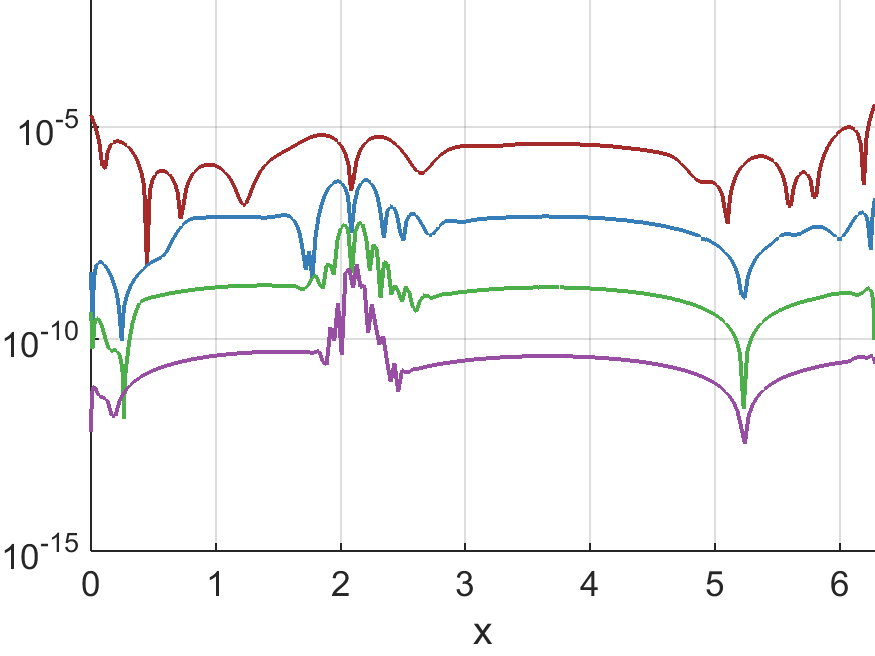}}	
\subfigure[$\tmend = \frac{2}{3} \times 2\pi$]{ \includegraphics[width=\figw]{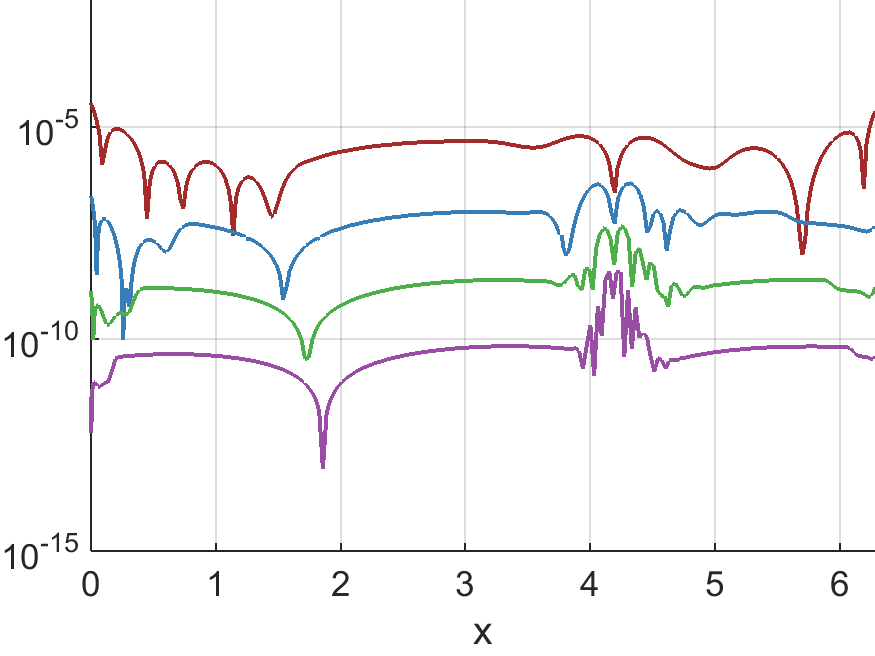}}	
\hfill
\subfigure[$\tmend = 1 \times 2\pi$]{ \includegraphics[width=\figw]{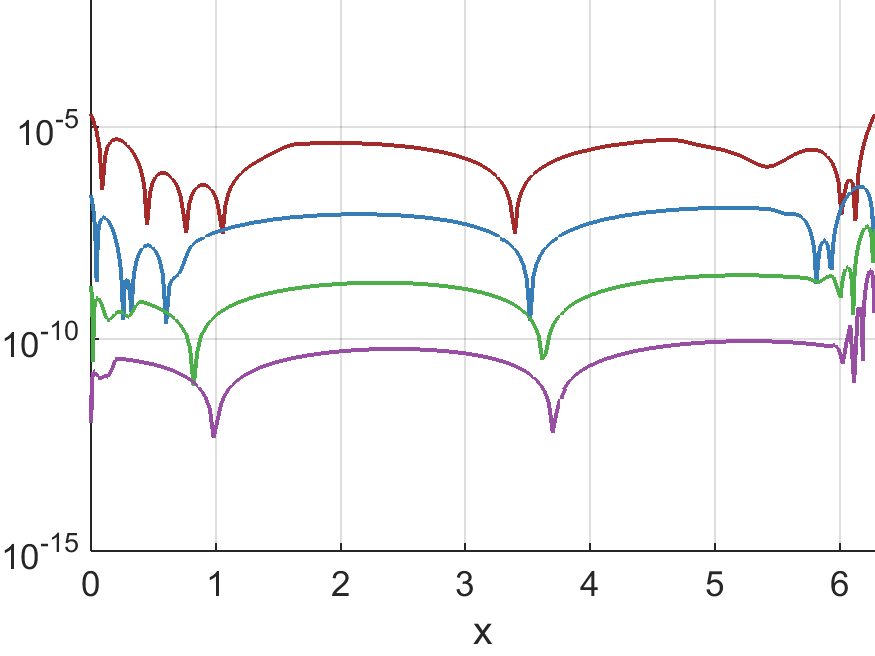}}	
\caption{Travelling spikes the in pointwise error graphs of the 
         filtered degree $\ddg=2$ \DG\ output.
        for the Dirichlet problem \ref{prob:cs-d}.
        \rtext{The pointwise errors are for mesh-size 
        	$h=20^{-1},40^{-1},80^{-1},160^{-1}$
        	(drawn in red, blue, green, purple respectively and from top to bottom in each panel)}
        }
\label{fig:jumpsDirichlet}
\end{figure}

\figref{fig:pbc:pw-errors} confirms, 
for all \ftime s,
the lower error of the \srv\ filter compared to the \rlkv\ filter.
However, instability of the \srv\ filter
at double precision for certain \ftime s $\tmend$ 
and higher degree of the \DG\ output increase the error,
even above that of the raw \DG\ output
(see \figref{fig:dbc:pw-errors} for $\ddg=2,3$
and \figref{fig:xt-speed:pw-errors} for $\ddg=3$).
Similarly, instability causes the error of \srv\ to exceed that of the
\newft\ filter.
Overall, when $\dDG=1$ or $\ddg=2$, point-wise errors of the \rlkv\ filter
are noticeably larger than those of the \srv\ filter and the \newft\ filter.
The point-wise errors of the \newft\ filter are on par 
with the symmetric SIAC filter (in gray).

We note the consistency between $L^2$ and $L^\infty$ convergence
and therefore clearly higher rates for 
\srv, \newft\ and the symmetric filters over \rlkv. 
%
When $\dDG=1$ and $\dDG=2$, the \newft\ and \srv\ filters 
show optimal $L^2$ and $L^\infty$ super-convergence rates of order $2\dDG+1$.
The convergence rate two for the \bdy s
for \rlkv\ and $\dDG=1$ has been verified
also with the original \rlkv\ code.



For $\ddg=3$ the point-wise errors and convergence rates of 
all filters oscillate when the calculations are close to machine precision. 
However, the point-wise errors and convergence rates of the \newft\ filter are
on par with that of the symmetric \siac\ filter, while the errors of 
the \srv\ filters are notably higher
and convergence rates become sub-optimal, and even lower than $\ddg+1$
for some $\tmend$ when $\ddg=3$.

\figref{fig:vecSymb} helps explain the instability of the \srv\ filters.
Plotting the alternating entries of the convolution vector
$
\vecSymb = 
   Q_{\refx} \brefx
$
at a boundary point $x$ (e.g., $x=0$ when $a=0$), we find 
the entries of the \srv\ filters 
to be two orders of magnitude larger than those of \newft.

\section{Conclusion}
The newly-discovered \bsiac\ \newft\ filters possess three main advantages,
especially when combined with the symmetric 
\siac\ filter in the interior. First,
the computation of the \newft\ filters and their application
for convolution is more stable due to their explicit representation
as small integer fractions.
Second, on canonical test equations, applying \newft\ filters
reduces the errors of the \DG\ output 
to that of optimal symmetric \siac\ filters.
Third, support of the \newft\ filters are of the same size as those of the 
symmetric \siac\ filters making them naturally compatible.


\section*{Acknowledgement}
\NSFsupport.
\ccchg{Jennifer Ryan kindly provided her group's \rlkv-code 
to allow us to confirm correctness of our symbolic 
implementation of \rlkv.}

\bibliographystyle{alpha}    
\bibliography{../15siac_mn/p,p}
\section*{Appendix: Numerical experiments - Error Graphs and Convergence Rates}
The following graphs compare the maximal pointwise errors and convergence rates
for the filtered \DG\ output.
The \DG\ output is of degree $\dDG=1,2,3$ with a uniform partition
of the $x$-domain into intervals of length $h=1/N$ and $N=20,40,80,160$.
Independent of $N$, a fixed number points in the \bdy s
are the result of filtering with (the symbolic version of)
\srv, \rlkv, or by \newft.
\ccchg{
Boundary filters are measured only in the \bdy s inclusive of the transition
region and errors of the symmetric filter are computed only for 
the interior exclusive of the transition region.
Under refinement the interior converges to the full domain
and for uniformly-spaced \DG\ output
the symmetric SIAC filter contributes an 
increasing number of values, while the number 
of values from the boundary filter remains the same.
Computing the $L^2$ error over the full domain, say $[0..2\pi]$,
does not assess boundary filters correctly since 
the error is dominated by the error of the increasing number
of symmetric-filtered samples. 
Consequently, errors and rates for the boundary filters
are \emph{not combined} with those of the symmetric filter.
(For all cases presented, the combined or global
$L^2$ error is that of the symmetric SIAC filter.)
}

\ccchg{
Unlike \figref{fig:compare-d2}, 
the graphs do \emph{not} show errors at domain points
$x$ at a fixed time $\tmend$.
Rather the errors are plotted as functions of \ftime\ $\tmend$,
i.e.\ as a \textbf{\emph{final-time series}}.
Each plot for \testref{prob:cs-p}
and \testref{prob:cs-d} shows the errors at $N=50$ uniformly-spaced 
{final times $\tmend$ in $[0,1]$ and $[0,2\pi]$ respectively}. 
Those for \testref{prob:vs-p} correspond to $N=30$ uniformly-spaced 
{$\tmend$ in $[0,2\pi]$}. That is, the graphs summarize the error
measurements for each \ftime\ $\tmend$.}
For example, the value of the \rlkv\ error graph at final time $T=0.4$ 
corresponds to solving the indicated Test equation by \DG\ up to time
$T=0.4$, then convolving with \rlkv\ at the boundary and computing the
error of the fixed number of \rlkv-filtered samples,
e.g.\ 18 samples for $\dDG=1$ and six samples per interval.
When $N=40$ the error of the symmetric SIAC filter is then computed at 
$N-18$ points.
The samples within each of the $N$ intervals ar uniformly distributed.
Changing the sample points to Gauss-Legendre points does not change
the picture.

All graphs share the same color and style assignments.
The raw $\DG$ output time series is graphed in light grey.
In \figref{fig:pbc:pw-errors}, \ref{fig:dbc:pw-errors},
and \ref{fig:xt-speed:pw-errors}, the raw $\DG$ output time series
is typically the top-most graph (largest error)
and the bottom-most in the graphs of
\figref{fig:pbc}, \ref{fig:dbc} and \ref{fig:xt-speed} that show the 
convergence-rate.
%
The symmetric SIAC filter for the given degree is 
graphed in light grey, typically located in the middle and relatively straight.
The symmetric SIAC filter only applies in the interior: 
for the different \ftime s, the maximal pointwise  error,
respectively the convergence rate
over the domain interior is plotted.
The boundary filters are \rlkv\ (red), \srv (green) and 
\newft\ filter (blue). 
In subfigures (a), (b), (c), (d), 
\emph{dashed graphs} correspond to $N=80$ 
and \emph{solid graphs} to $N=160$, a halving of $h$. 
For (e) and (f) 
In subfigures (e) and (f)
\emph{dashed graphs} correspond to $N=40$  
and \emph{solid graphs} to $N=80$.

\figref{fig:pbc} for \testref{prob:cs-p} 
and 
\figref{fig:dbc} for \testref{prob:cs-d},  
and 
\figref{fig:xt-speed} for \testref{prob:vs-p} 
display the \emph{ $L^\infty$ and $L^2$ convergence rates}
of the filtered \DG\ outputs obtained after 
convolving with the filters
(symmetric SIAC in the interior -- \srv, \rlkv\ or \newft\ in the \bdy).
\rtext{All computations are performed in double precision.}

\newcommand{\leftb}{Left boundary}
\newcommand{\rightb}{Right boundary}
\newcommand{\capMax}{
Abscissa: \ftime\ $T$ at which the \DG\ output is computed;
ordinate: maximum point-wise errors of filtered \DG\ data;
Dashed graphs correspond to 
$N=80$ in (a), (b), (c), (d) and to $N=40$ in (e), (f).
Solid graphs correspond to 
$N=160$ in (a), (b), (c), (d) and to $N=80$ in (e), (f).
}
\newcommand{\capRateAll}[1]{
$L^2$- and $L^\infty$ convergence rates of the convolved \DG\ data
at left and right boundary. Rows $#1$ show rates for degree $\dDG=#1$.
}

\def\wtwo{.45\textwidth}
\begin{figure}[ht!]
\centering 
\subfigure[\leftb, $\dDG=1$]
{\includegraphics[width=\wtwo]{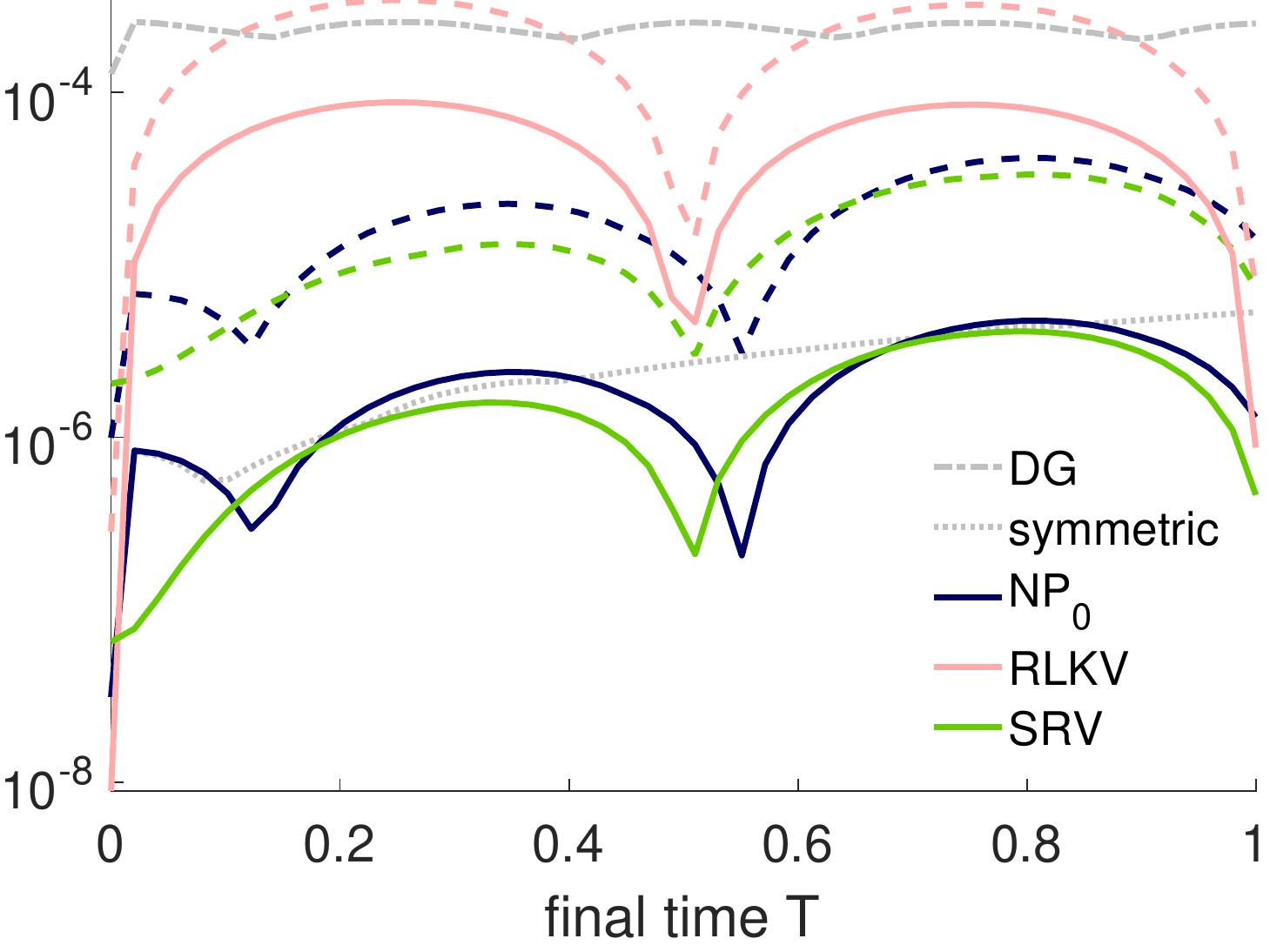}}
\hfill
\subfigure[\rightb, $\dDG=1$]
{\includegraphics[width=\wtwo]{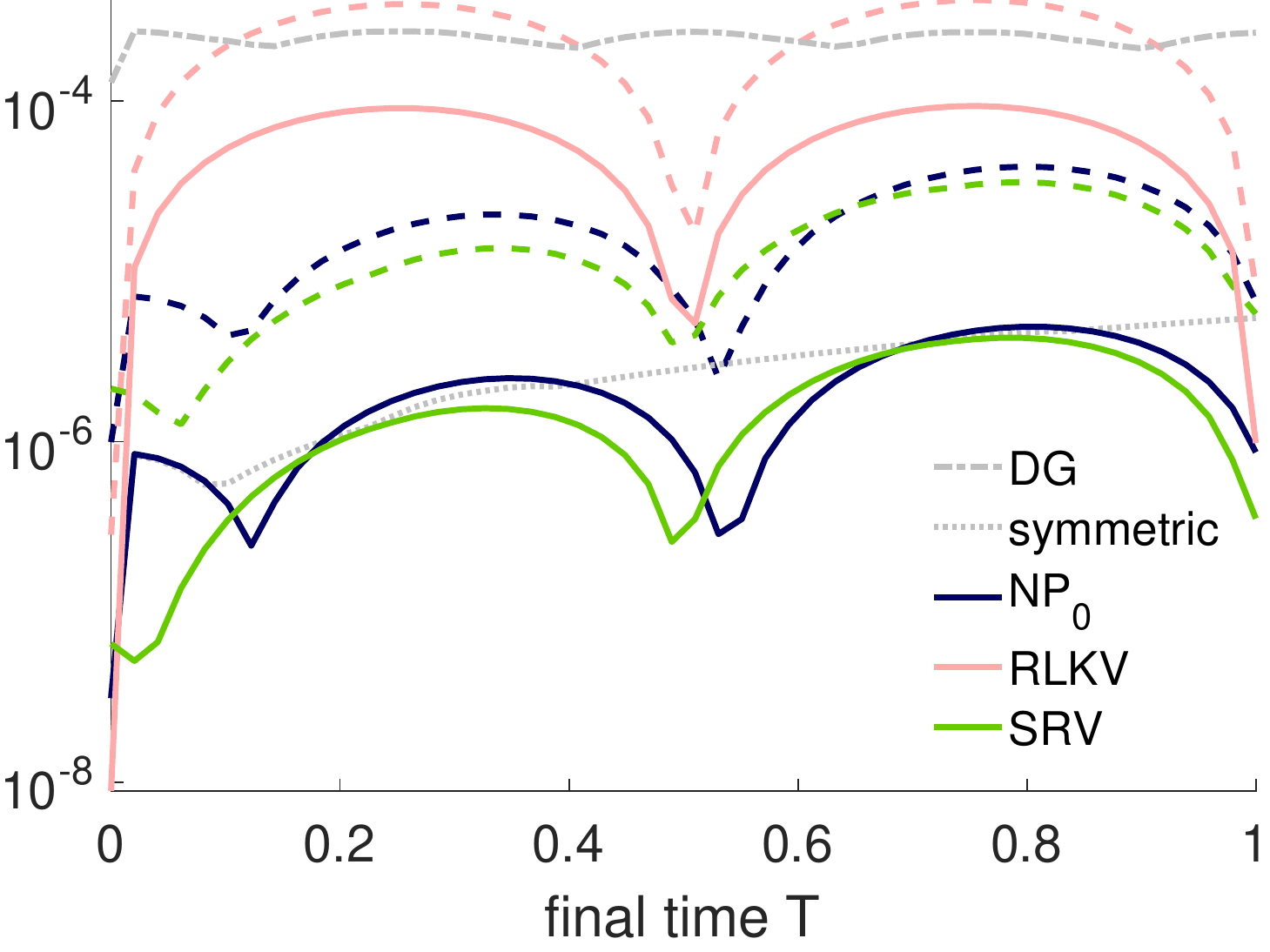}}

\subfigure[\leftb, $\dDG=2$]
{\includegraphics[width=\wtwo]{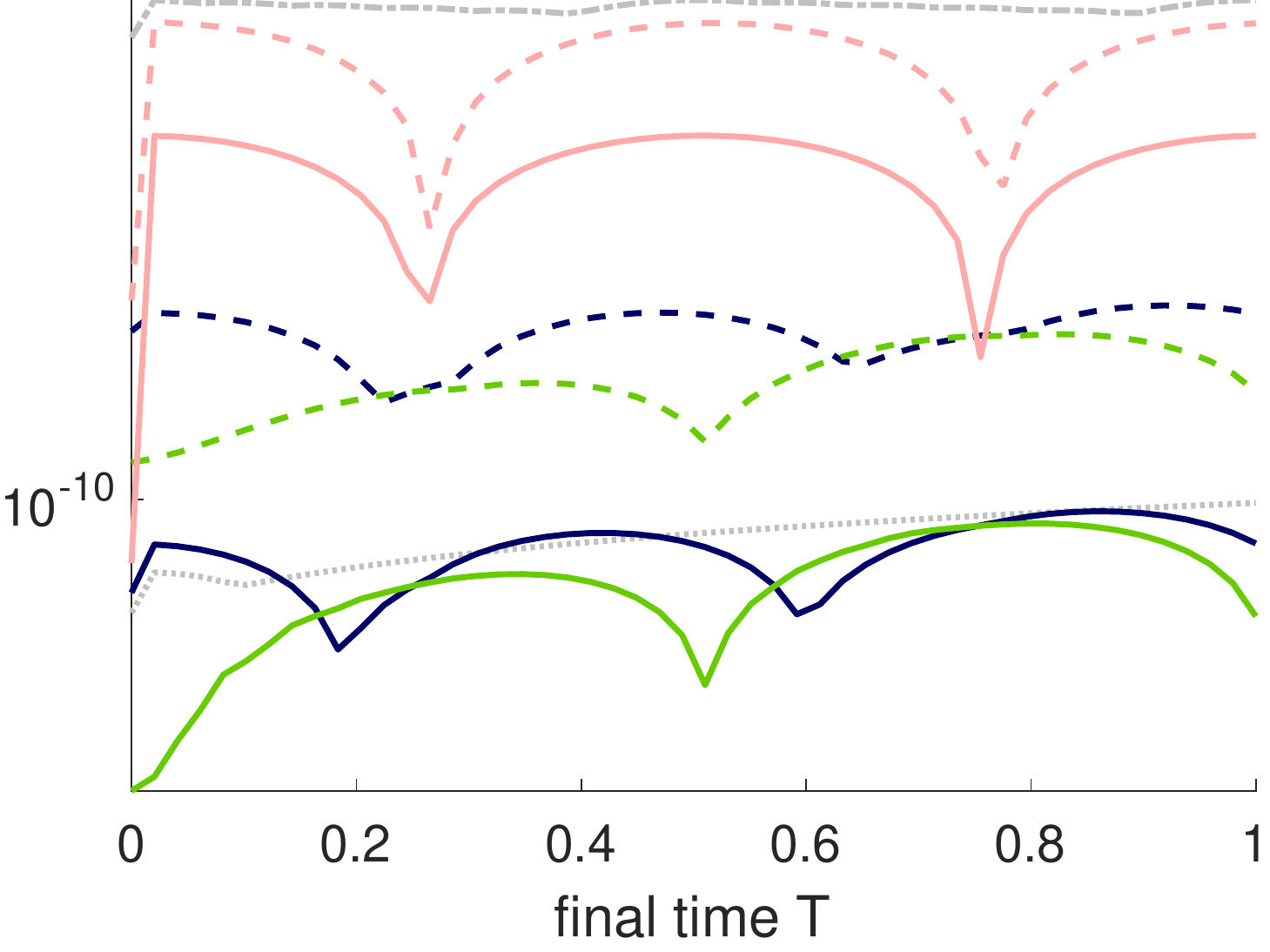}}
\hfill
\subfigure[\rightb, $\dDG=2$]
{\includegraphics[width=\wtwo]{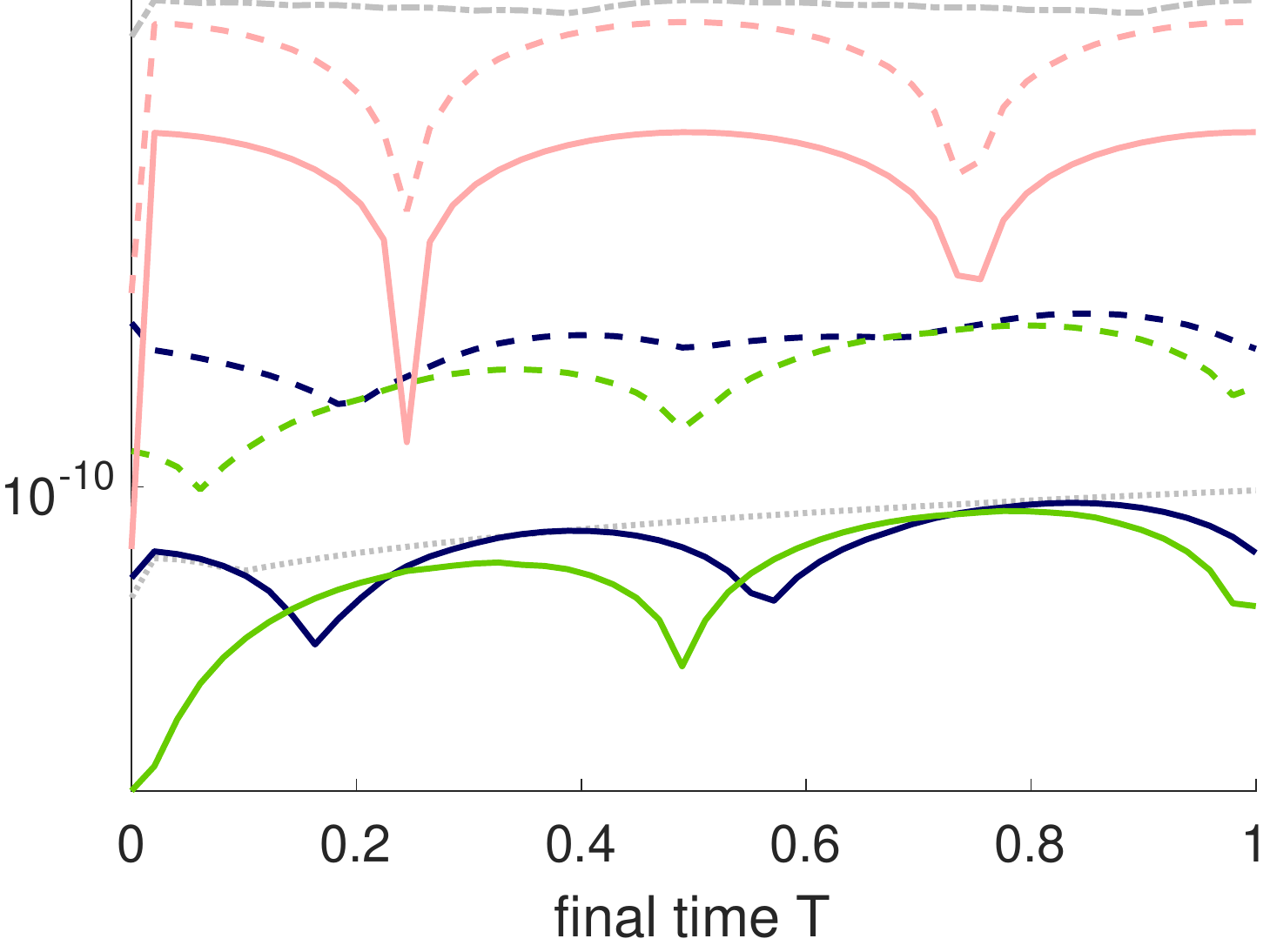}}

\subfigure[\leftb, $\dDG=3$]
{\includegraphics[width=\wtwo]{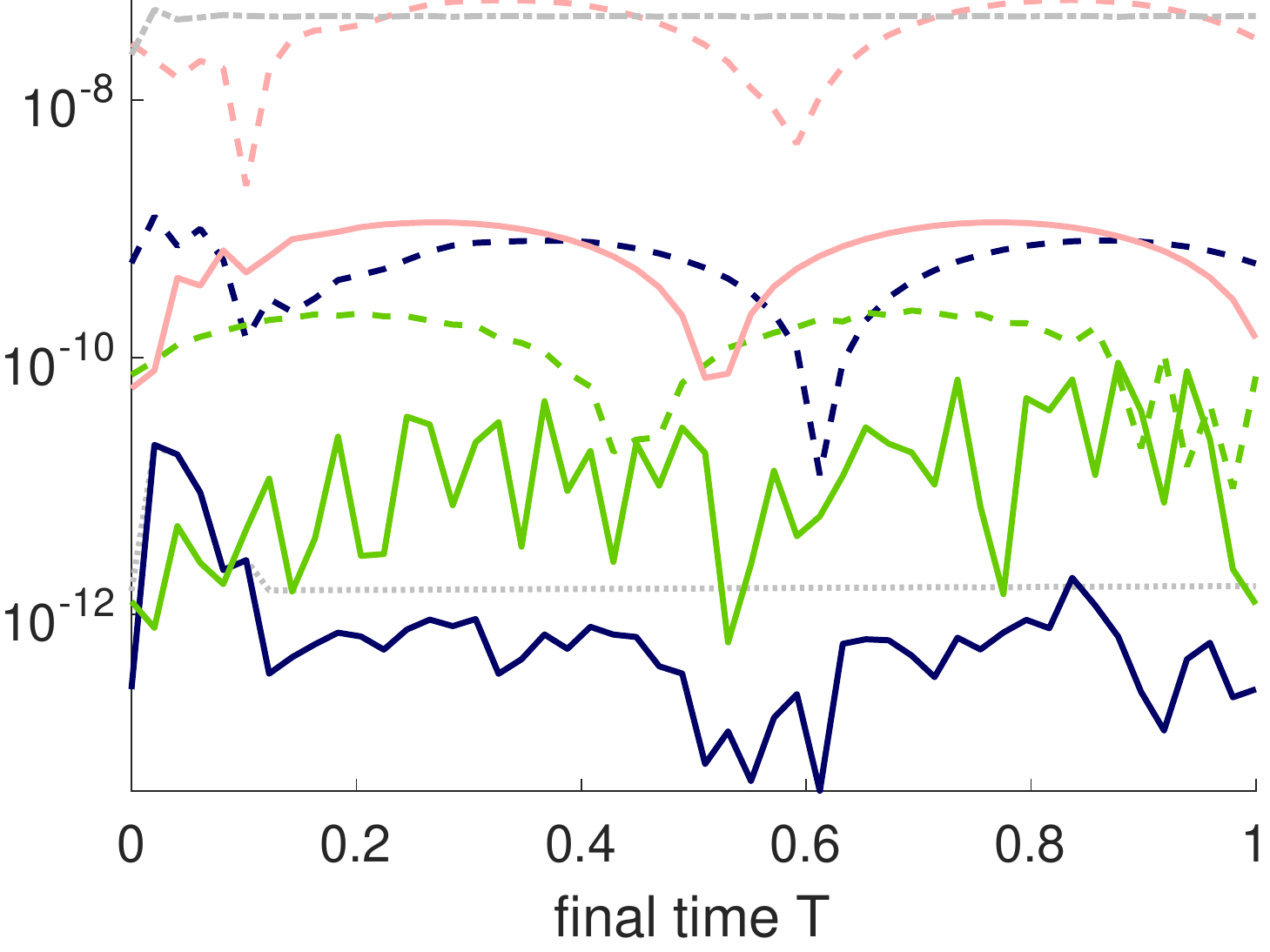}}
\hfill
\subfigure[\rightb; $\dDG=3$]
{\includegraphics[width=\wtwo]{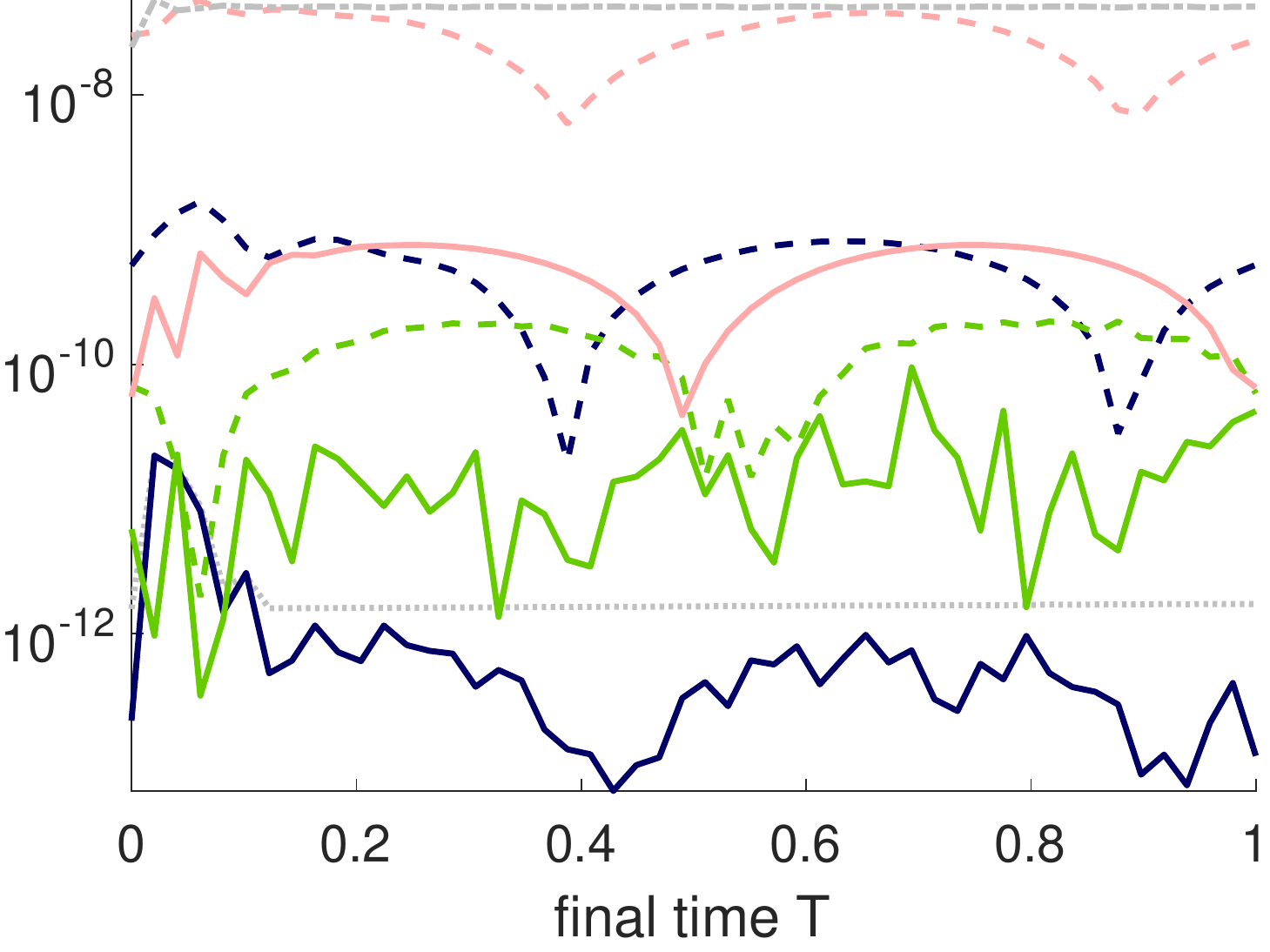}}
\caption{
\probErr{prob:cs-p}{constant}{periodic}
\capMax
}
\label{fig:pbc:pw-errors}
\end{figure}

\def\wtwo{.45\textwidth}
\begin{figure}[ht!]
\centering 
\subfigure[\leftb, $\dDG=1$]
{\includegraphics[width=\wtwo]{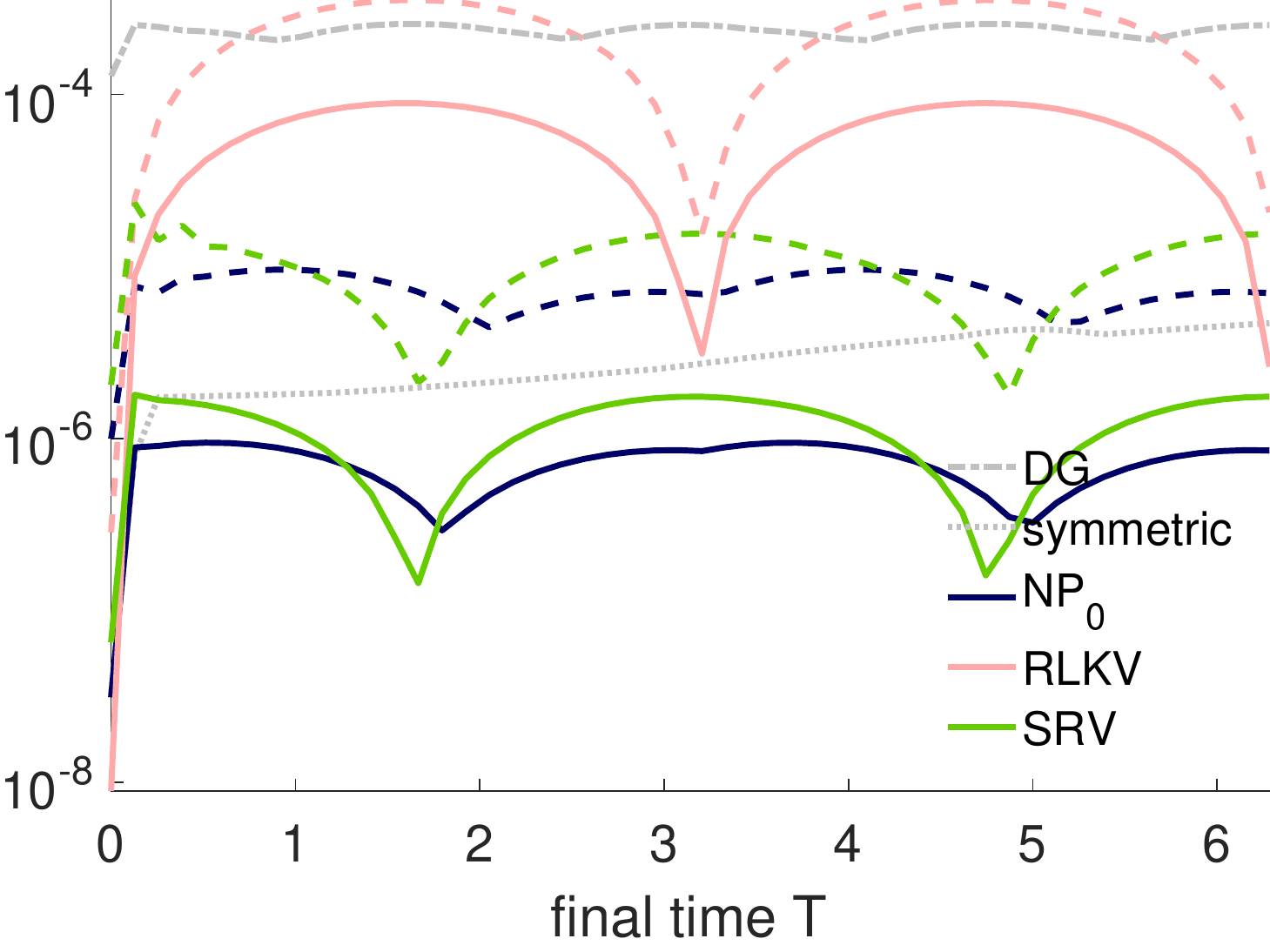}}
\hfill
\subfigure[\rightb, $\dDG=1$]
{\includegraphics[width=\wtwo]{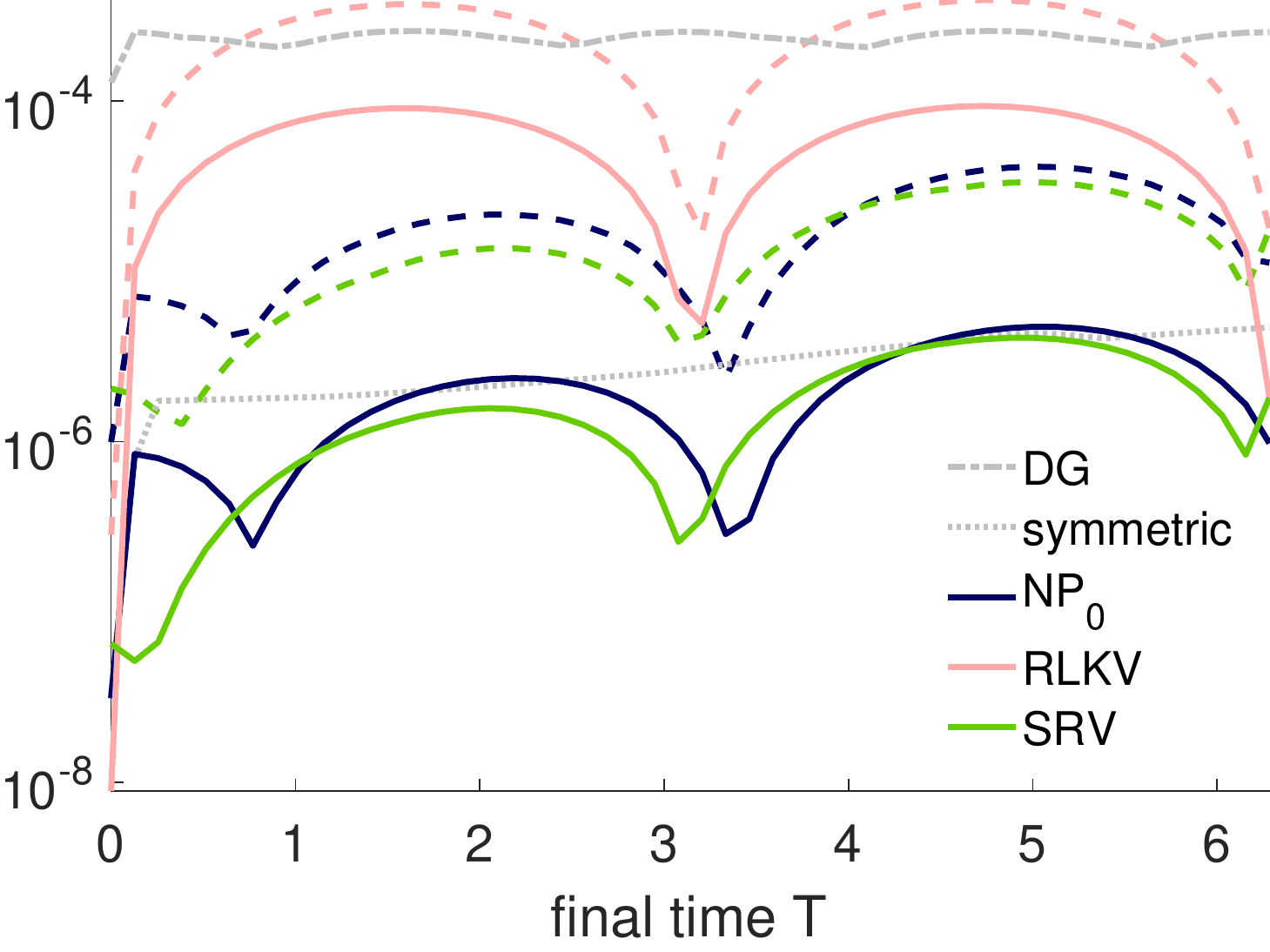}}

\subfigure[\leftb, $\dDG=2$]
{\includegraphics[width=\wtwo]{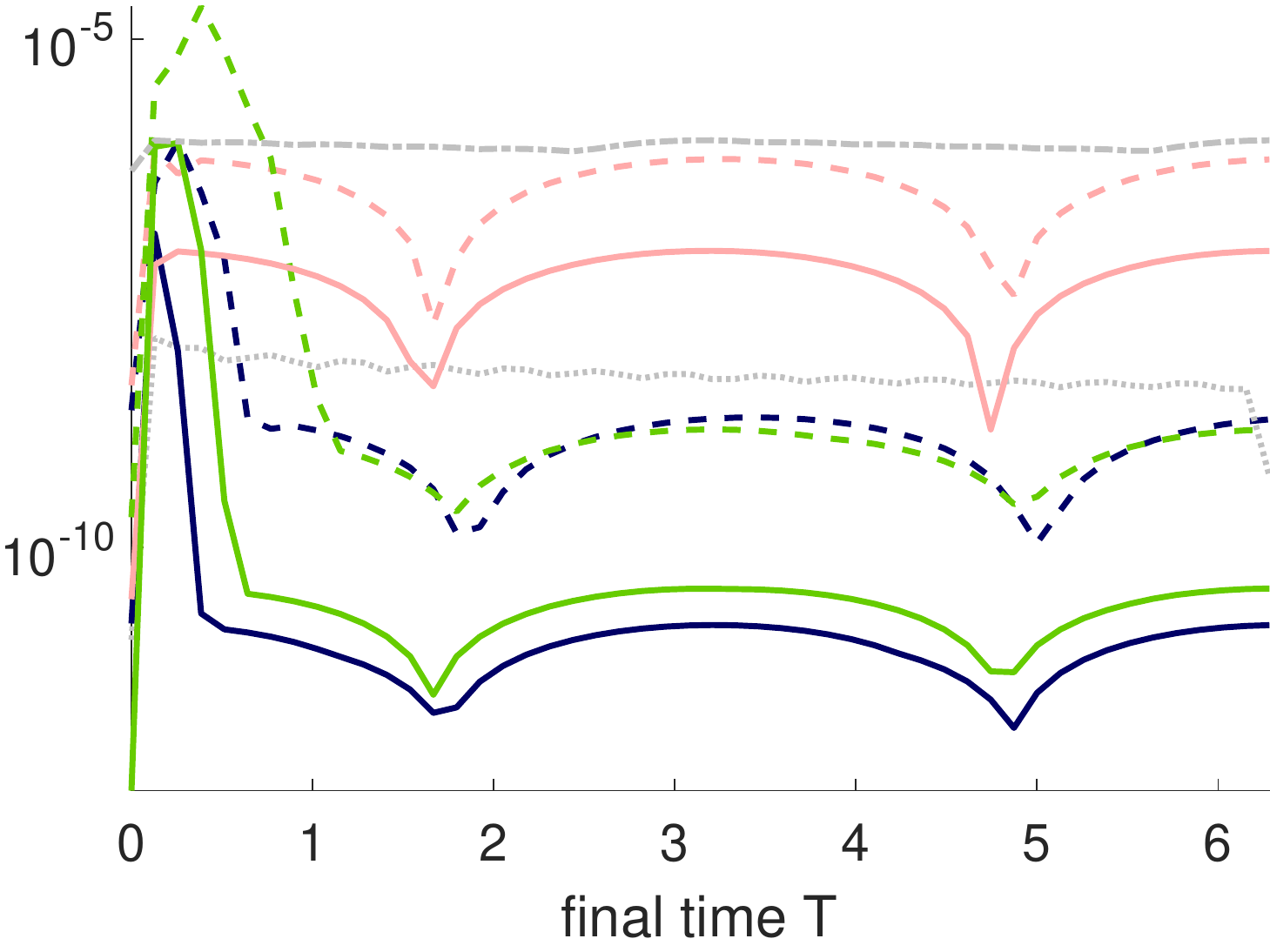}}
\hfill
\subfigure[\rightb, $\dDG=2$]
{\includegraphics[width=\wtwo]{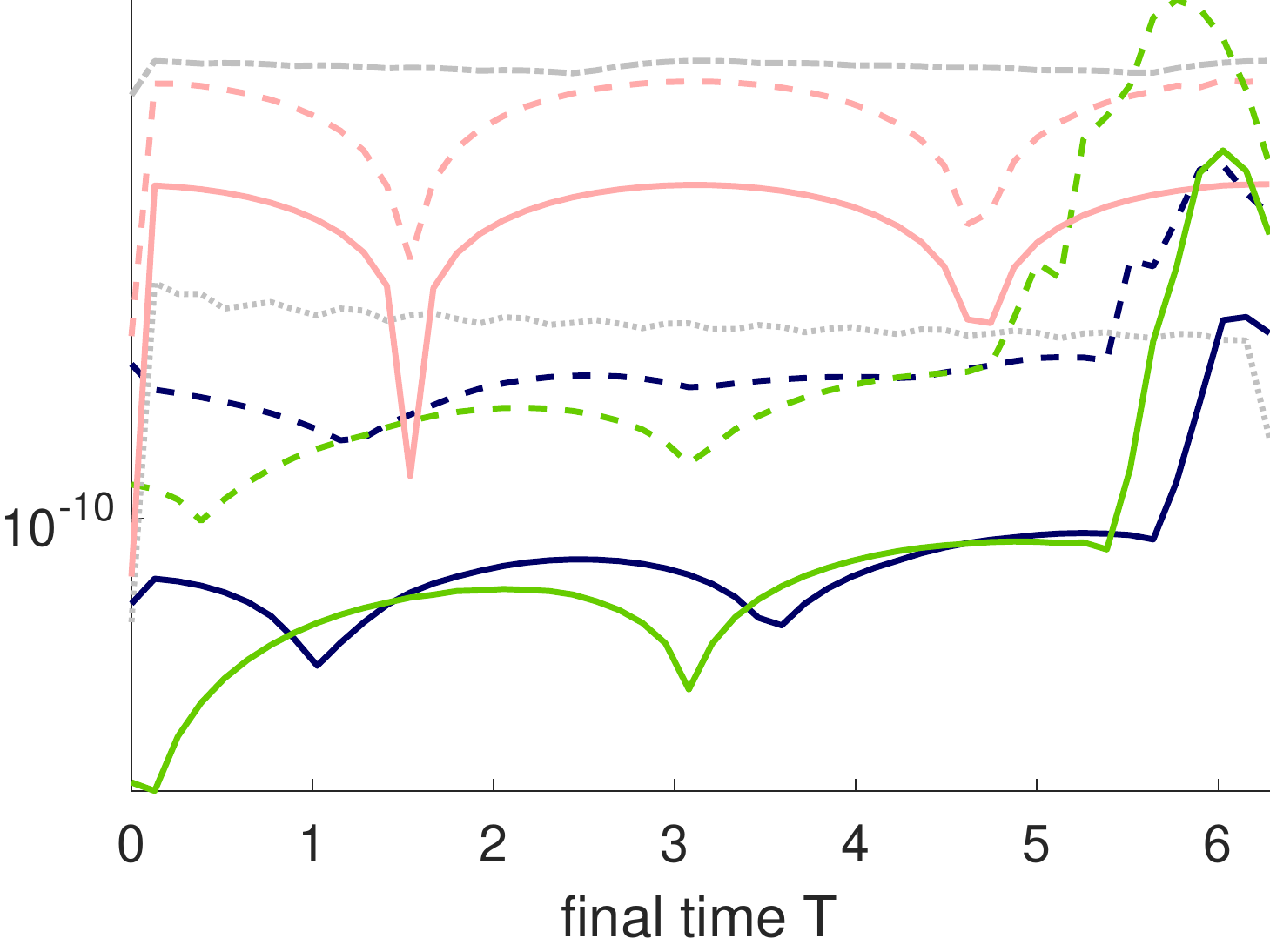}}  

\subfigure[\leftb, $\dDG=3$]
{\includegraphics[width=\wtwo]{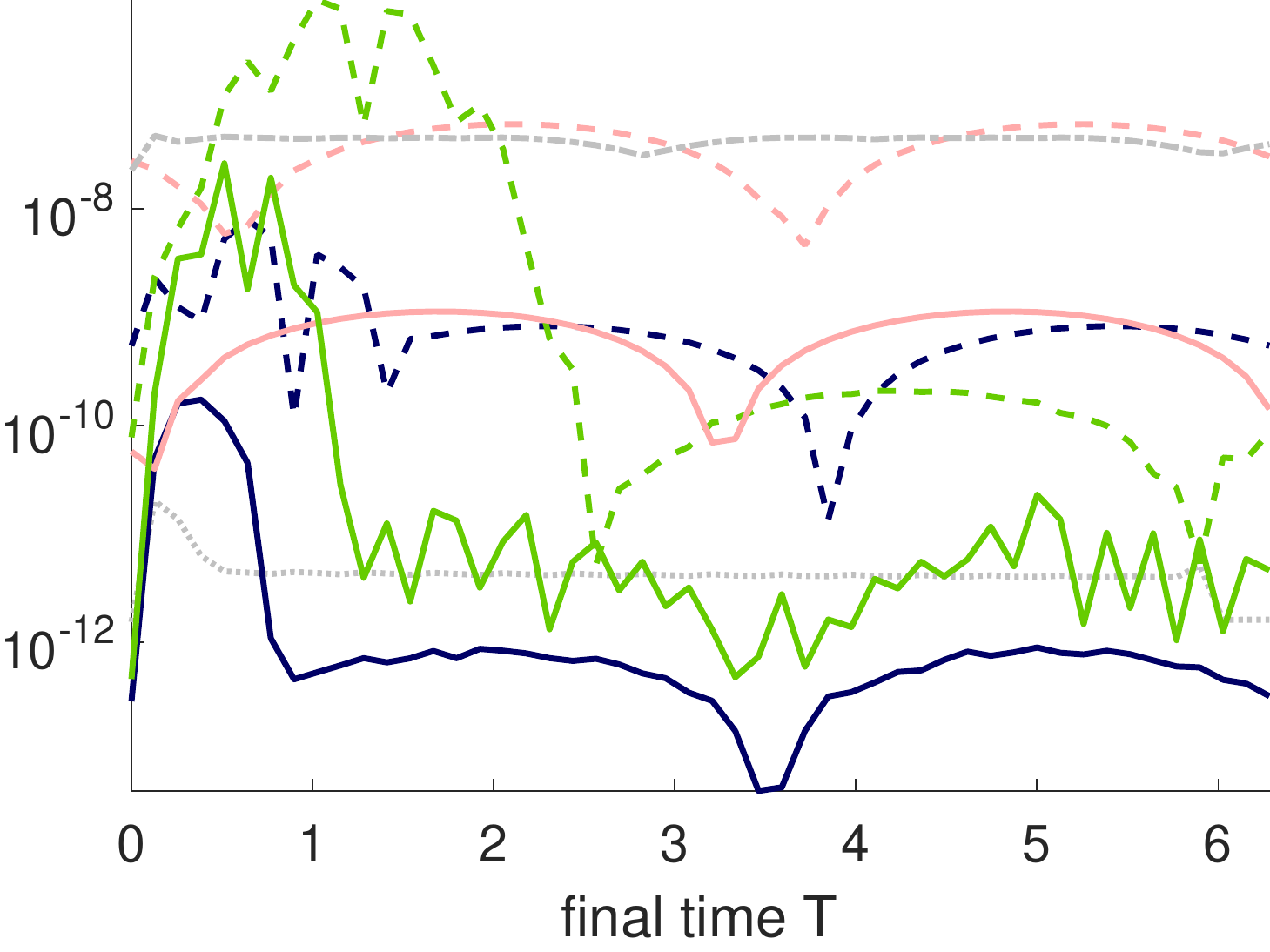}}
\hfill
\subfigure[\rightb; $\dDG=3$]
{\includegraphics[width=\wtwo]{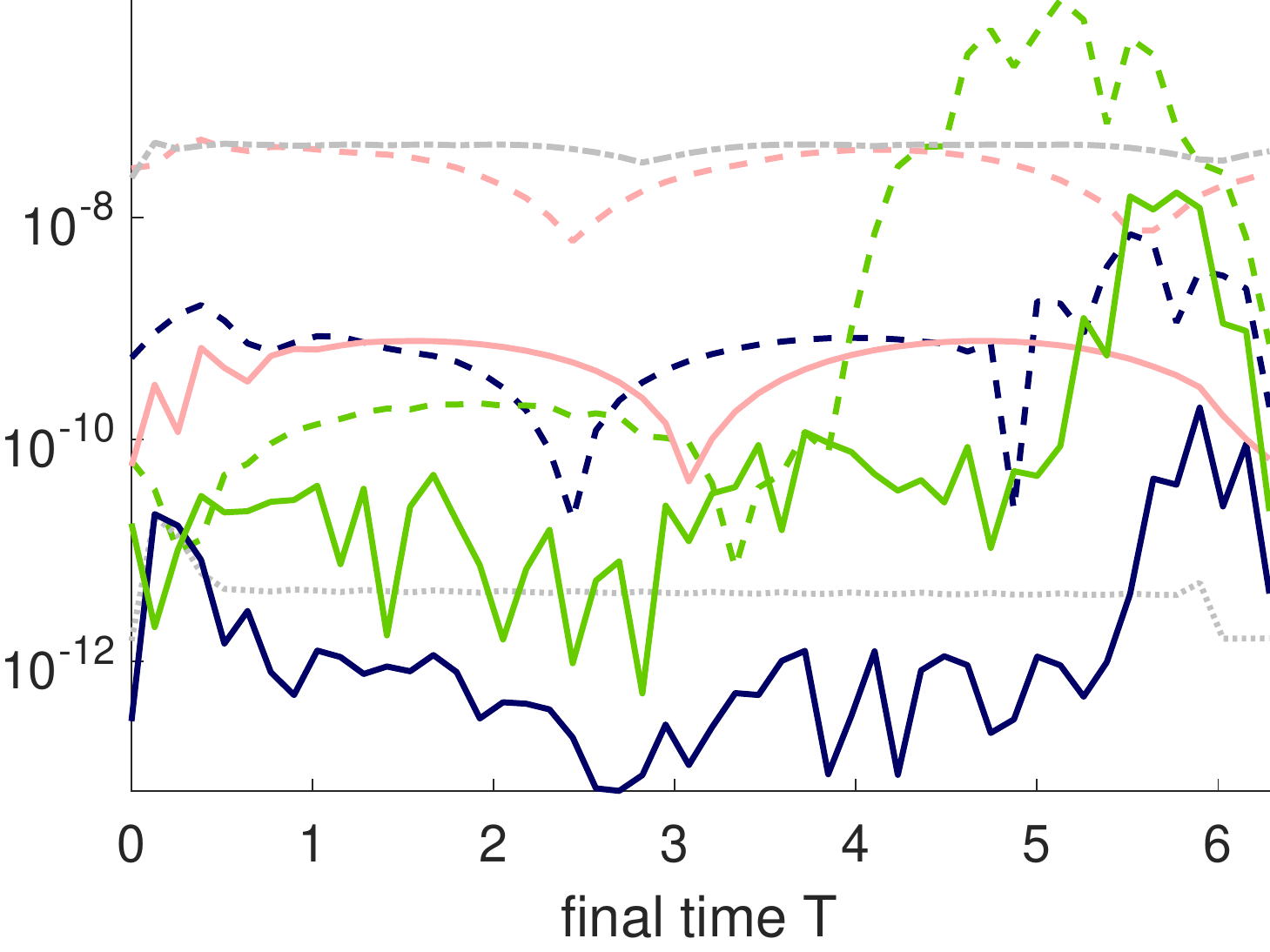}}
\caption{
\probErr{prob:cs-d}{constant}{Dirichlet}
\capMax
}
\label{fig:dbc:pw-errors}
\end{figure}
 
\def\wtwo{.45\textwidth}
\begin{figure}[ht!] 
\centering 
\subfigure[\leftb, $\dDG=1$]
{\includegraphics[width=\wtwo]{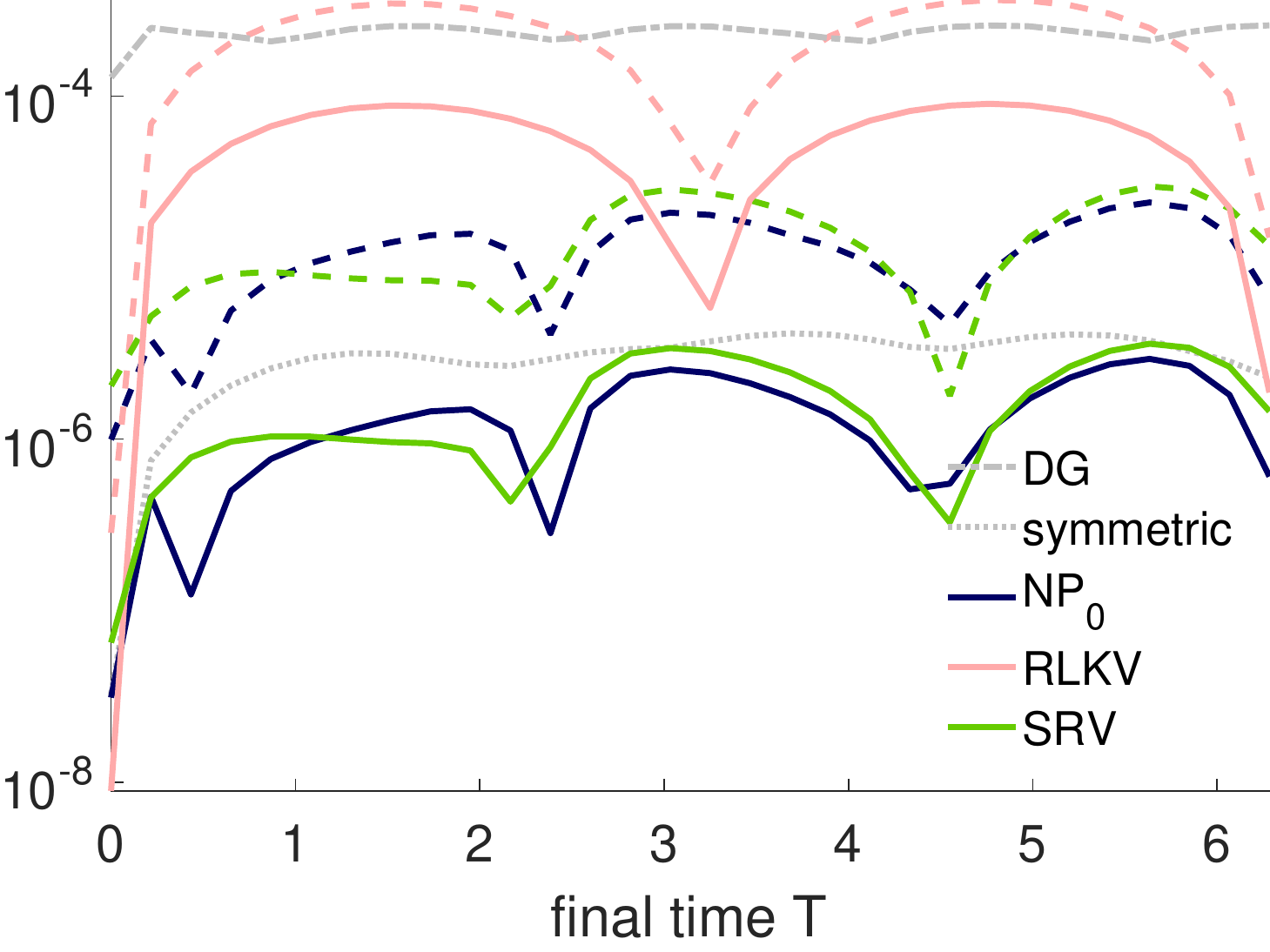}}
\hfill
\subfigure[\rightb, $\dDG=1$]
{\includegraphics[width=\wtwo]{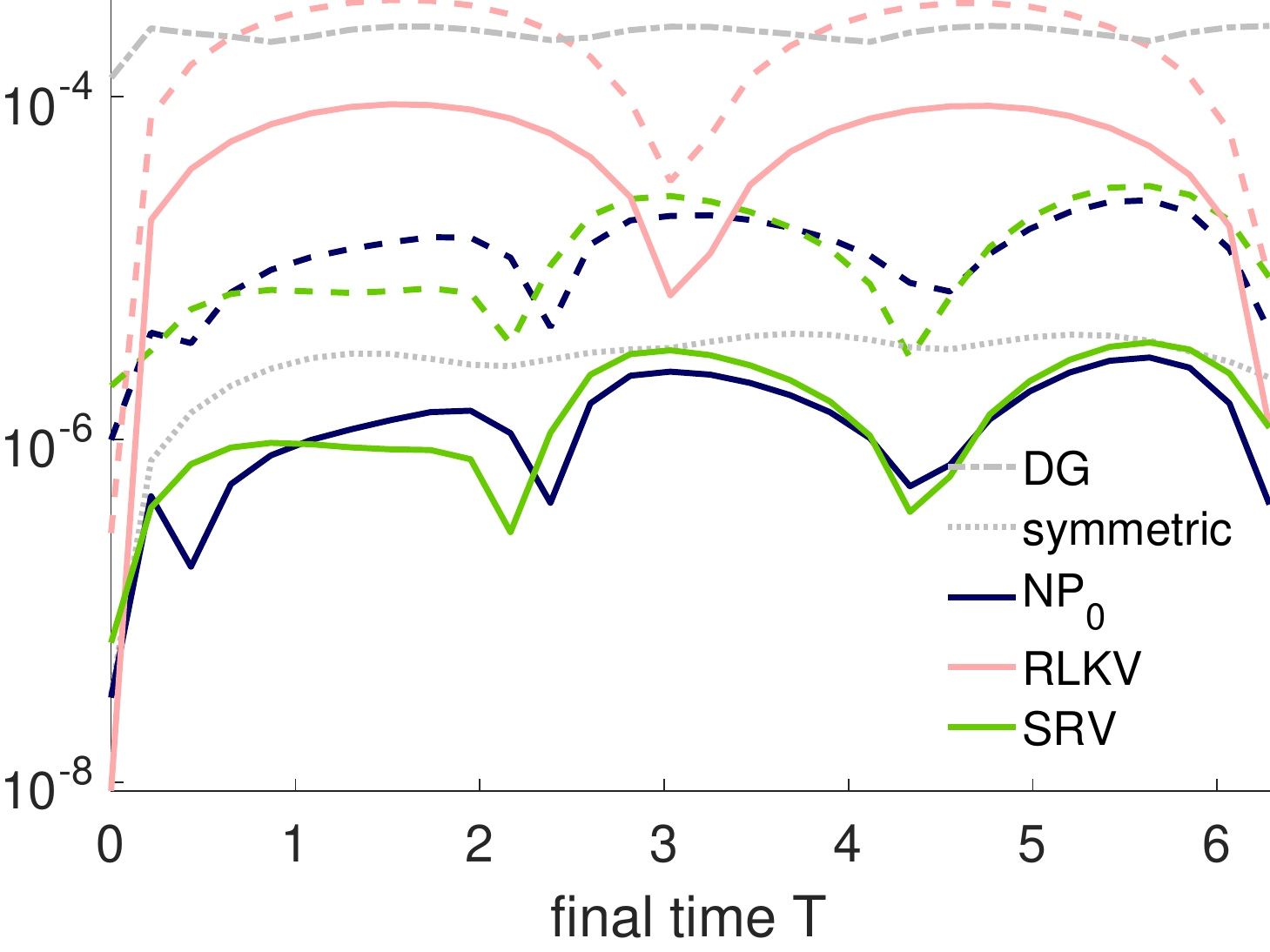}}

\subfigure[\leftb, $\dDG=2$]
{\includegraphics[width=\wtwo]{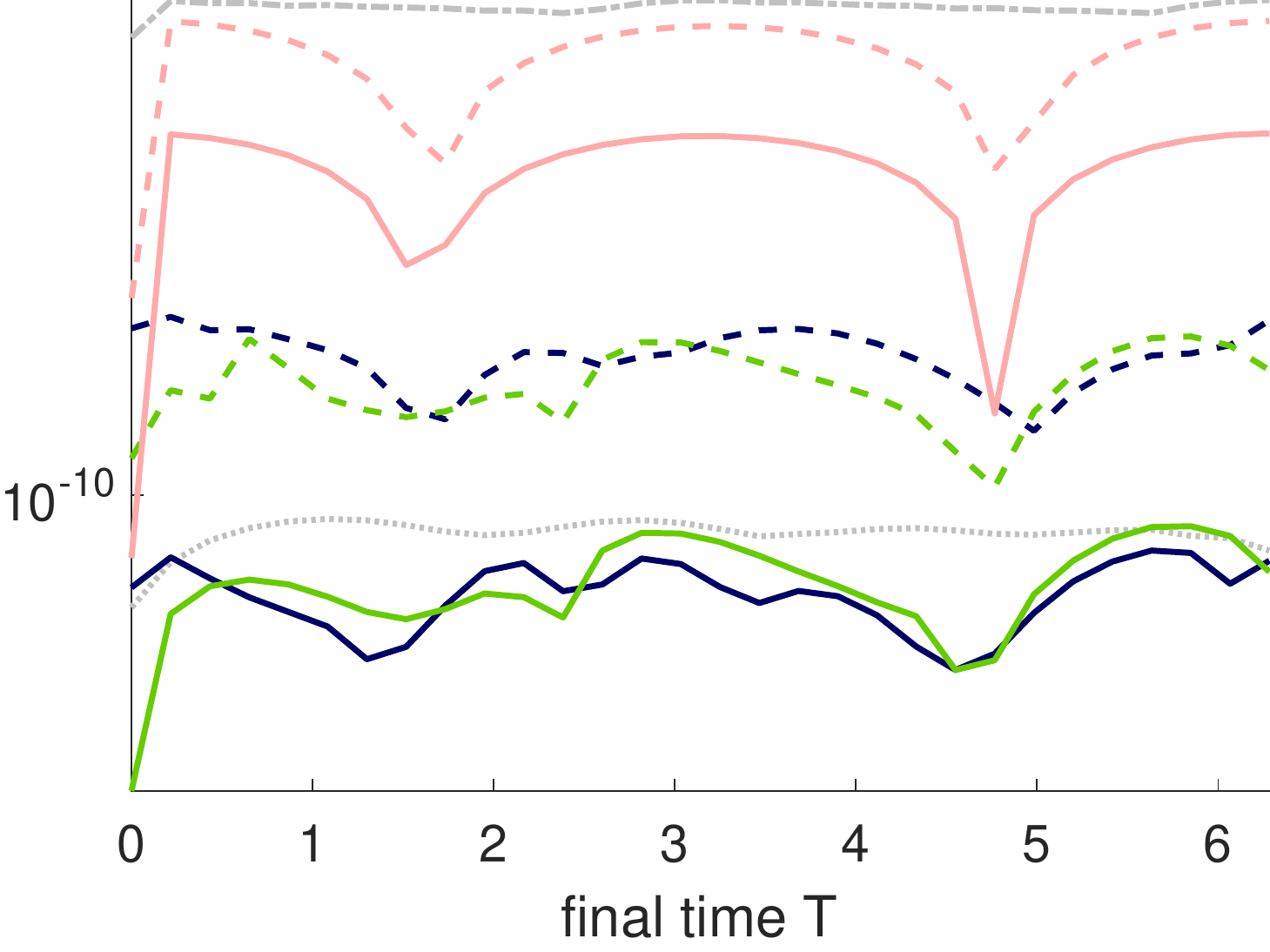}}
\hfill
\subfigure[\rightb, $\dDG=2$]
{\includegraphics[width=\wtwo]{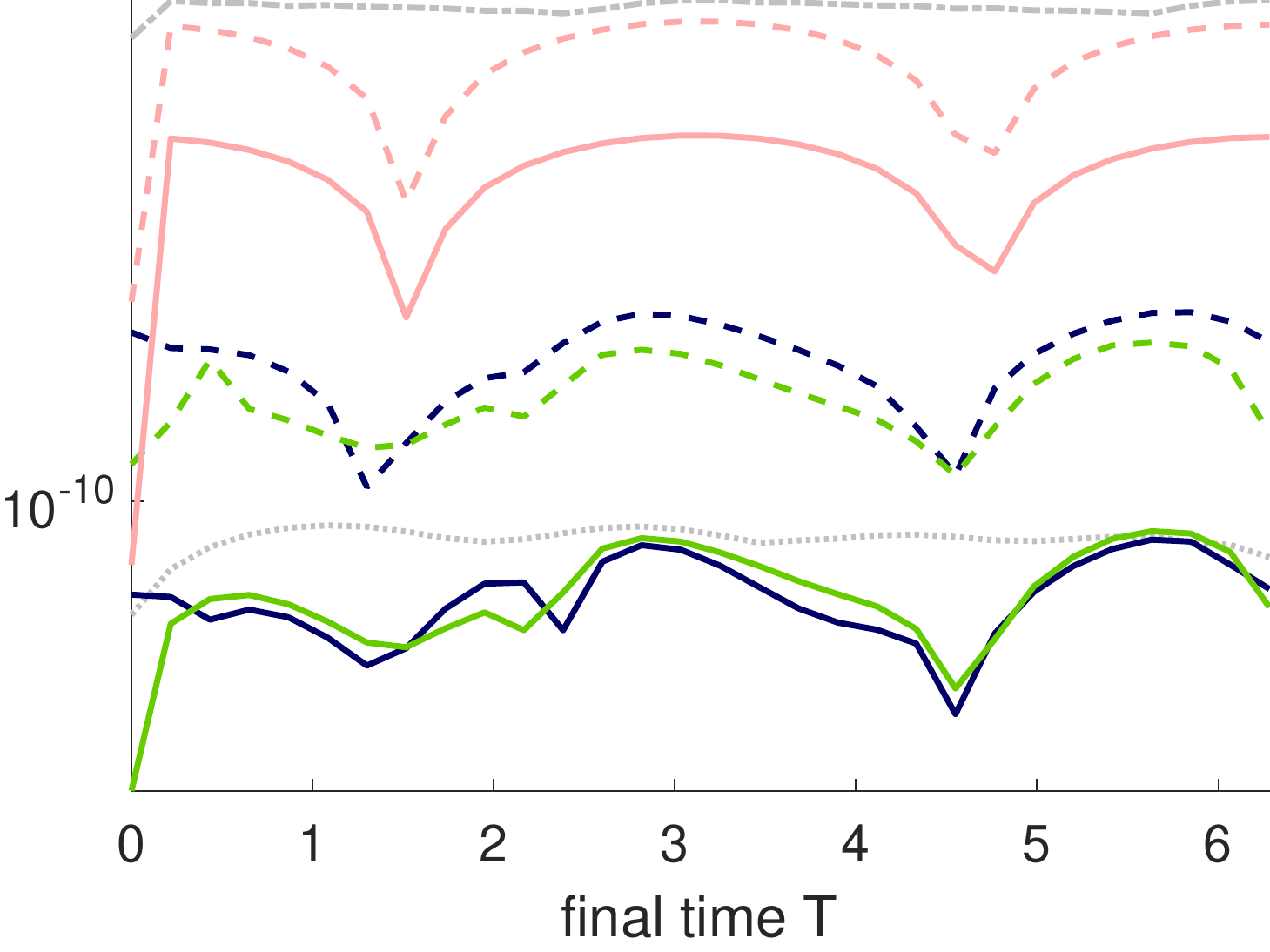}}

\subfigure[\leftb, $\dDG=3$]
{\includegraphics[width=\wtwo]{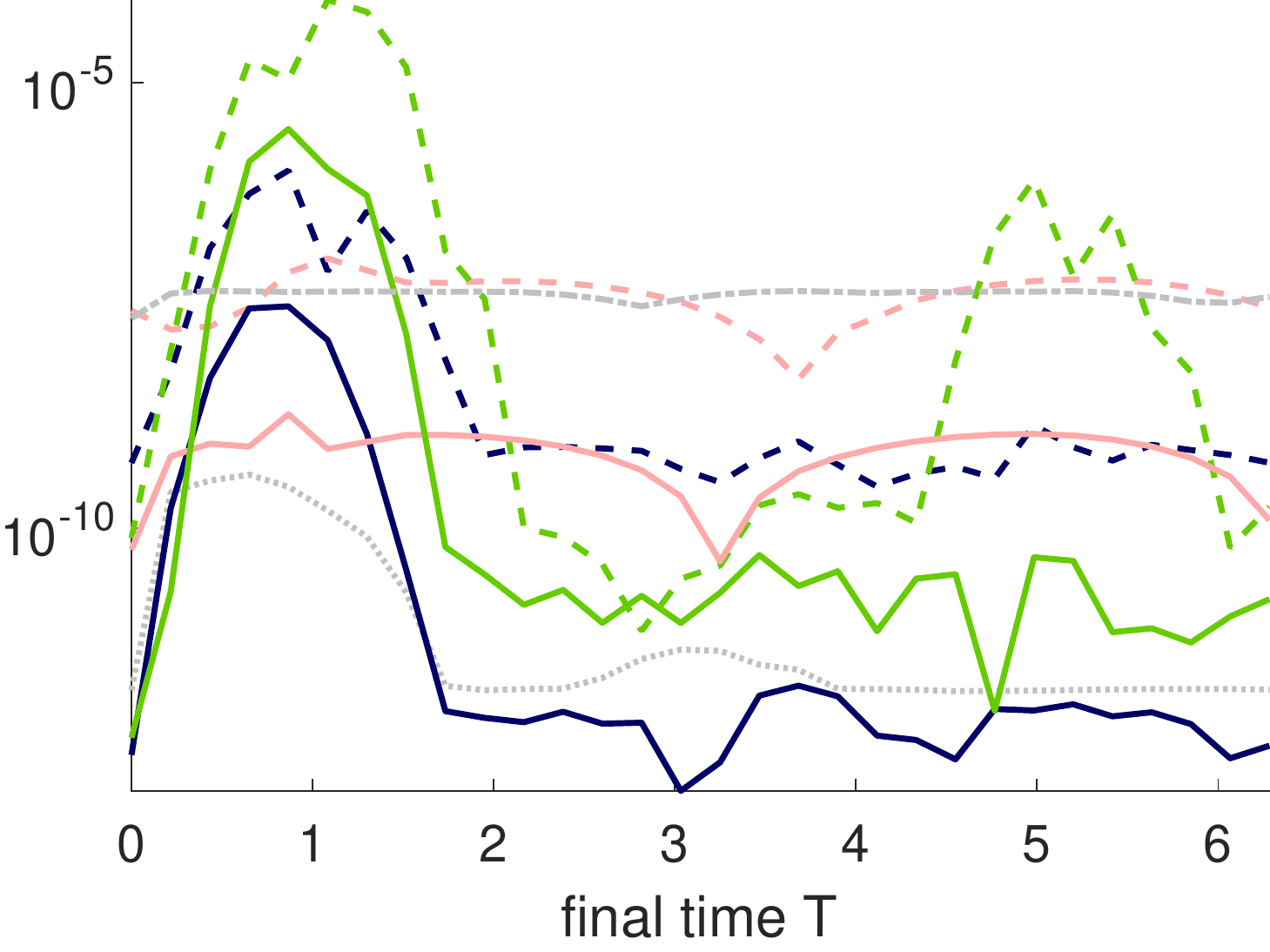}}
\hfill
\subfigure[\rightb; $\dDG=3$]
{\includegraphics[width=\wtwo]{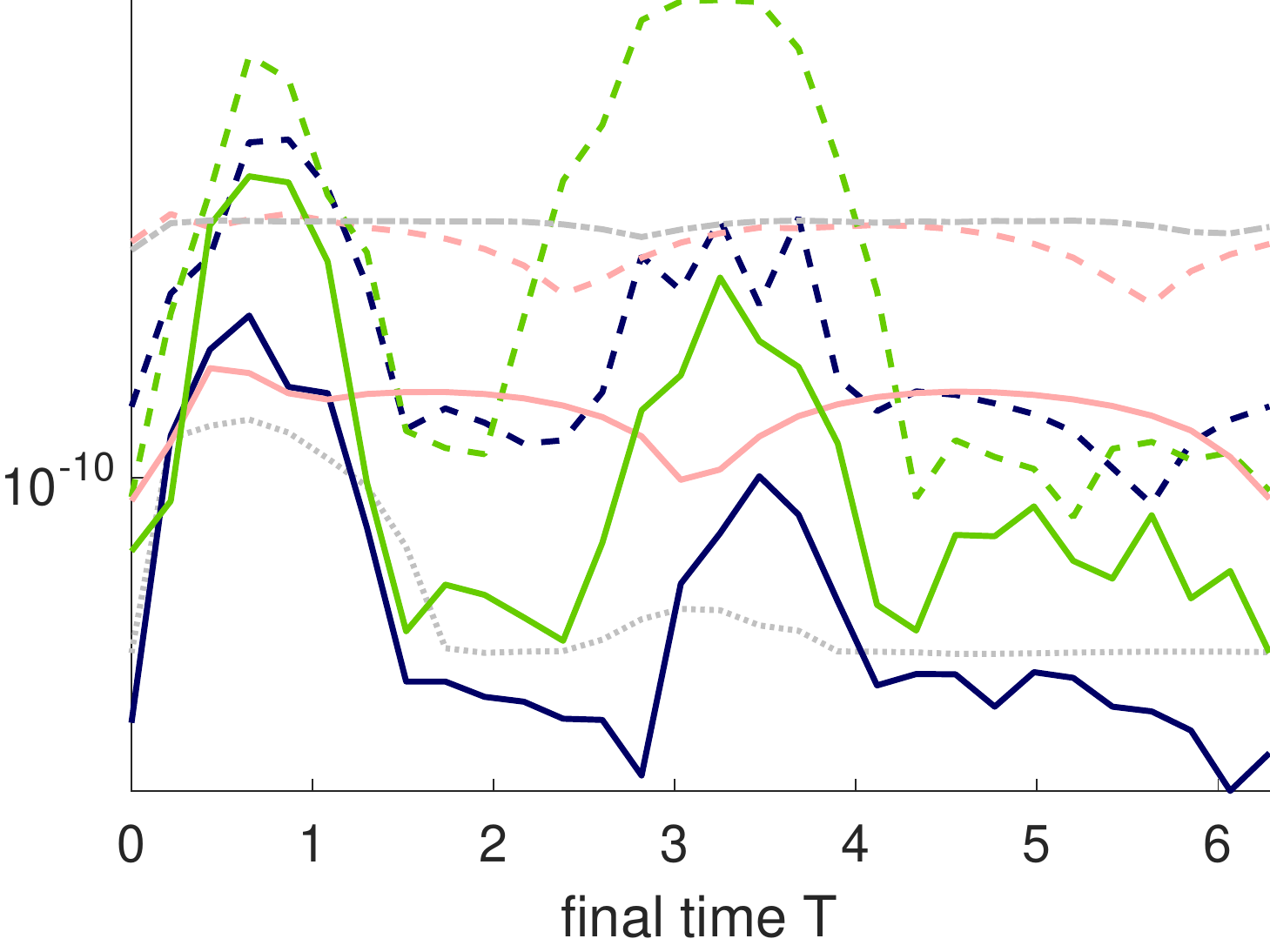}}
\caption{
\probErr{prob:vs-p}{variable}{periodic}
\capMax
}
\label{fig:xt-speed:pw-errors}
\end{figure}


\def\wtwo{.23\textwidth}
\def\mskp{.15\textwidth}
\newcommand{\mylabel}{
$L^\infty$: left \hskip\mskp $L^2$: left \hskip\mskp $L^2$: right \hskip\mskp
$L^\infty$: right}
 
\begin{figure}[ht!]
\centering 
1
\includegraphics[width=\wtwo]{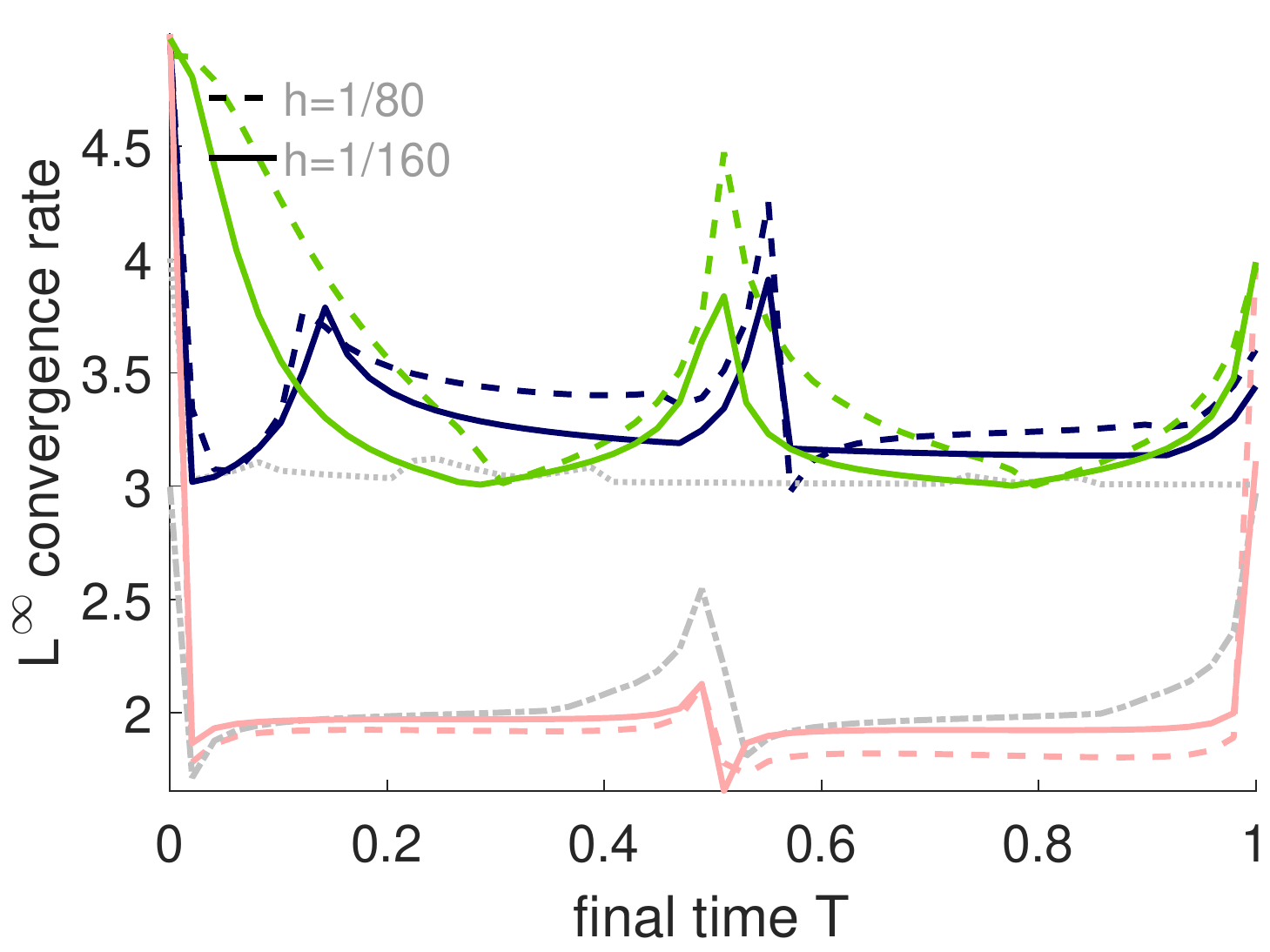}
\includegraphics[width=\wtwo]{figs/pbcd1t1__L2errLeft}
\includegraphics[width=\wtwo]{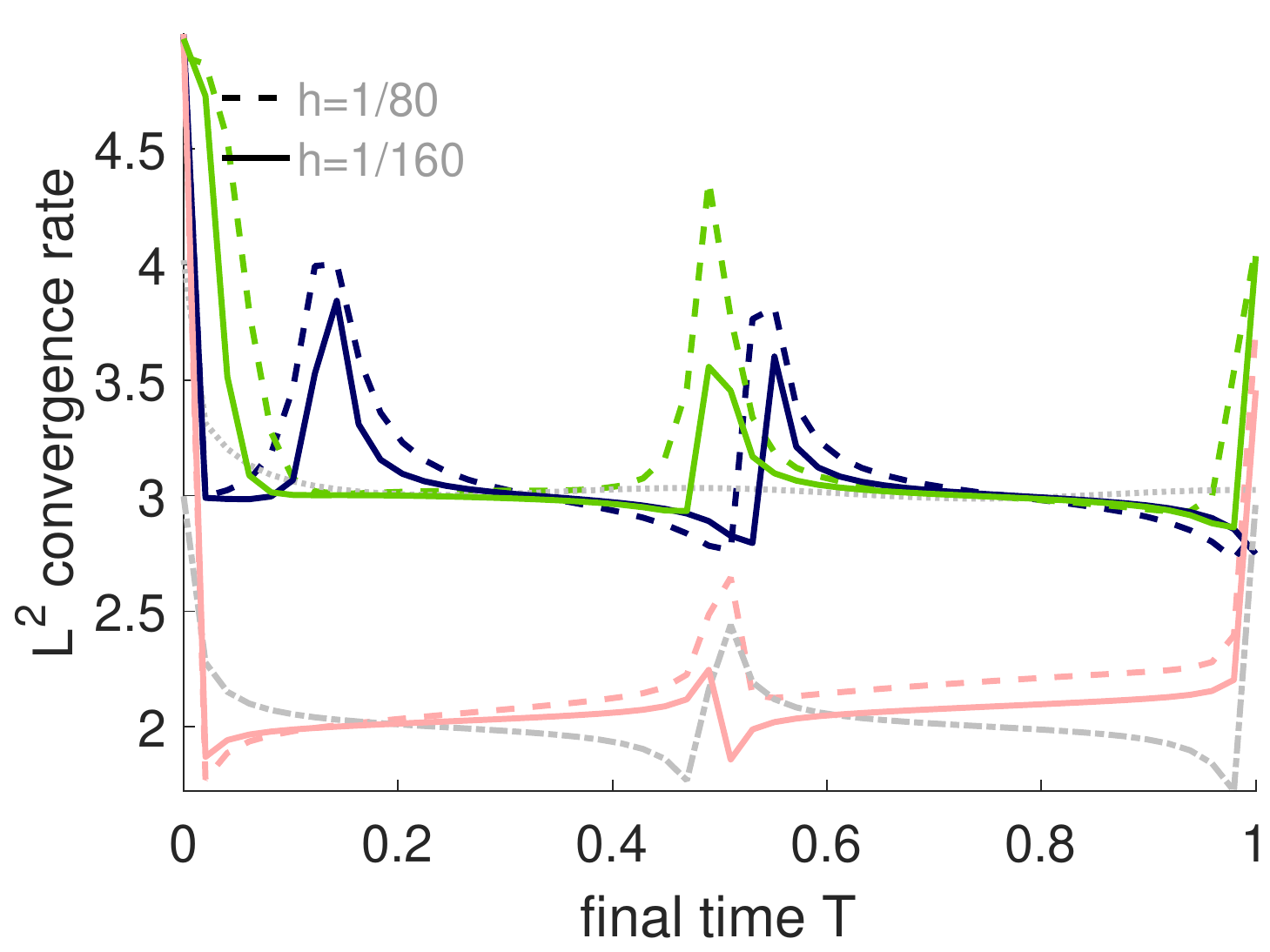}
\includegraphics[width=\wtwo]{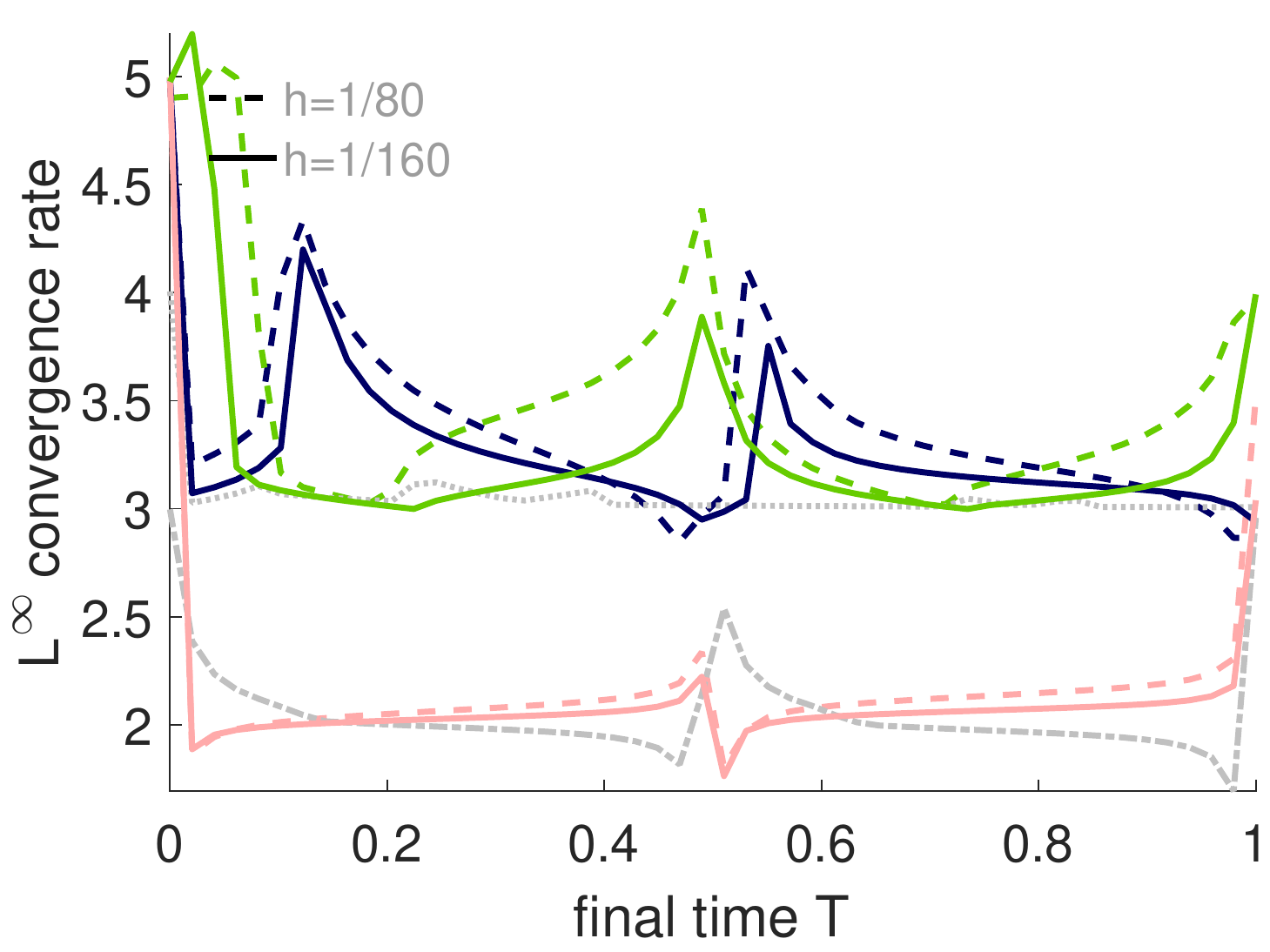}
\\
2
\includegraphics[width=\wtwo]{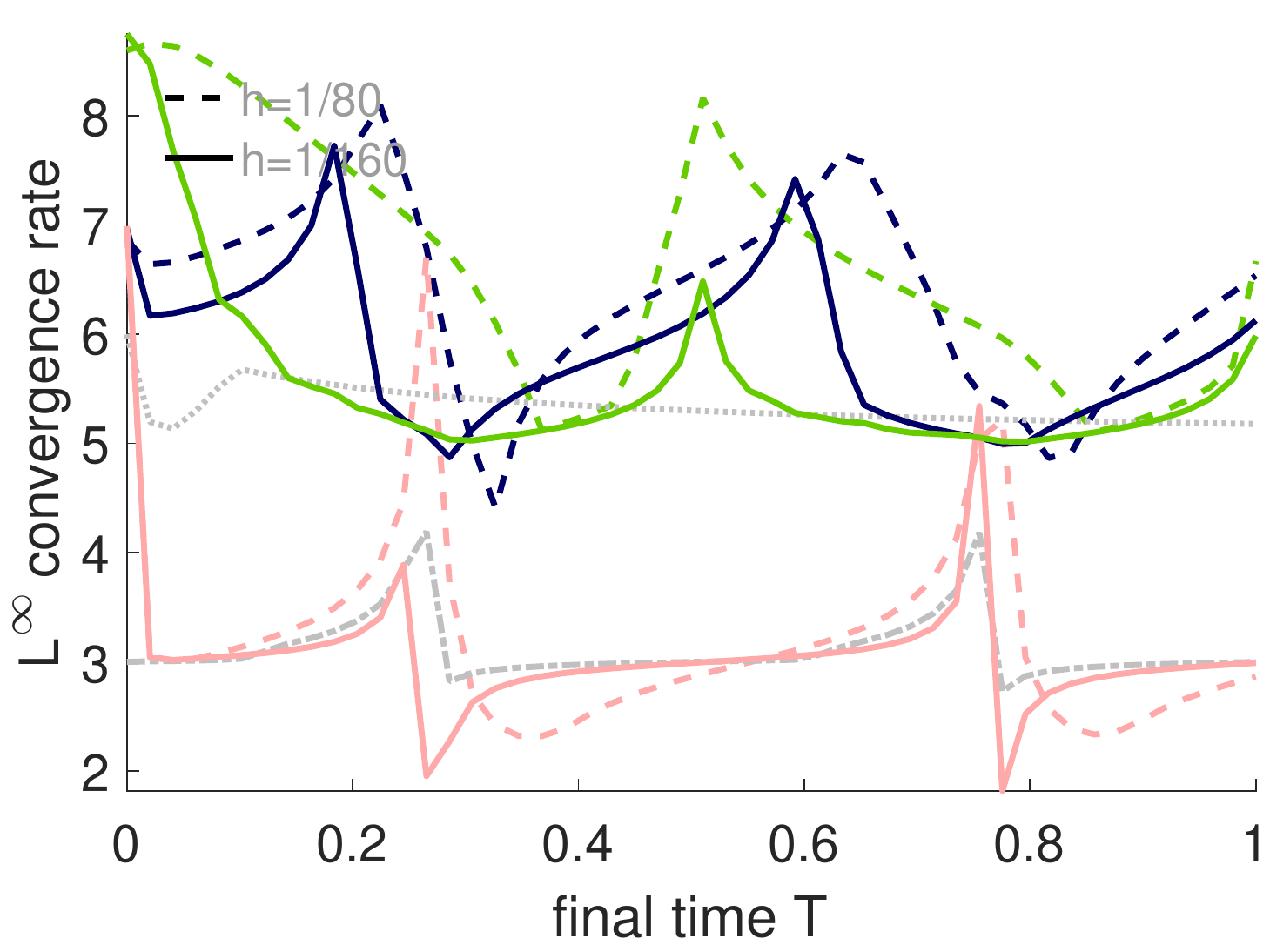}
\includegraphics[width=\wtwo]{figs/pbcd2t1__L2errLeft}
\includegraphics[width=\wtwo]{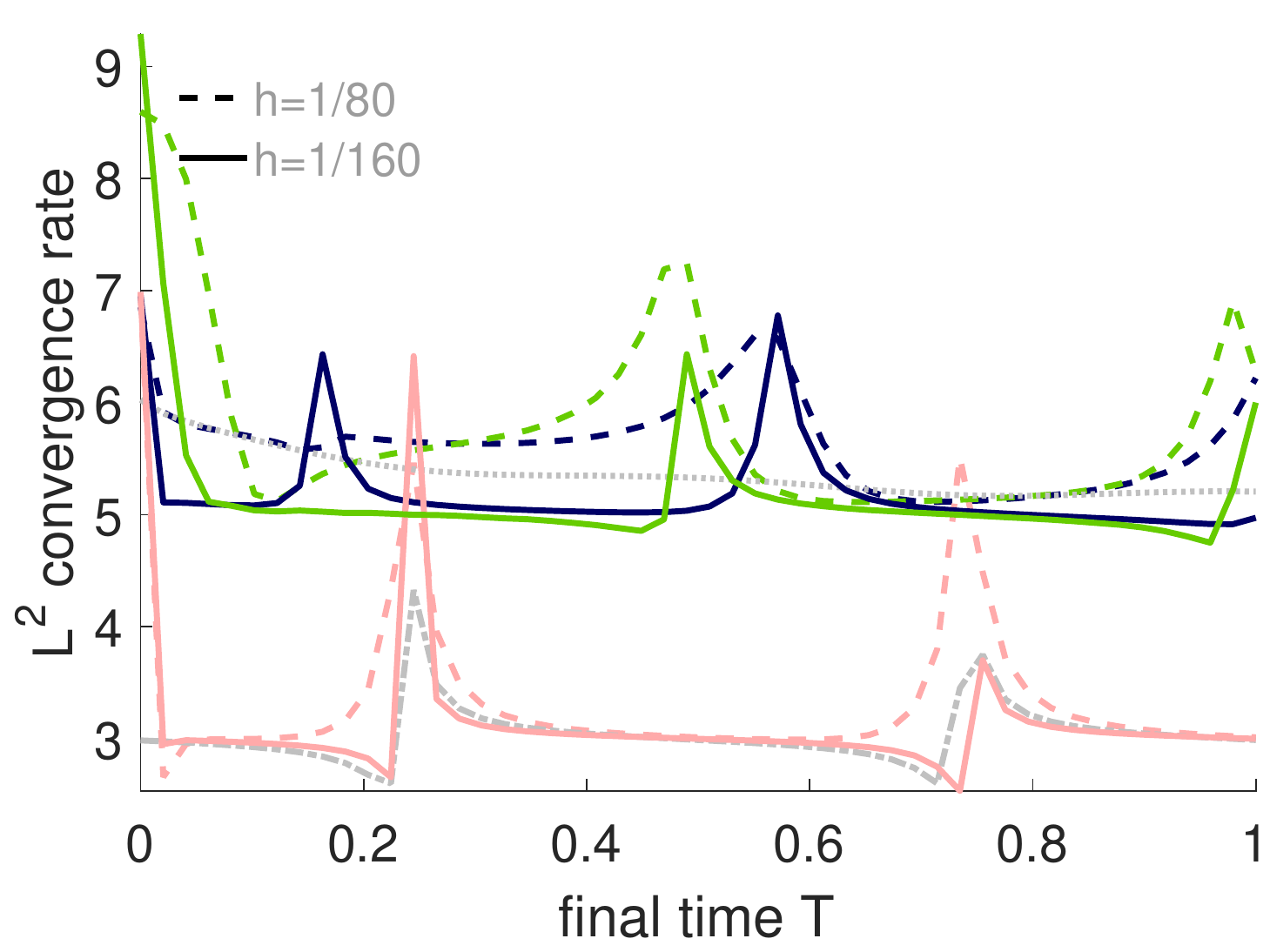}
\includegraphics[width=\wtwo]{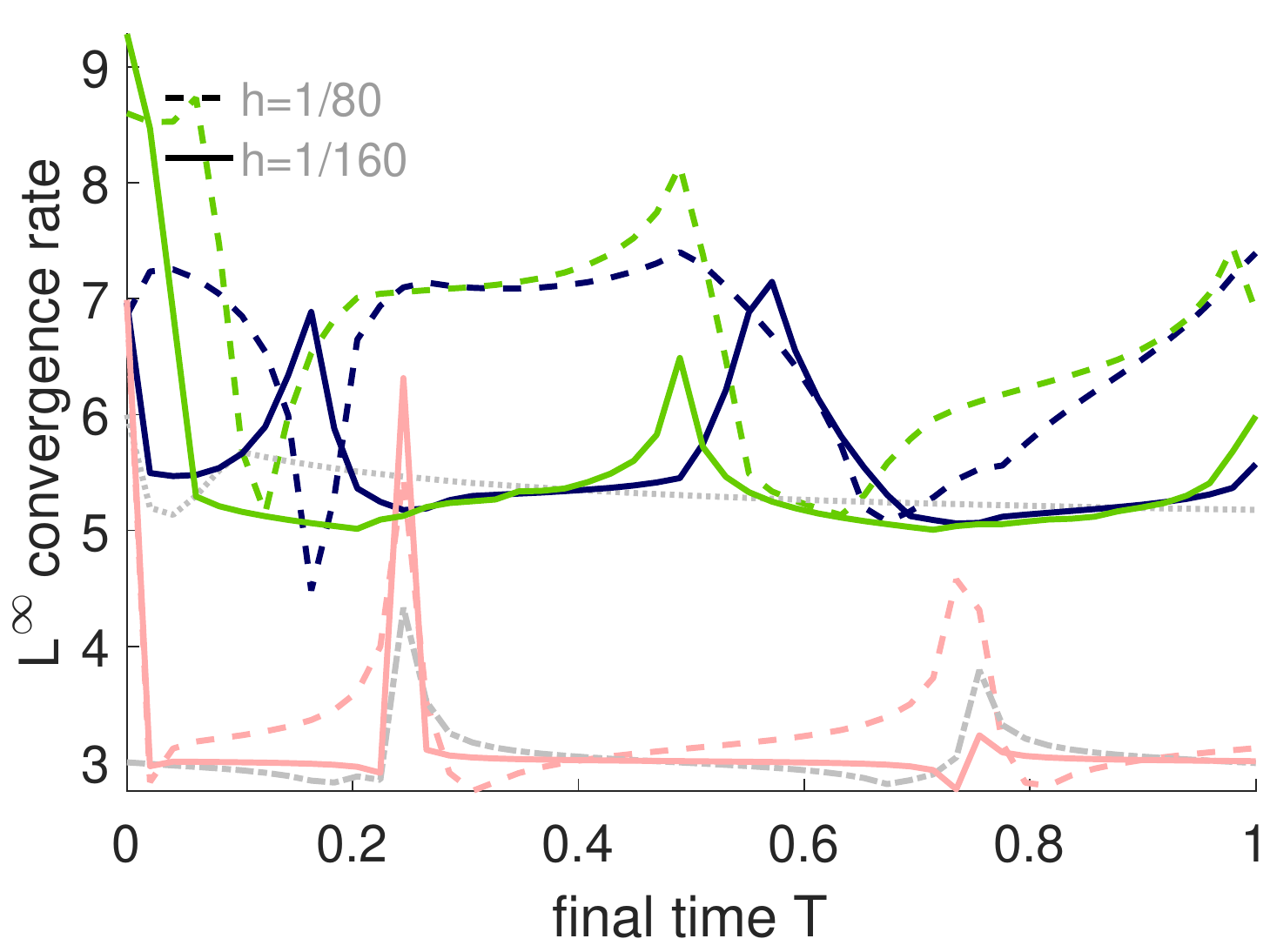} 
\\
3
\includegraphics[width=\wtwo]{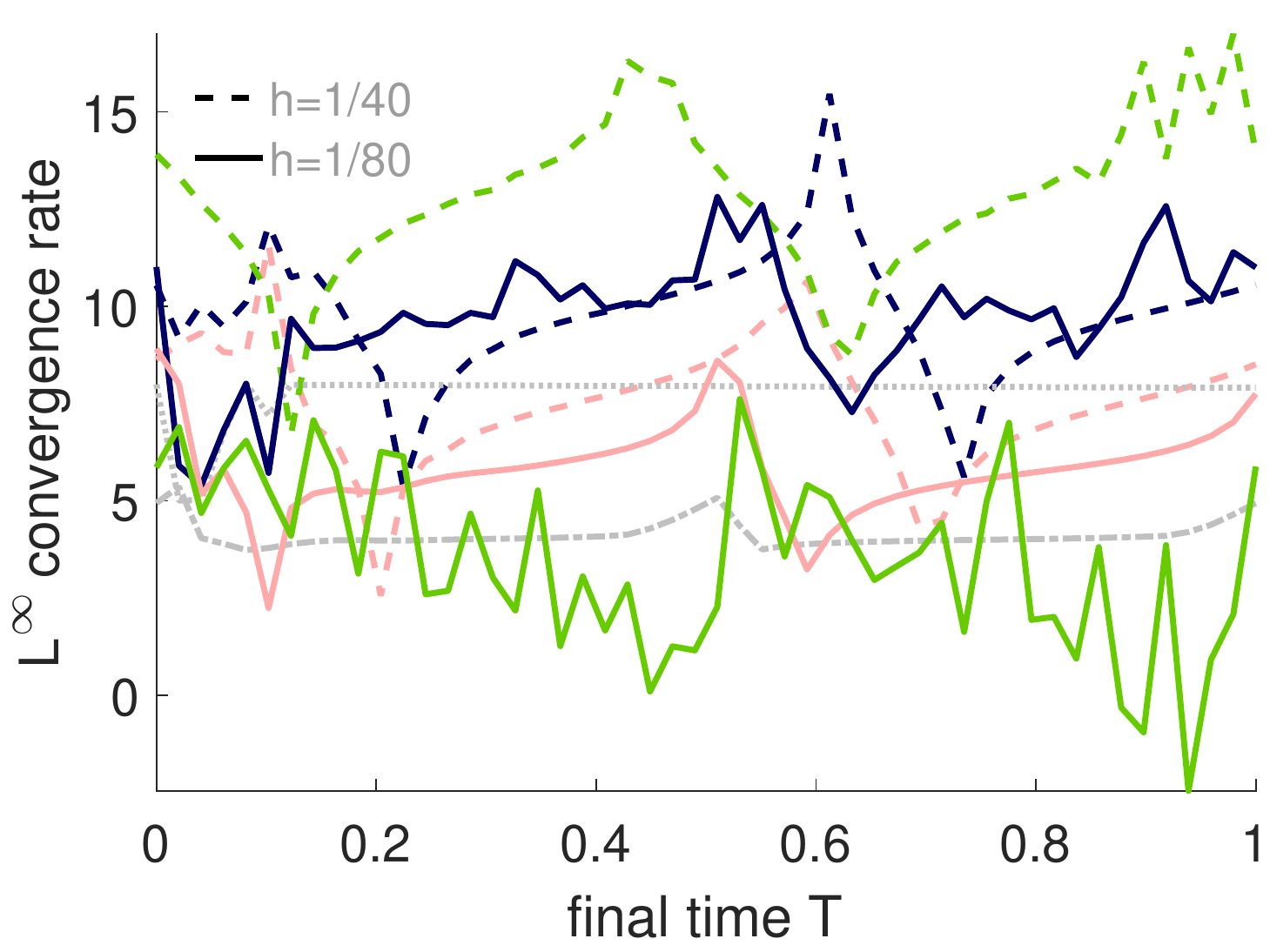}
\includegraphics[width=\wtwo]{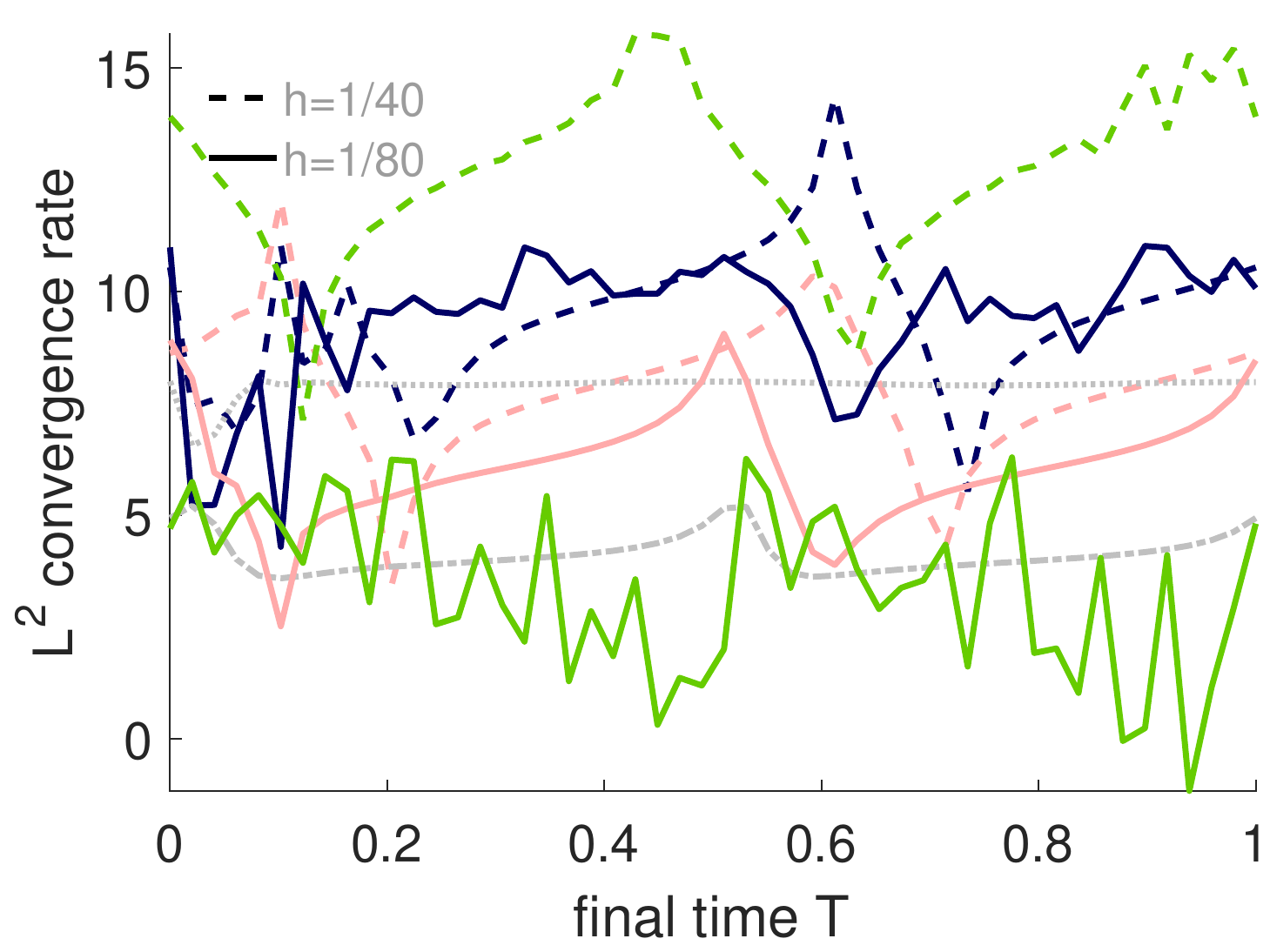}
\includegraphics[width=\wtwo]{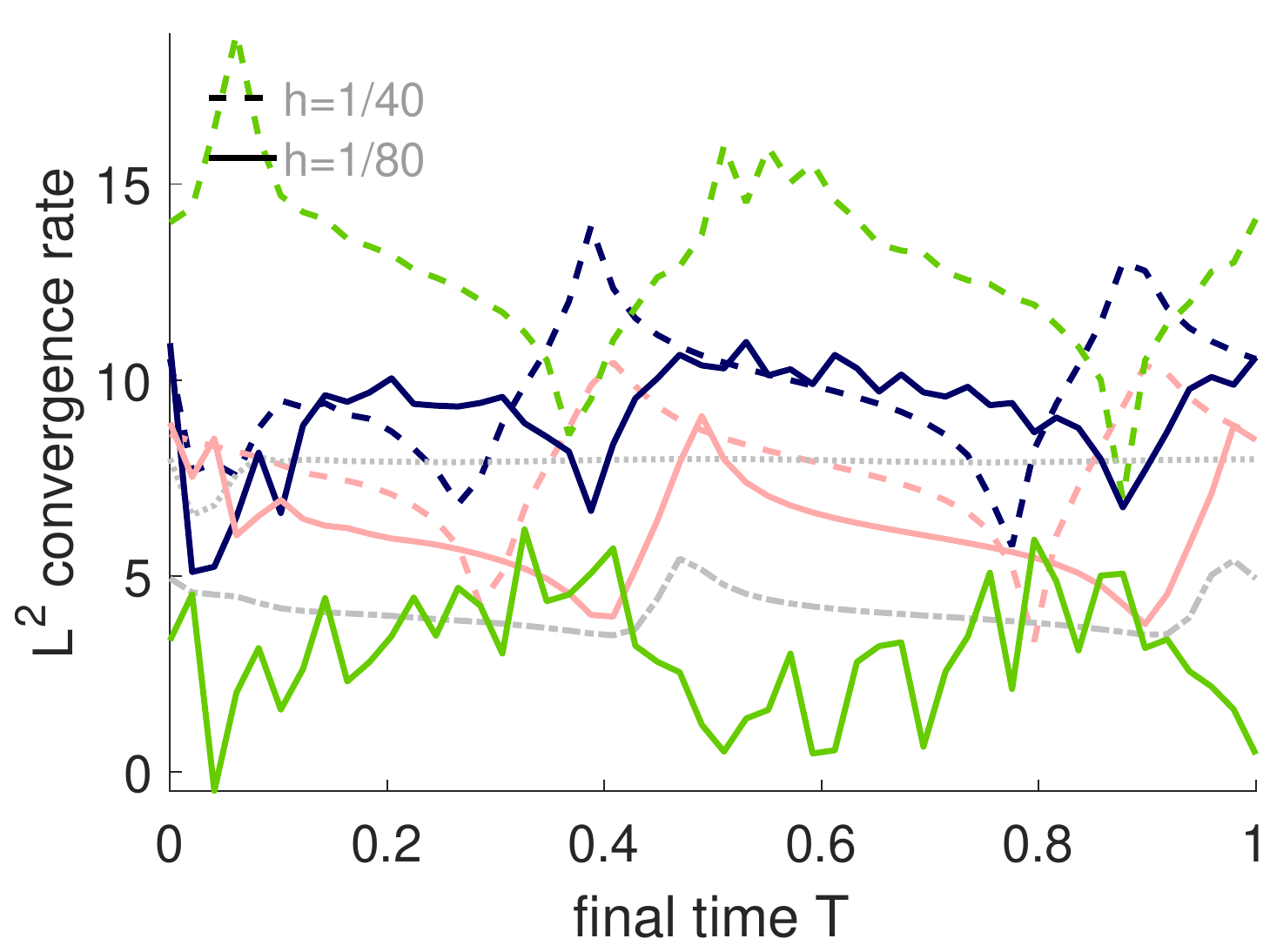}
\includegraphics[width=\wtwo]{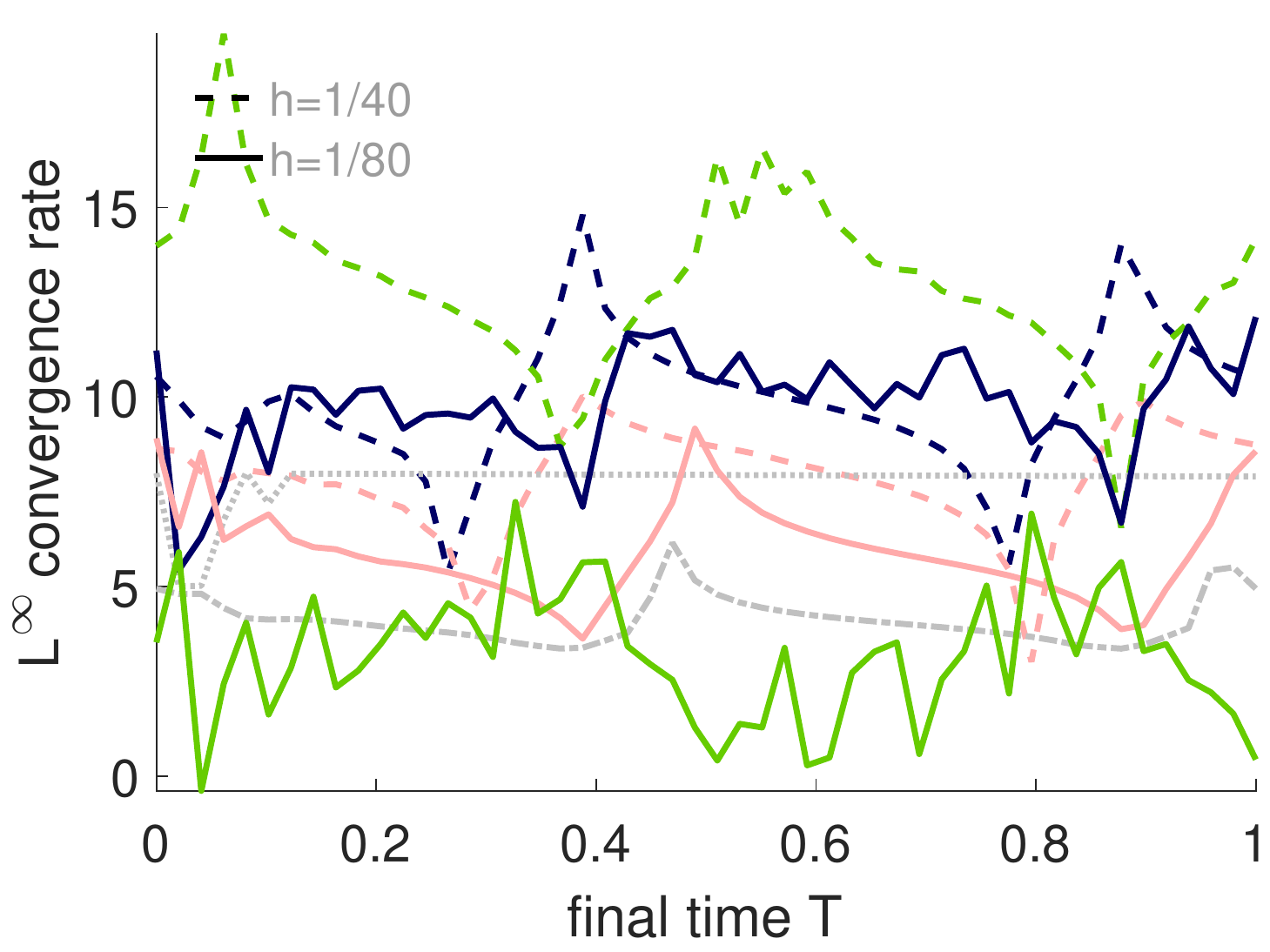} 
\mylabel
\caption{
\probRef{prob:cs-p}{constant}{periodic}
\capRateAll{1,2,3}
The graphs have the same indicative colors and styles as those in 
\figref{fig:pbc:pw-errors}, \ref{fig:dbc:pw-errors} and
\ref{fig:xt-speed:pw-errors}.
}
\label{fig:pbc} 
\end{figure}


\begin{figure}[ht!]
\centering 
1 \includegraphics[width=\wtwo]{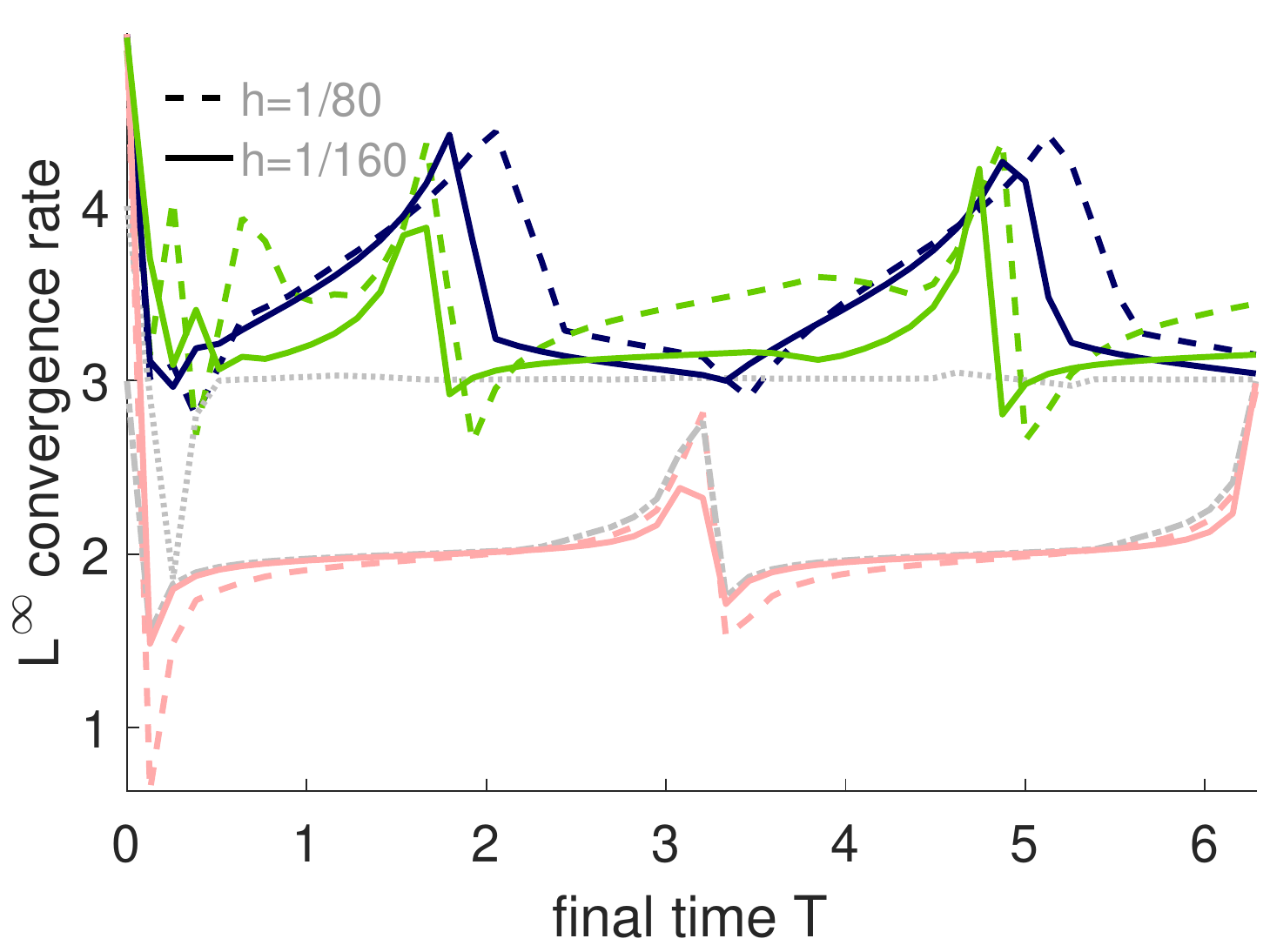}
\includegraphics[width=\wtwo]{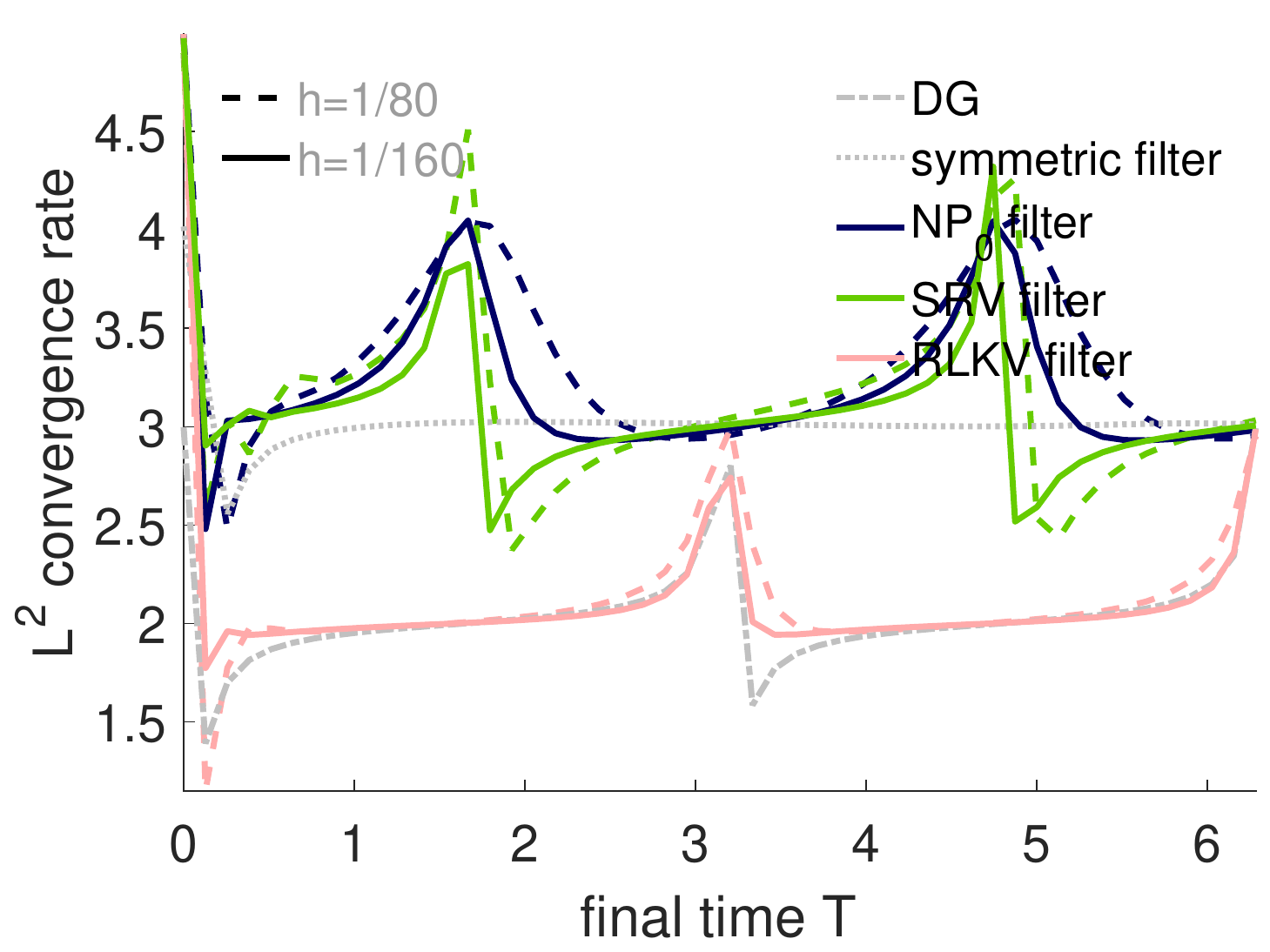}
\includegraphics[width=\wtwo]{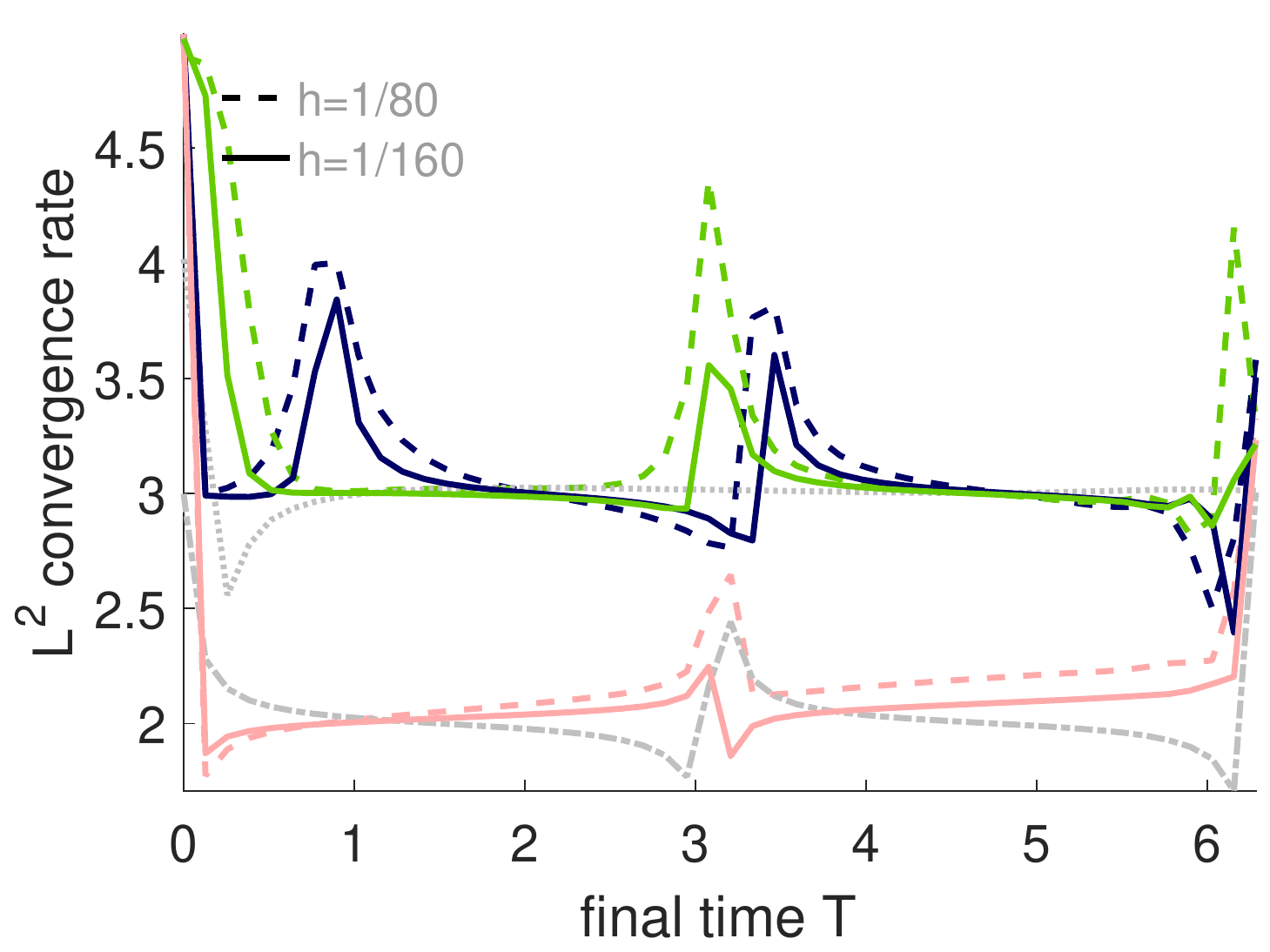}
\includegraphics[width=\wtwo]{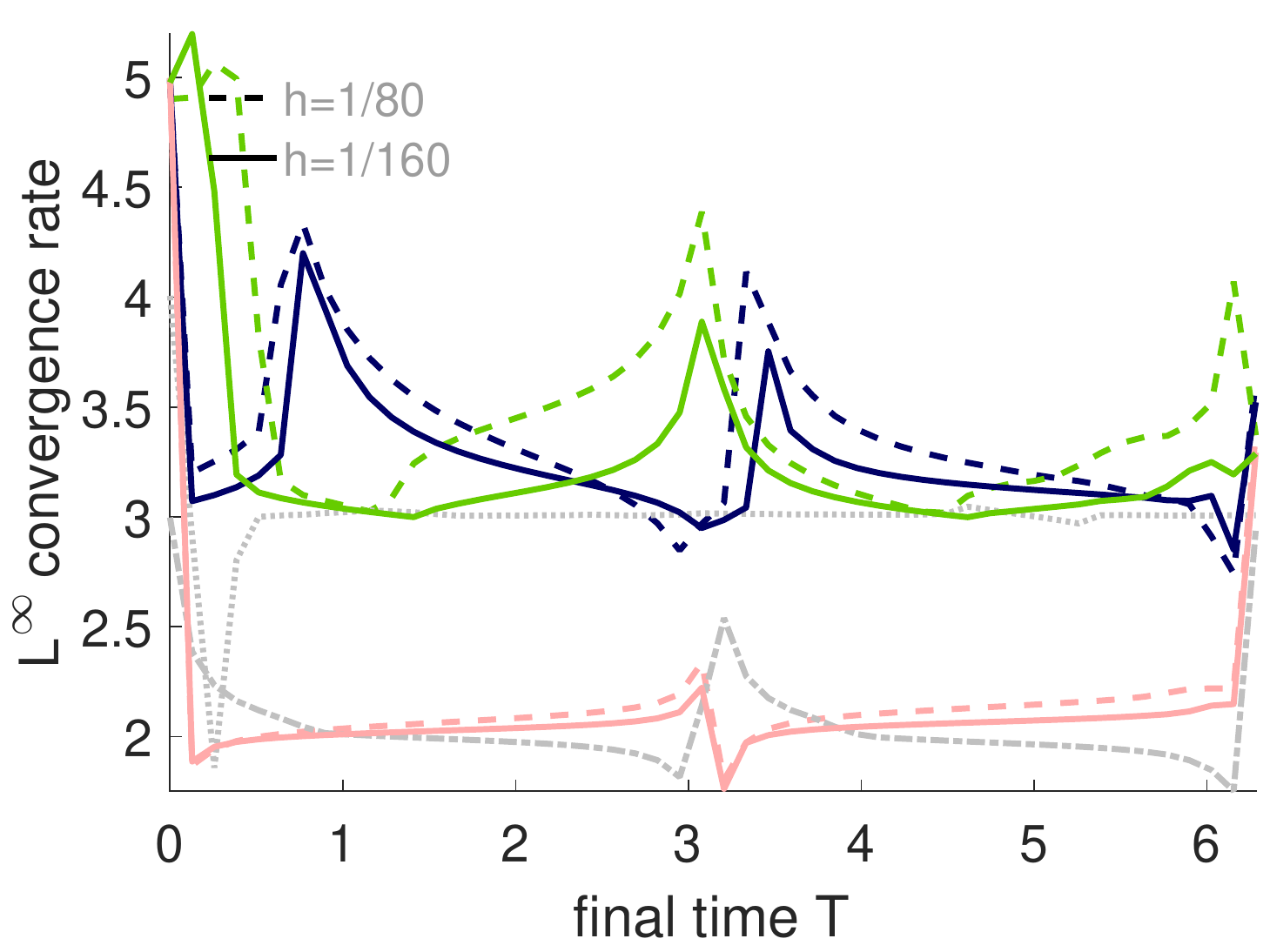}
\\
2 \includegraphics[width=\wtwo]{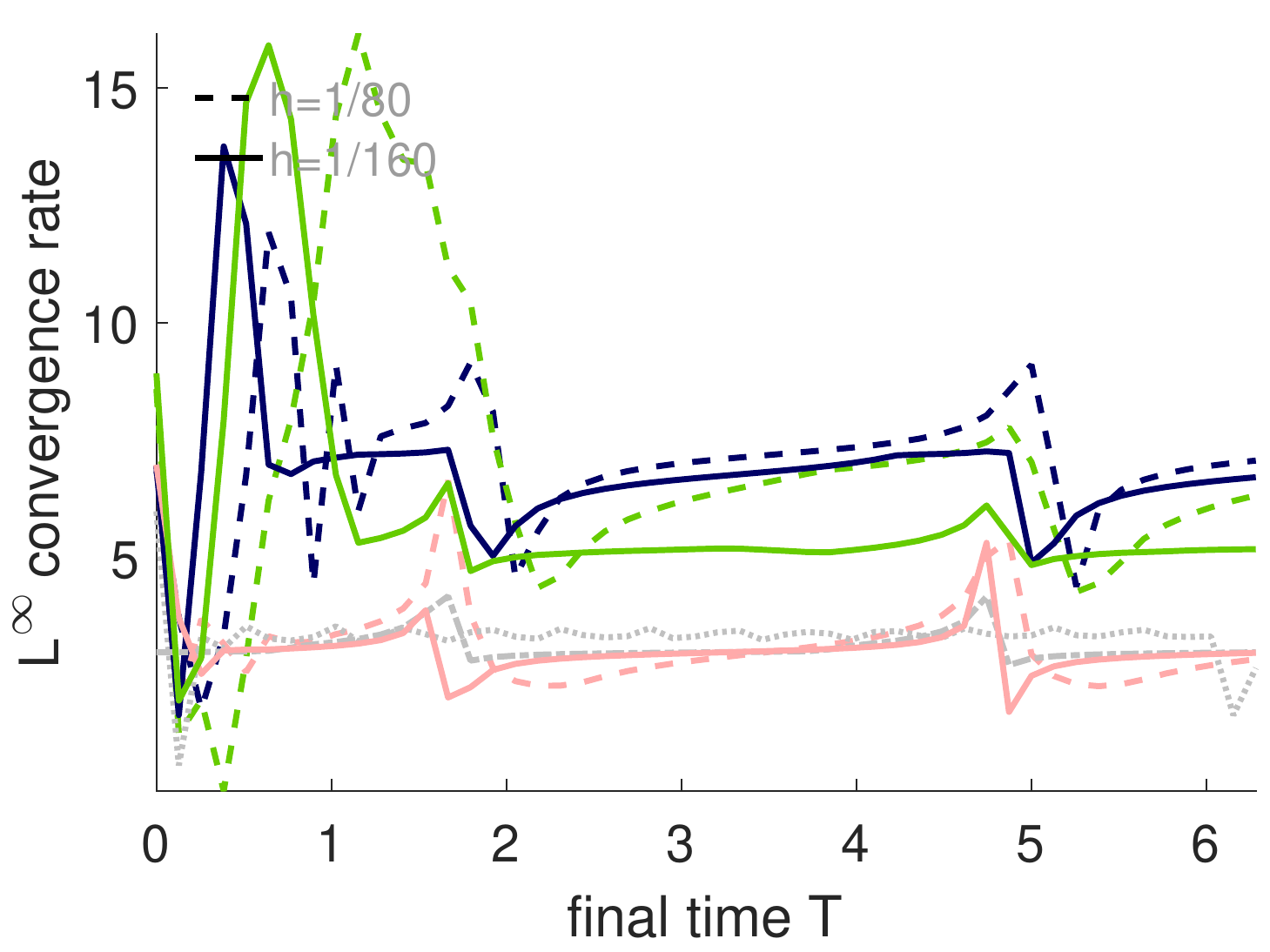}
\includegraphics[width=\wtwo]{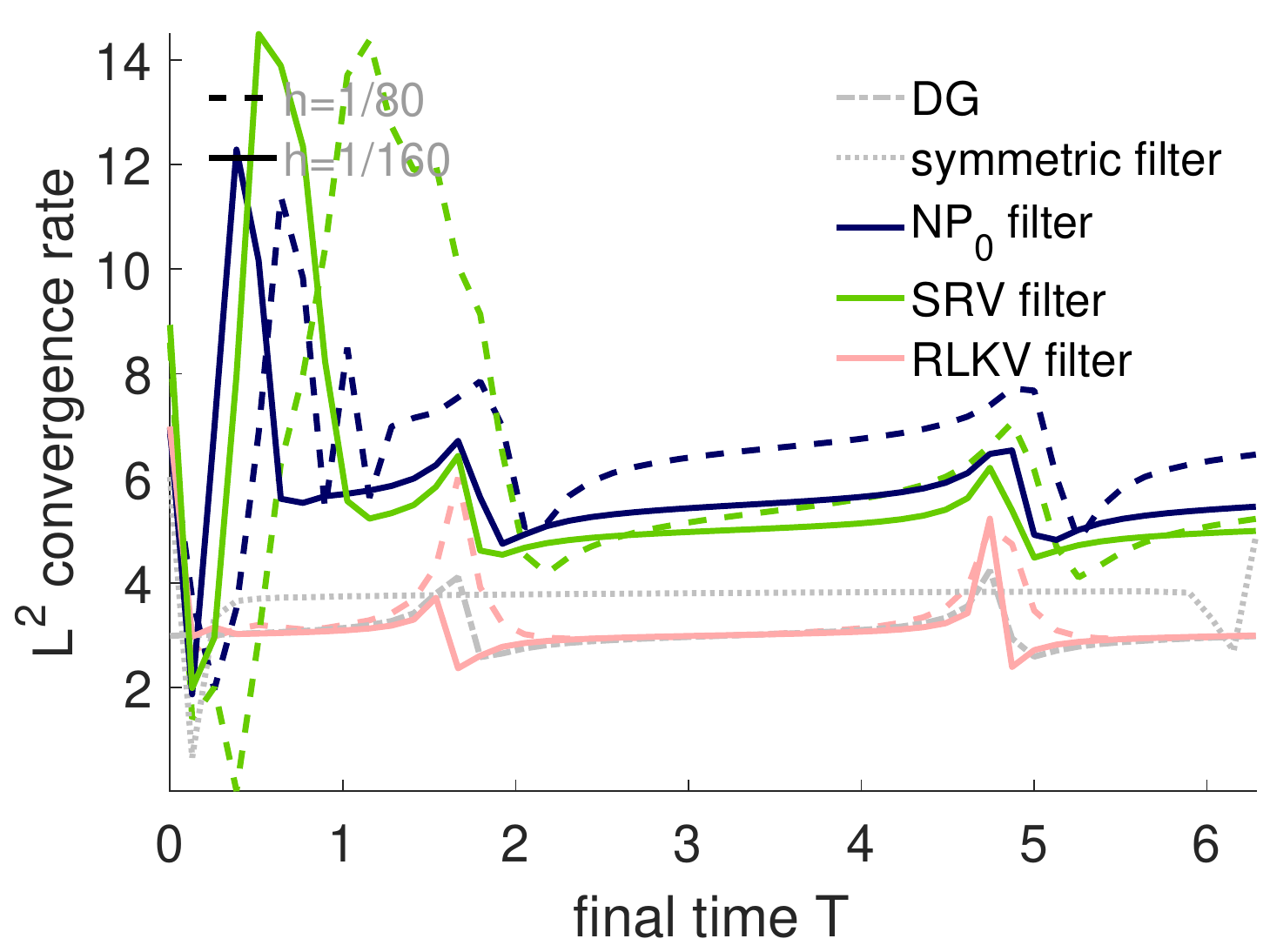}
\includegraphics[width=\wtwo]{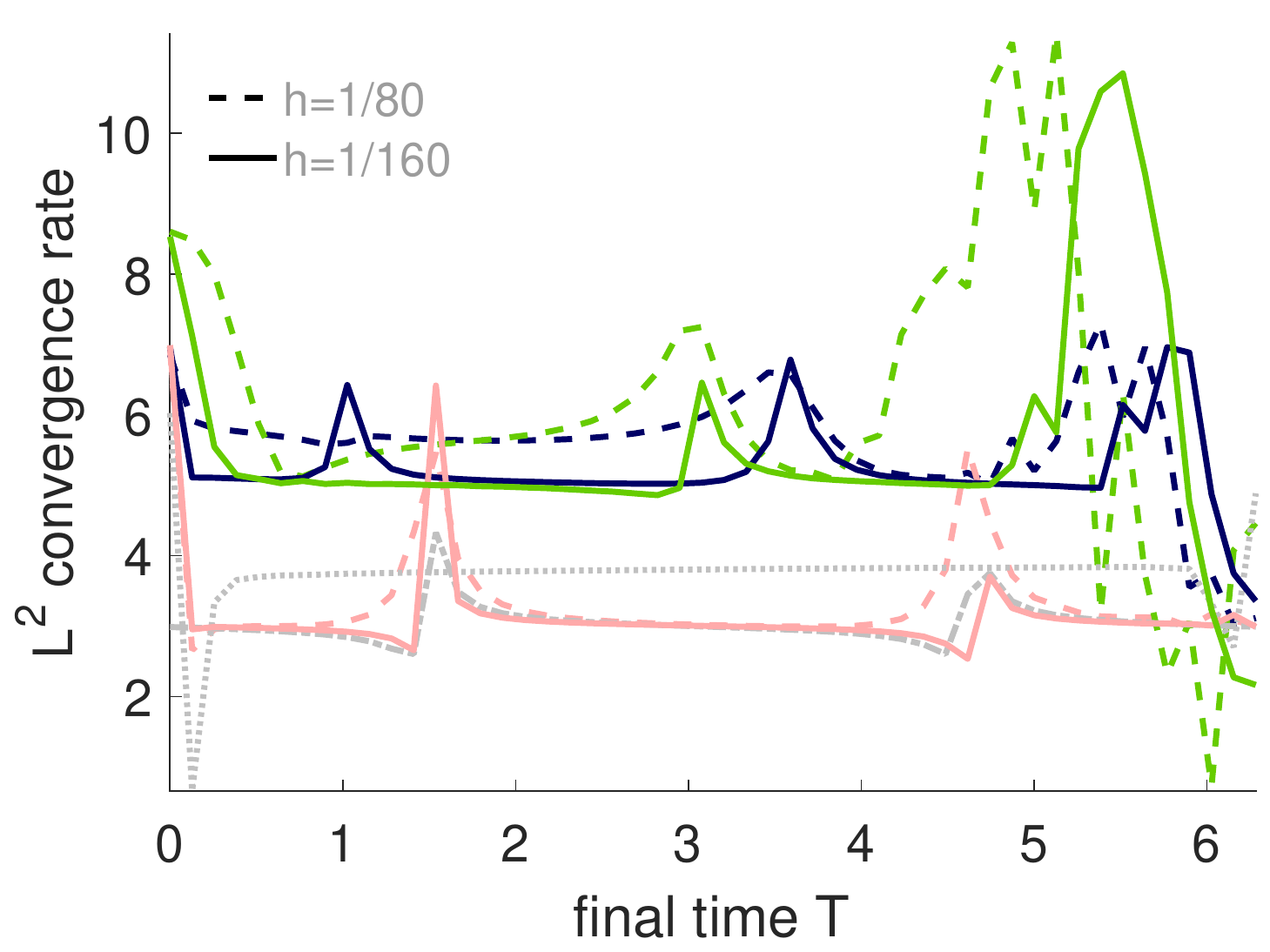}
\includegraphics[width=\wtwo]{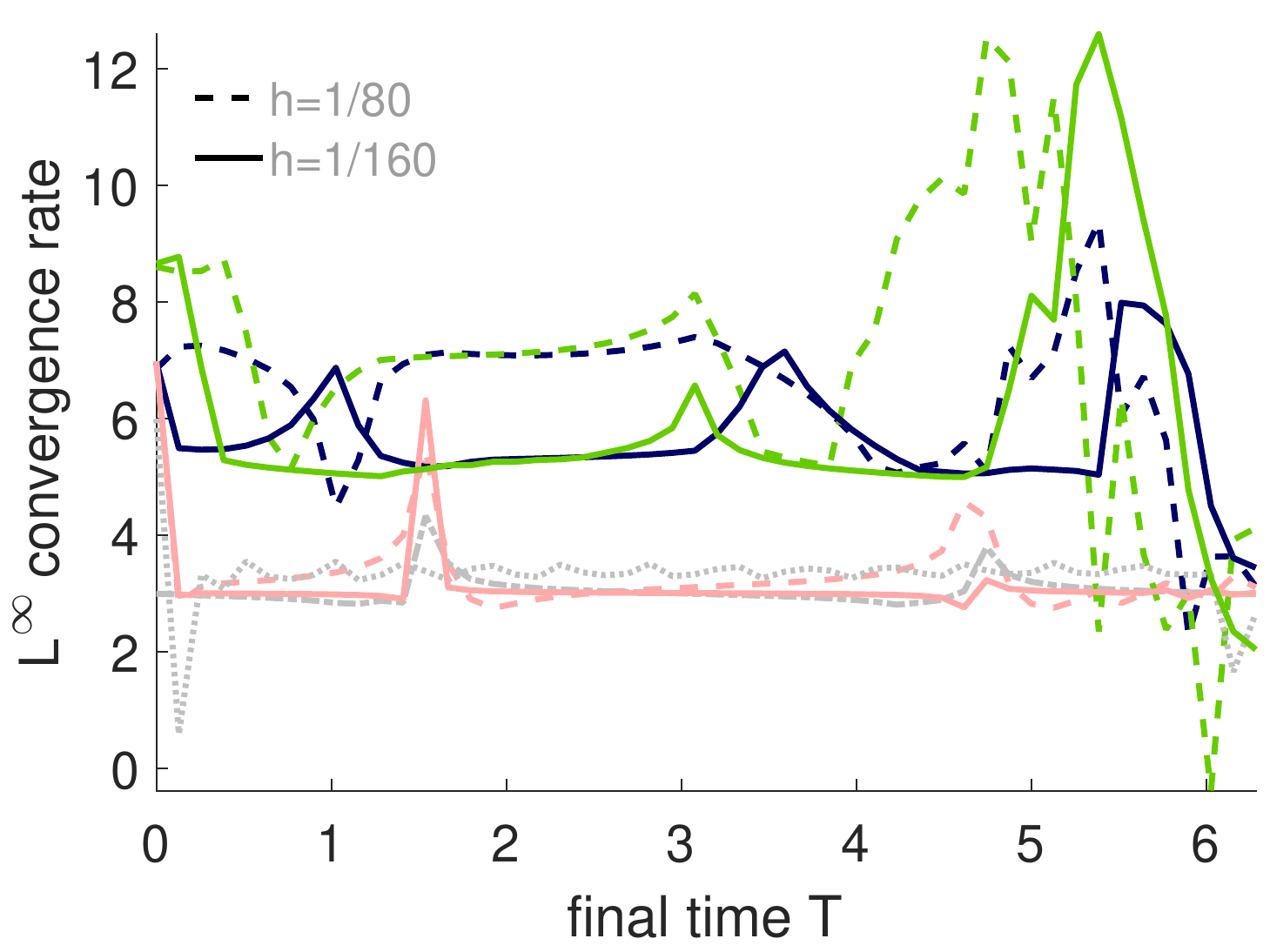}
\\
3 \includegraphics[width=\wtwo]{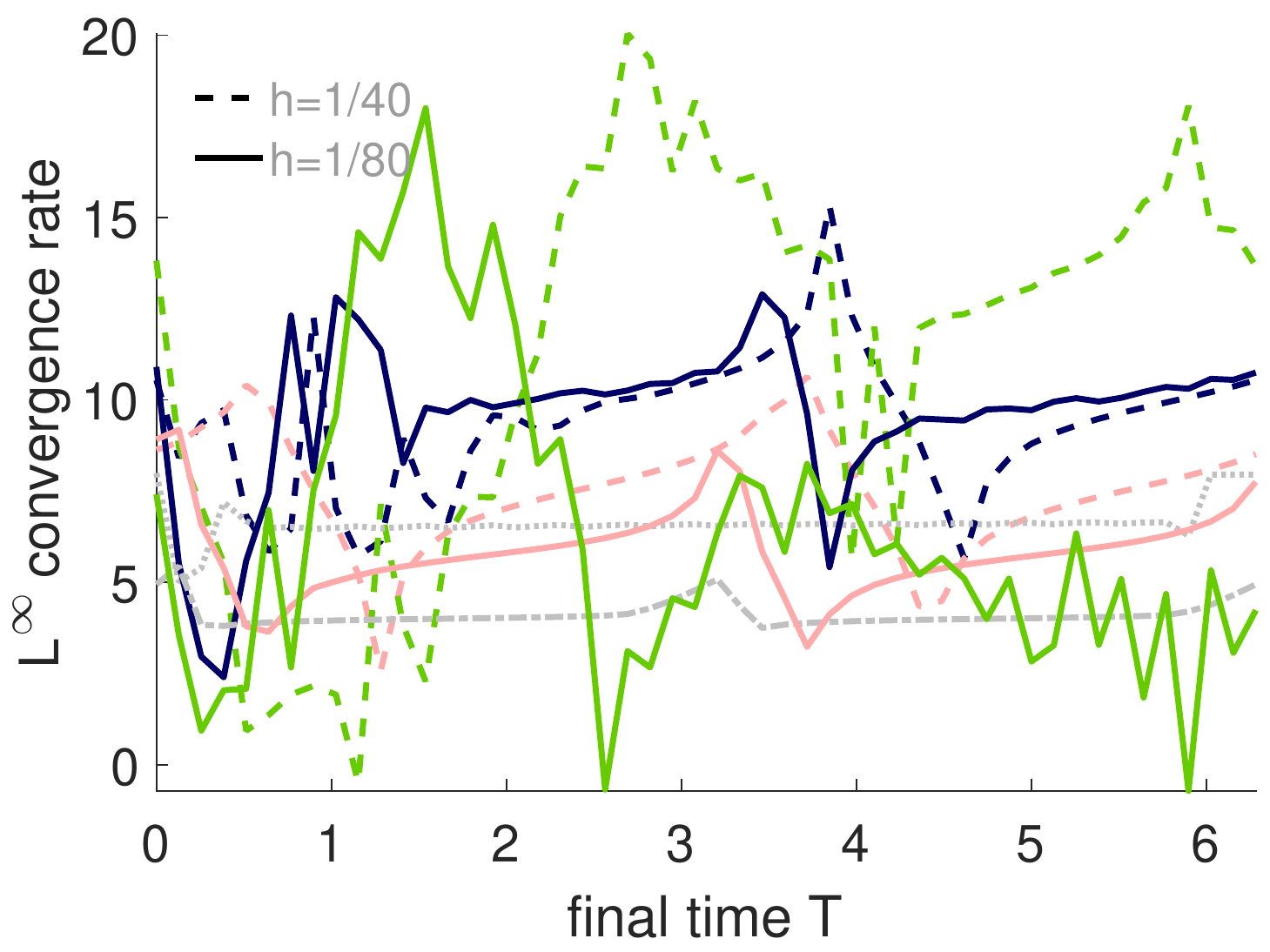}
\includegraphics[width=\wtwo]{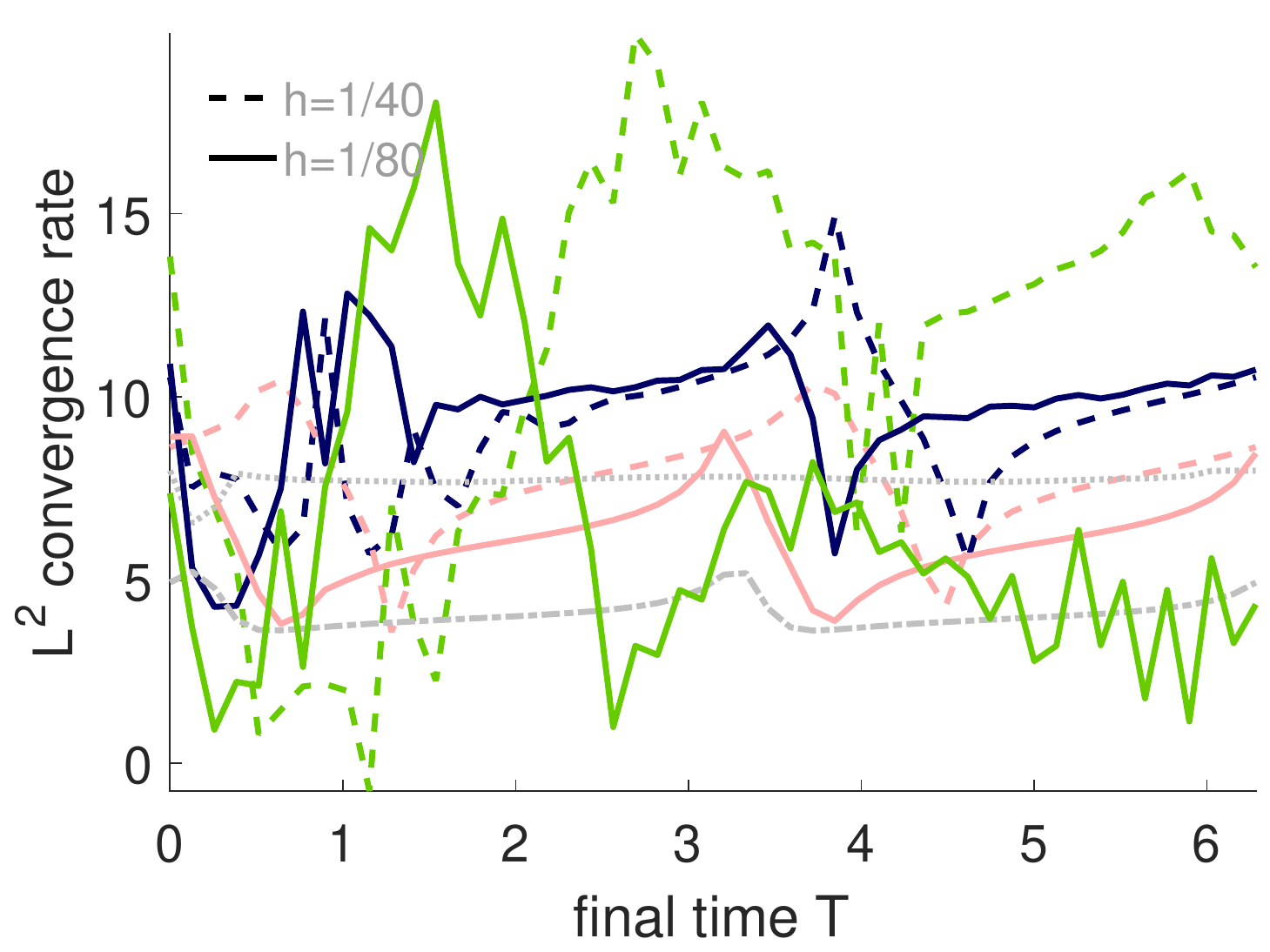}
\includegraphics[width=\wtwo]{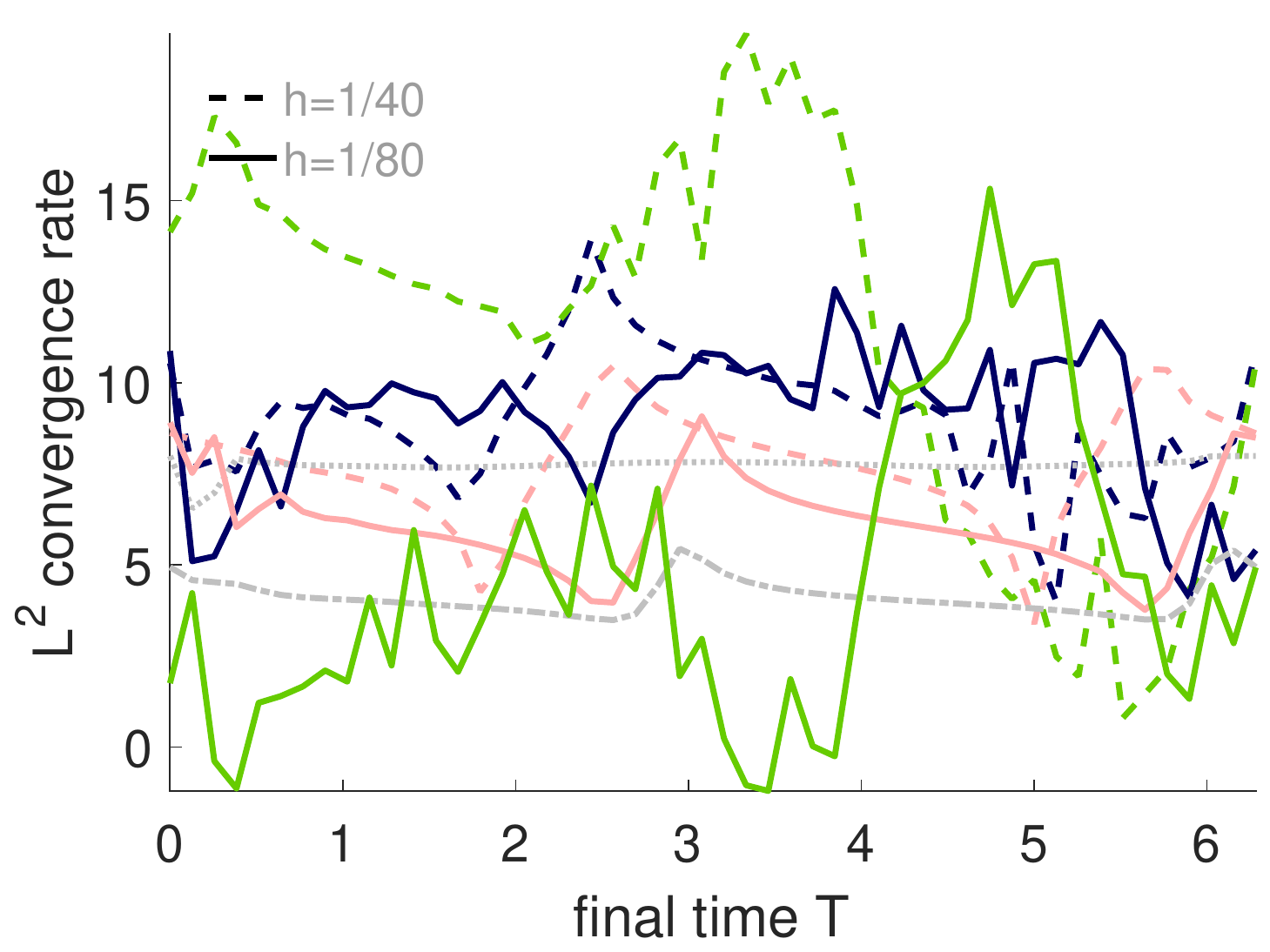}
\includegraphics[width=\wtwo]{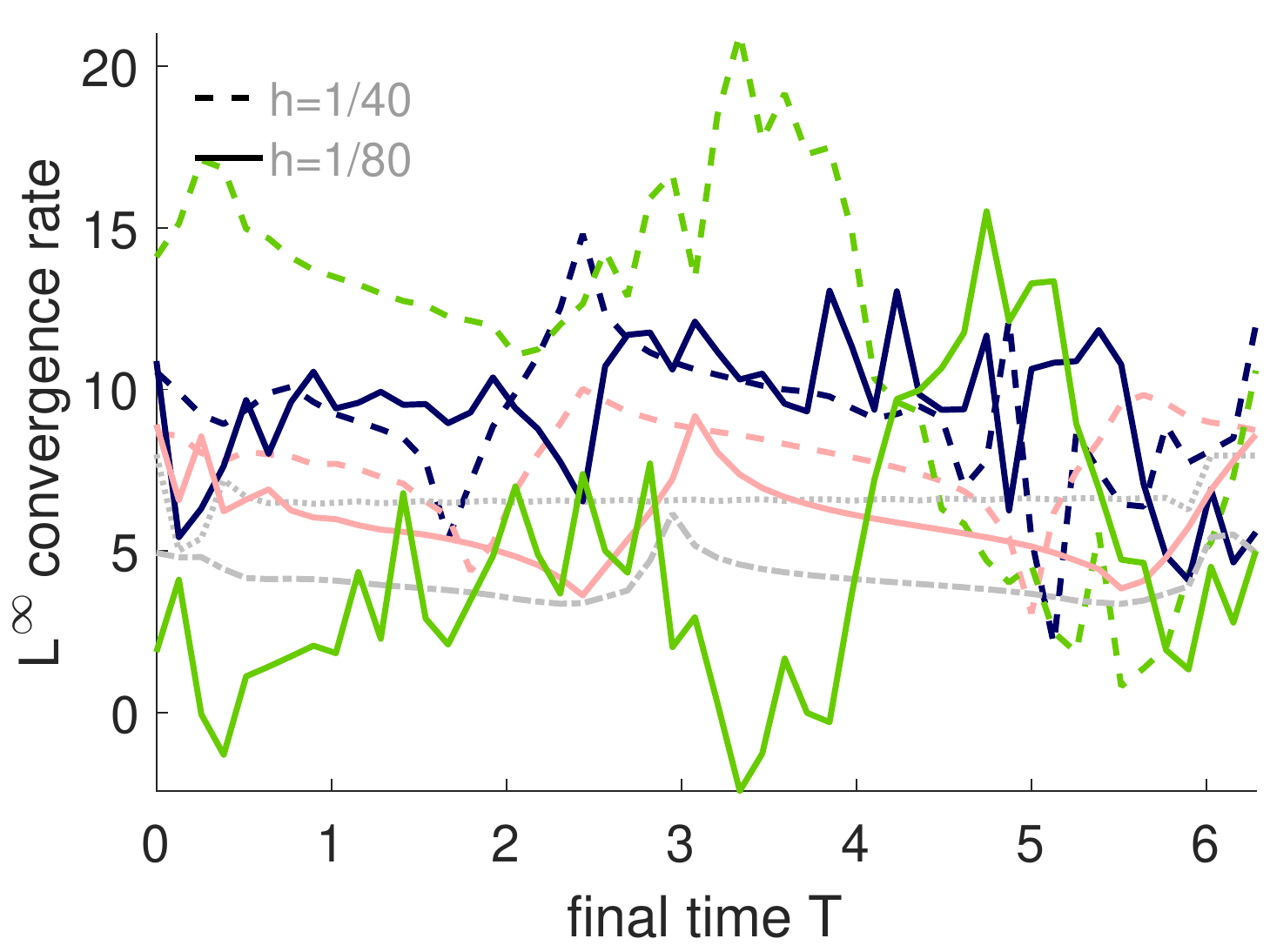}
\mylabel
\caption{
\probRef{prob:cs-d}{constant}{Dirichlet}
\capRateAll{1,2,3}
}
\label{fig:dbc} 
\end{figure} 


\begin{figure}[ht!]
\centering 
1 \includegraphics[width=\wtwo]{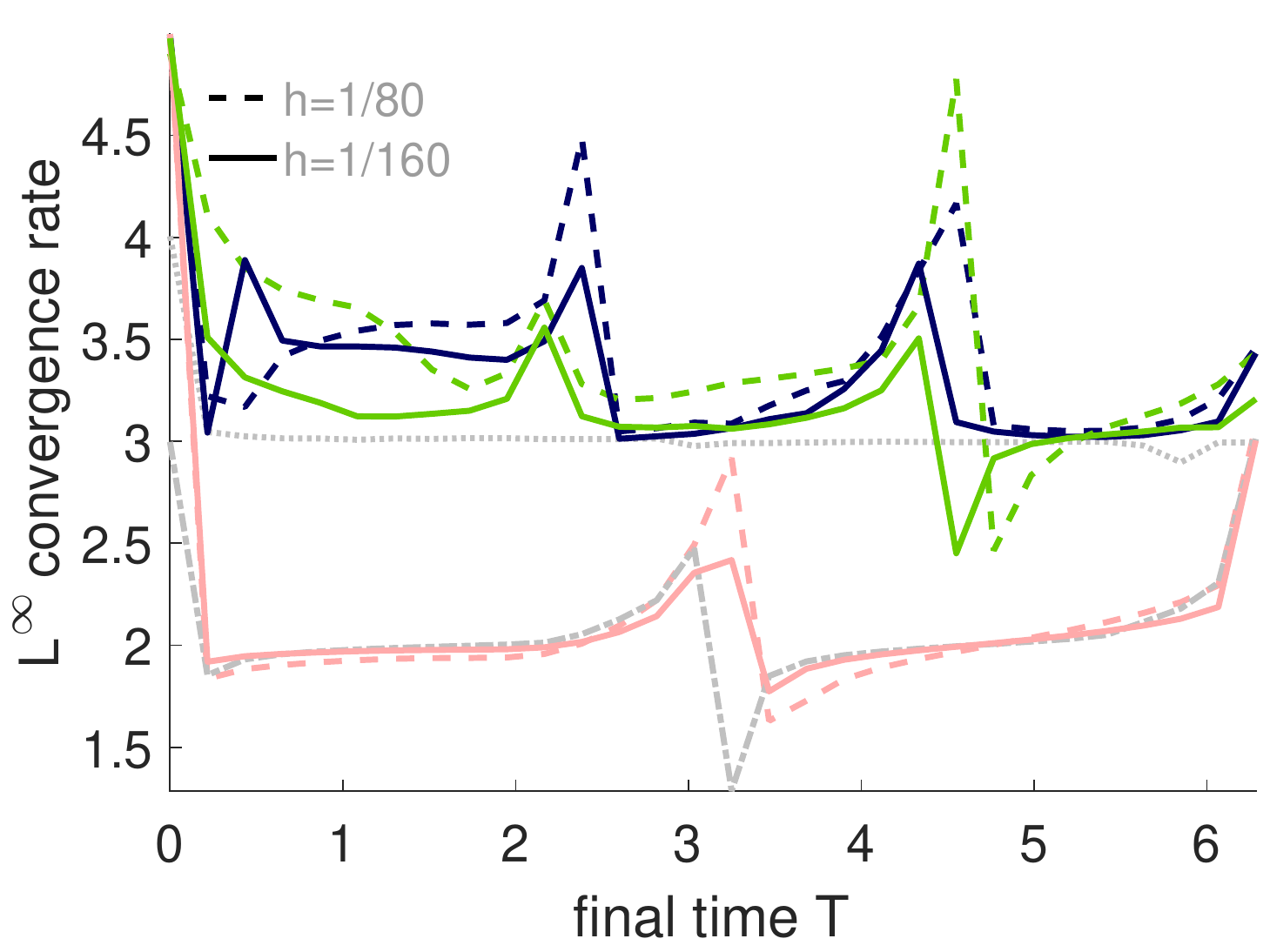} 
\includegraphics[width=\wtwo]{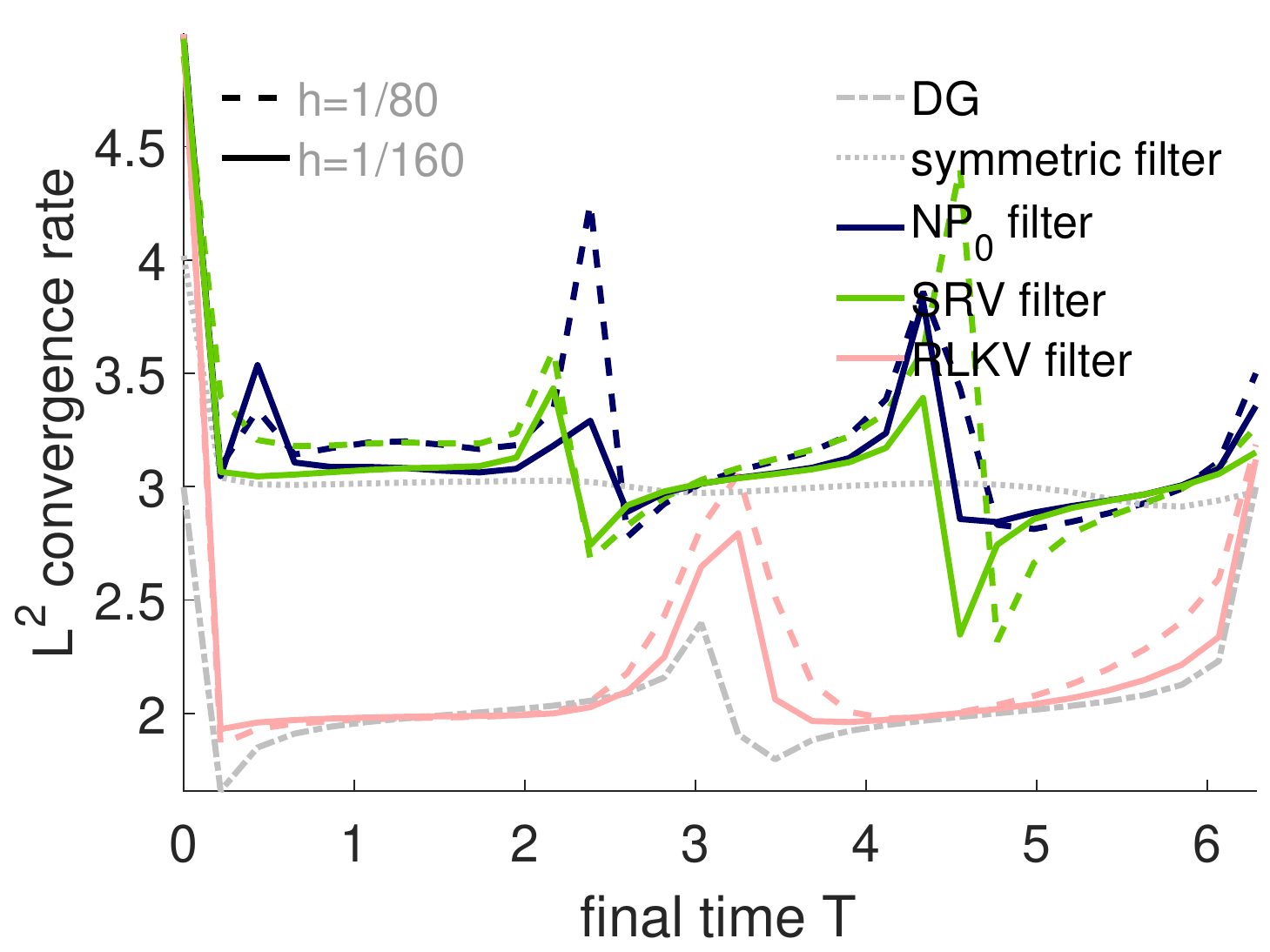} 
\includegraphics[width=\wtwo]{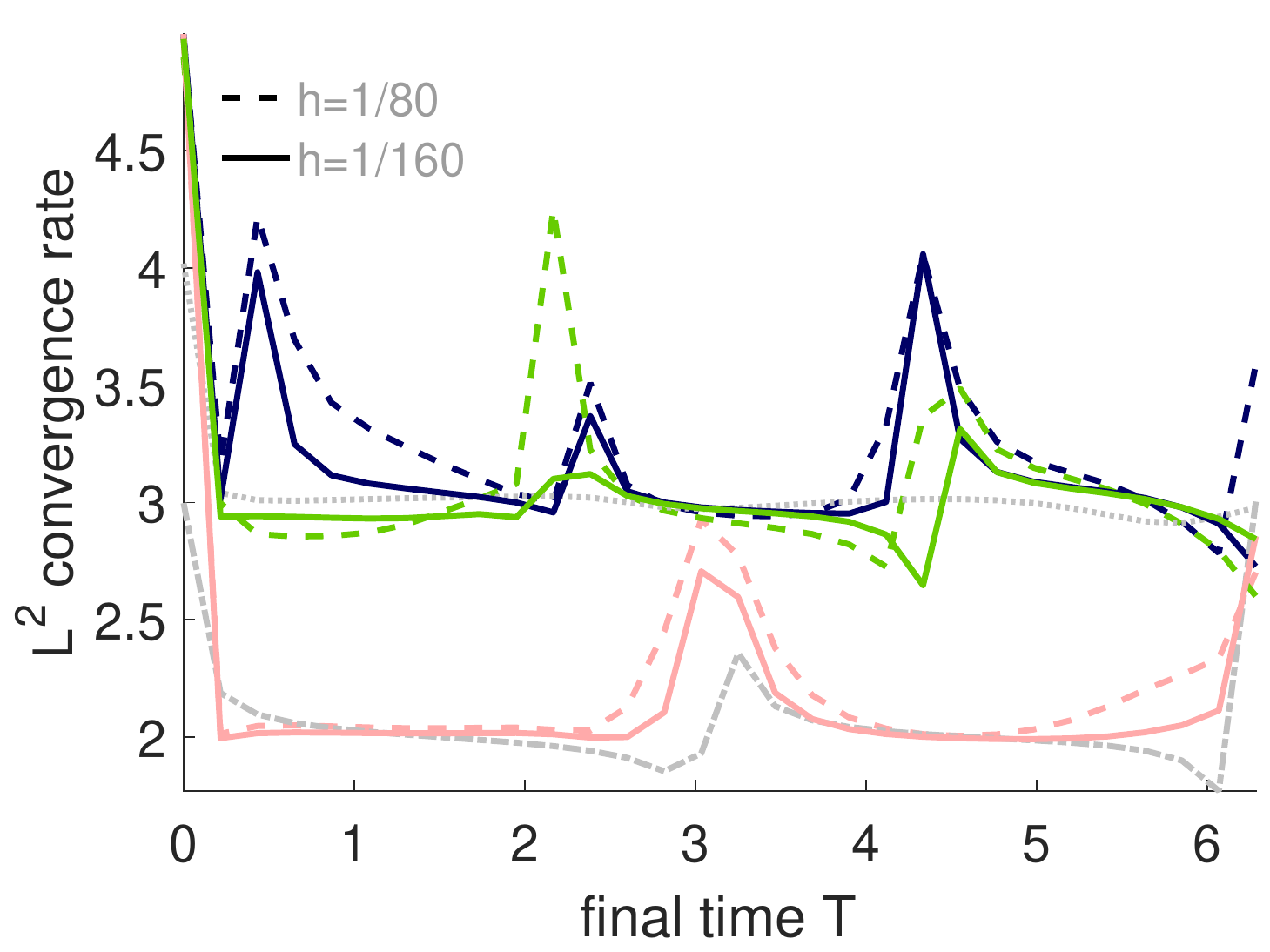} 
\includegraphics[width=\wtwo]{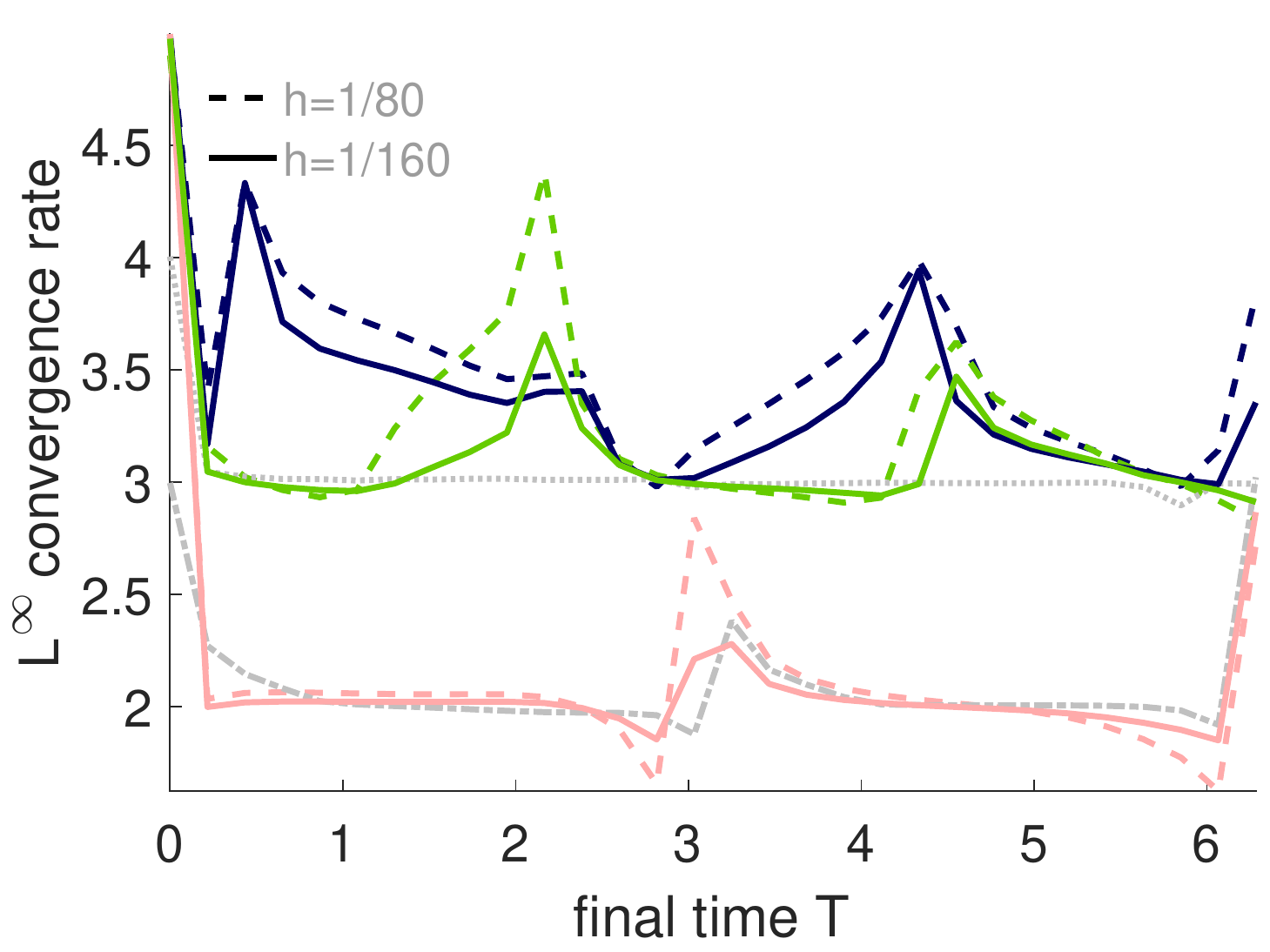} 
\\
2 \includegraphics[width=\wtwo]{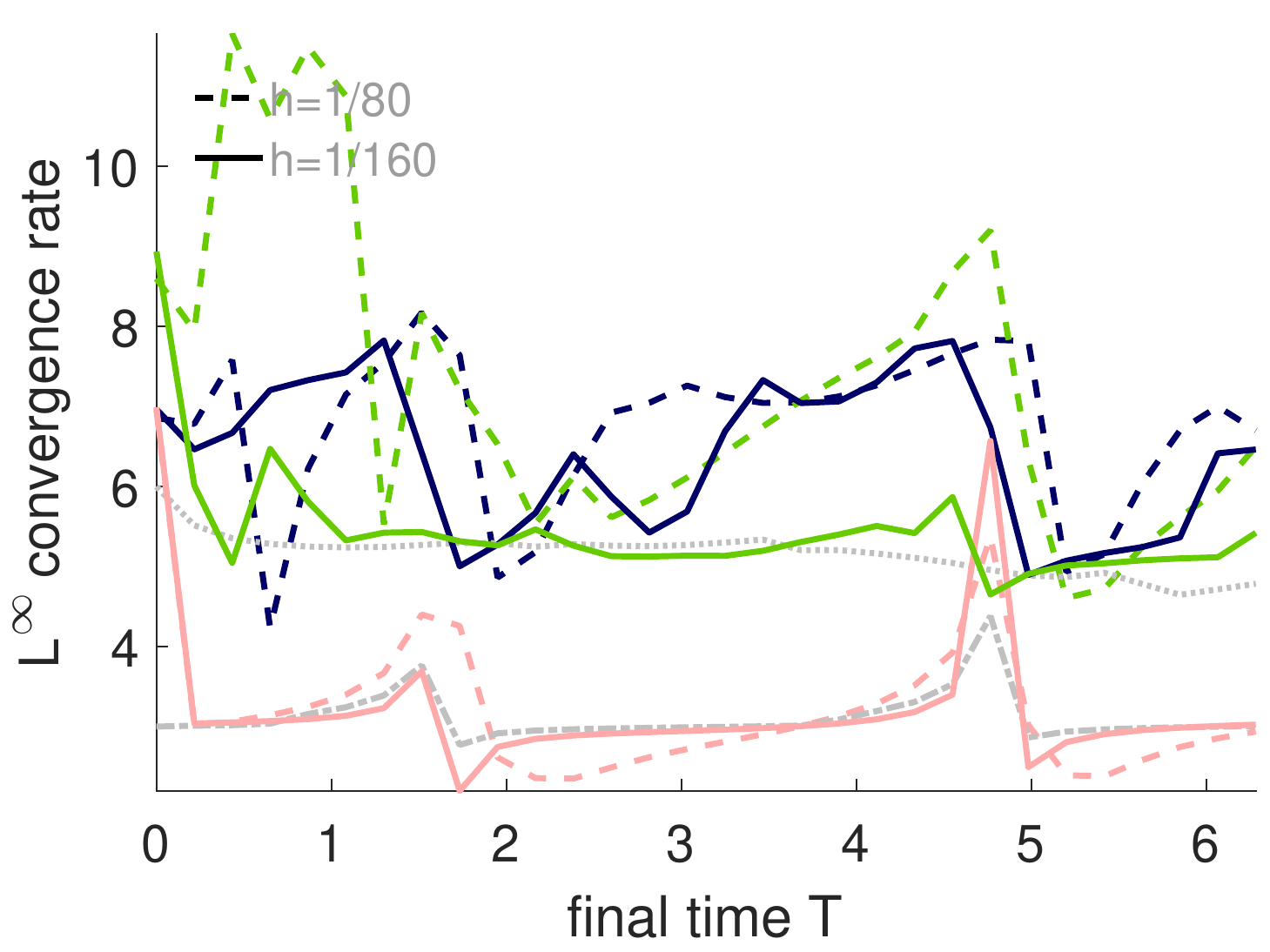} 
\includegraphics[width=\wtwo]{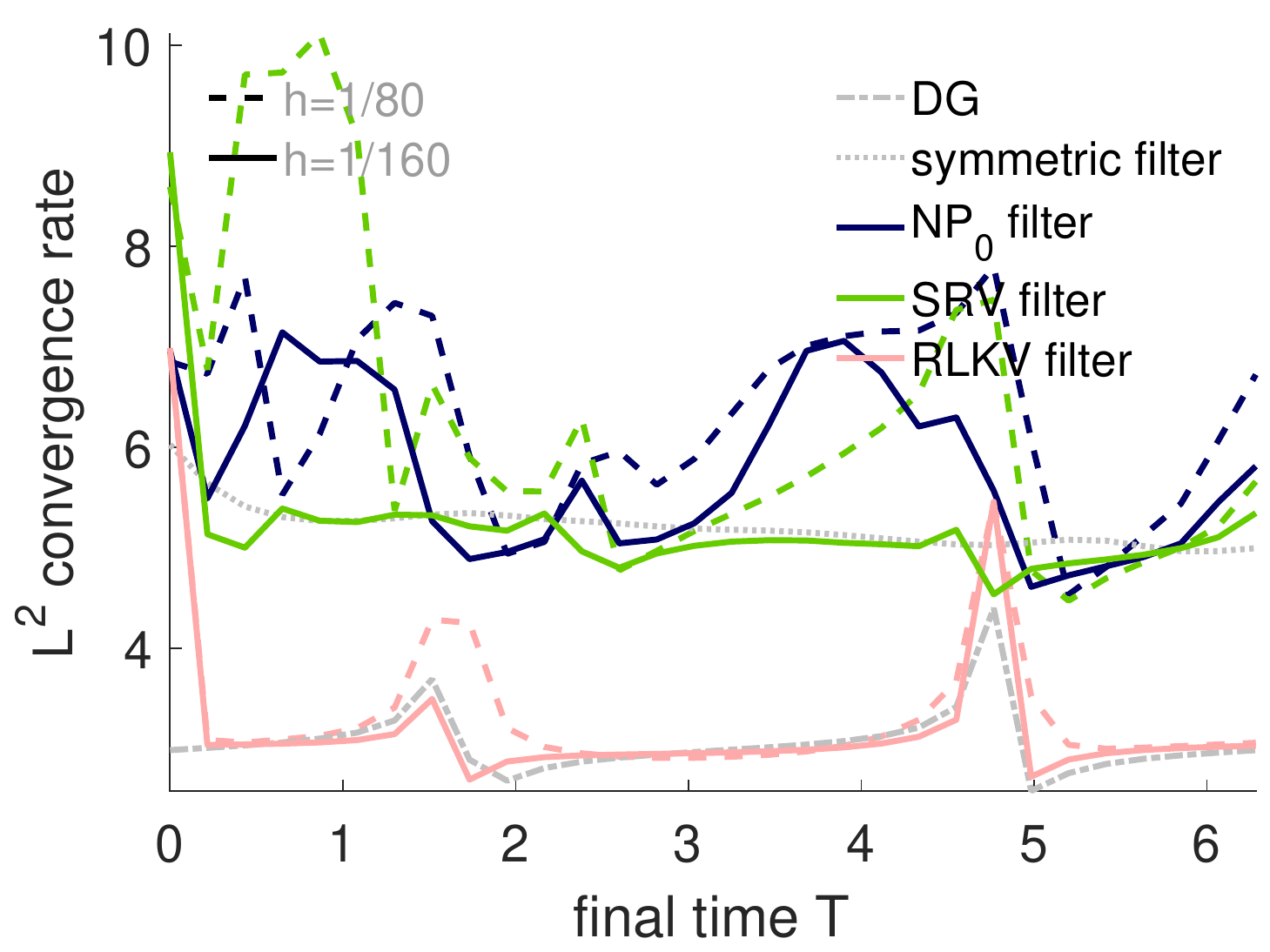} 
\includegraphics[width=\wtwo]{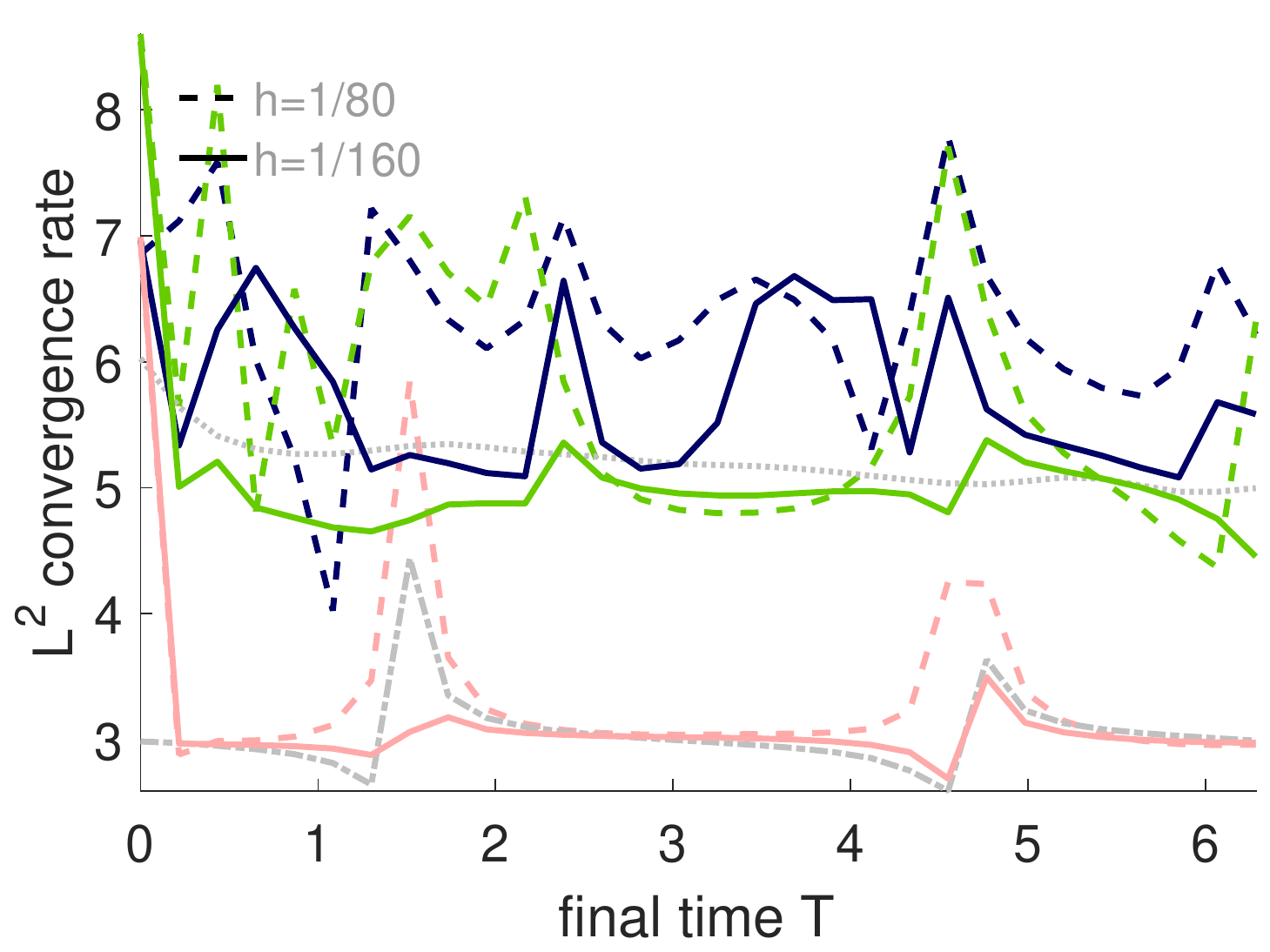}
\includegraphics[width=\wtwo]{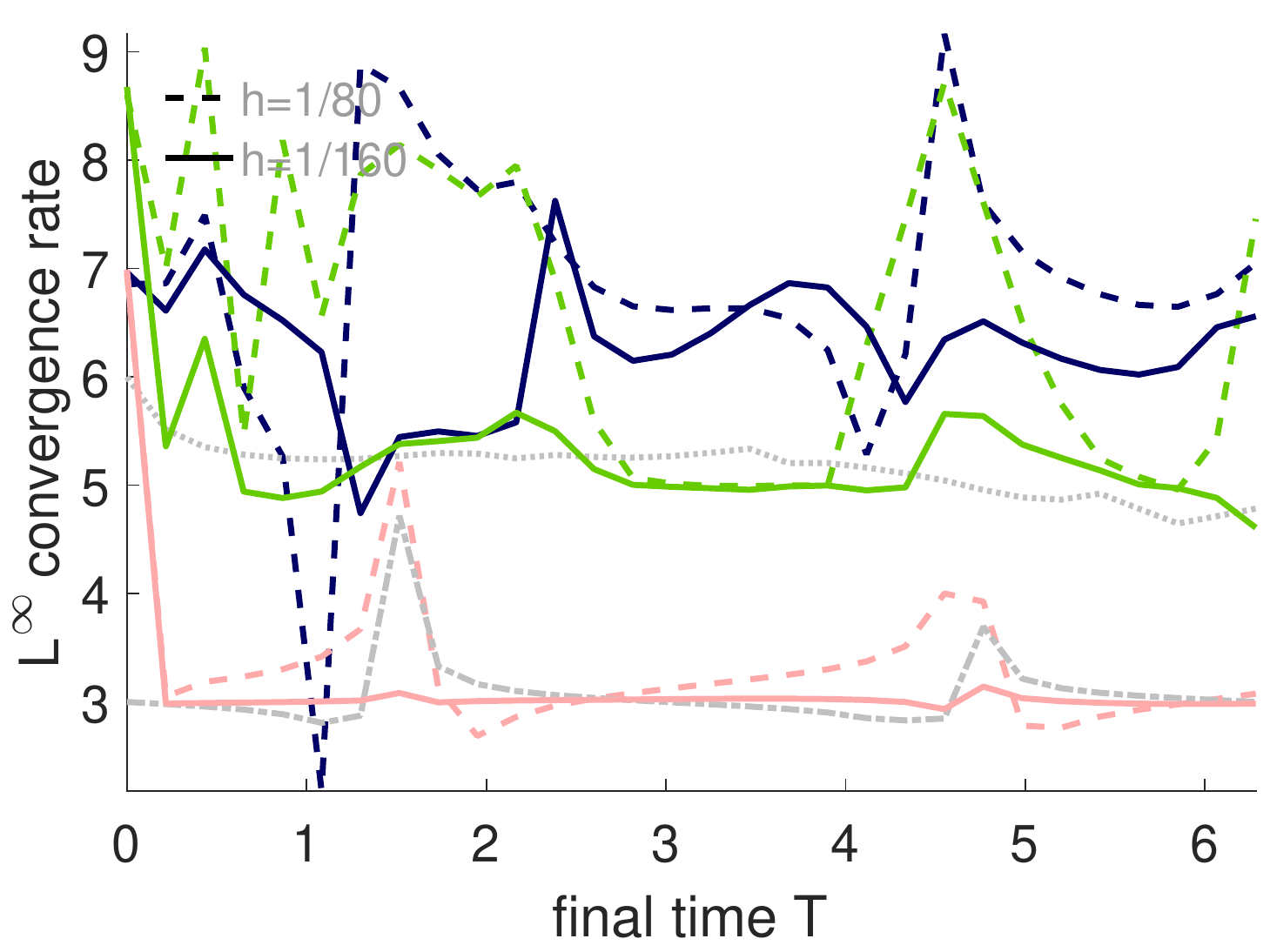}
\\
3 \includegraphics[width=\wtwo]{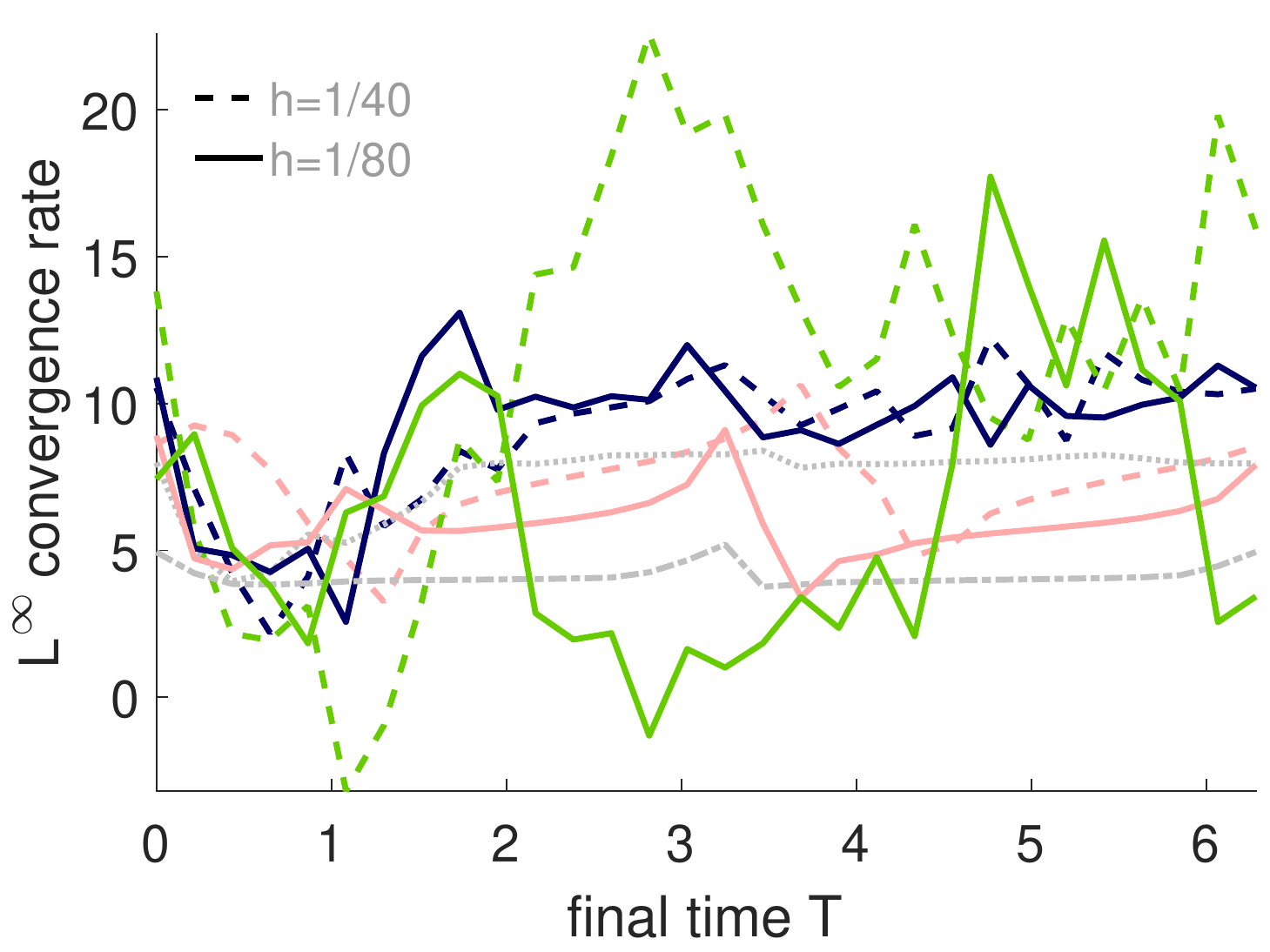} 
\includegraphics[width=\wtwo]{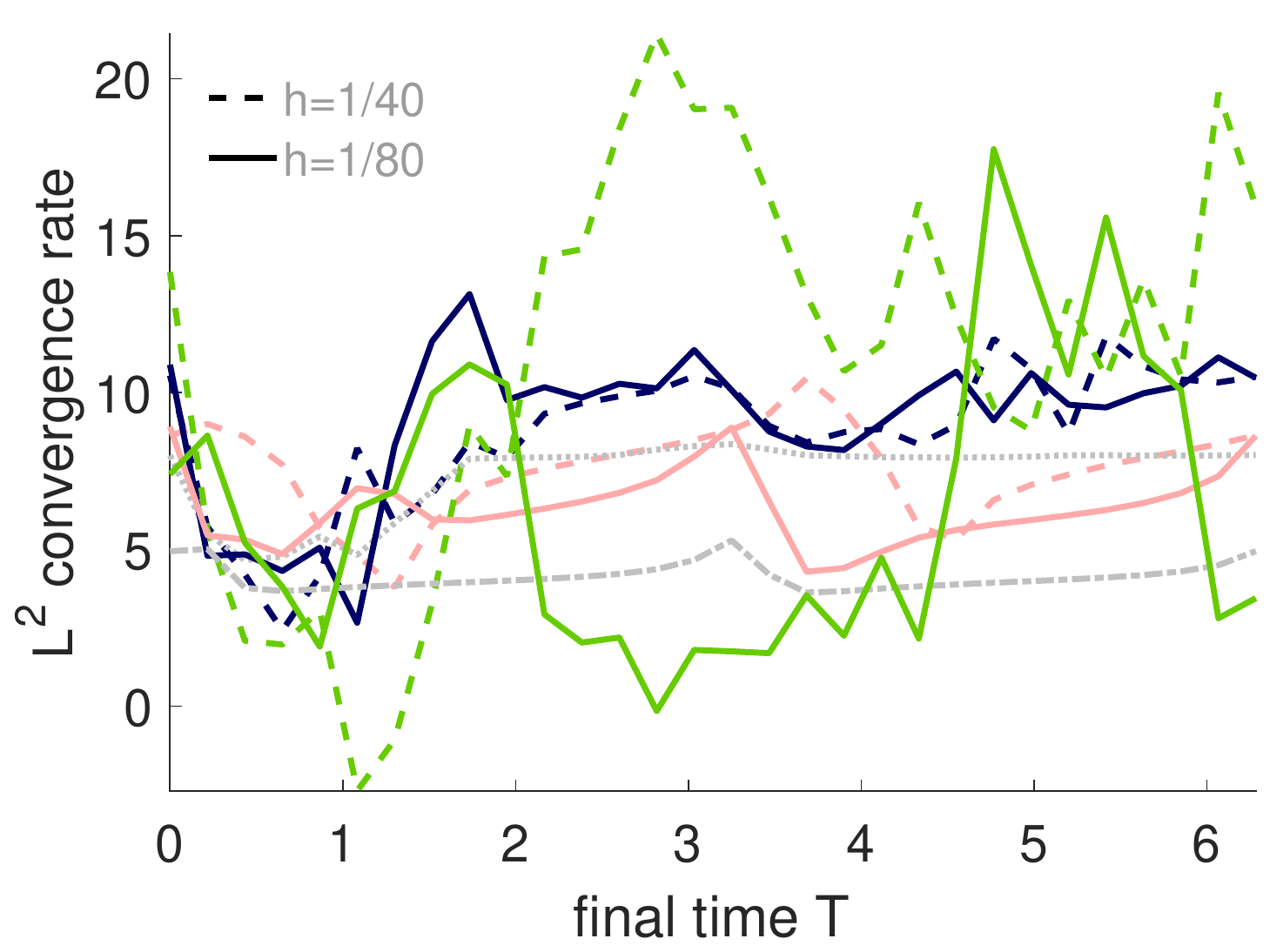}
\includegraphics[width=\wtwo]{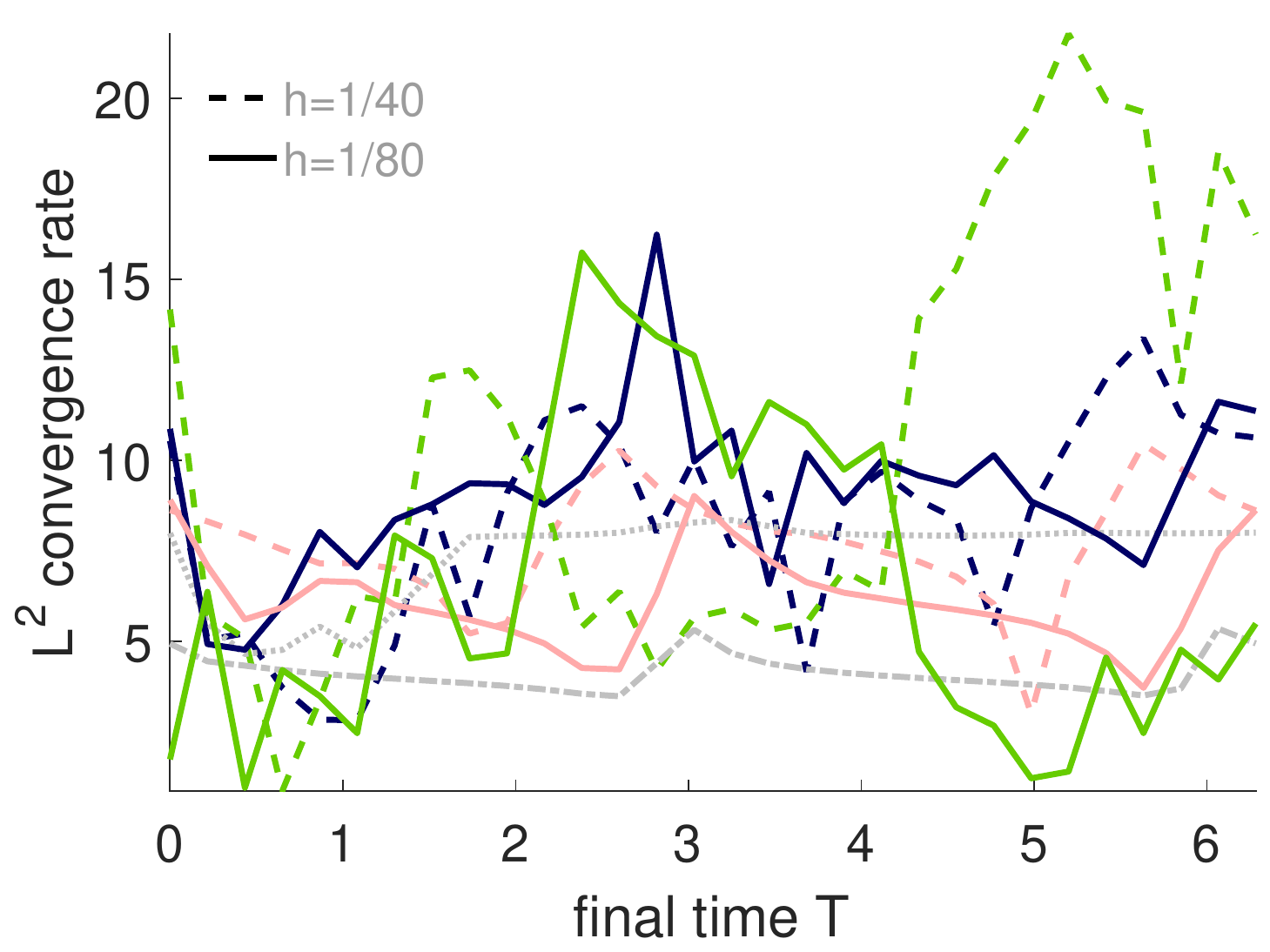}
\includegraphics[width=\wtwo]{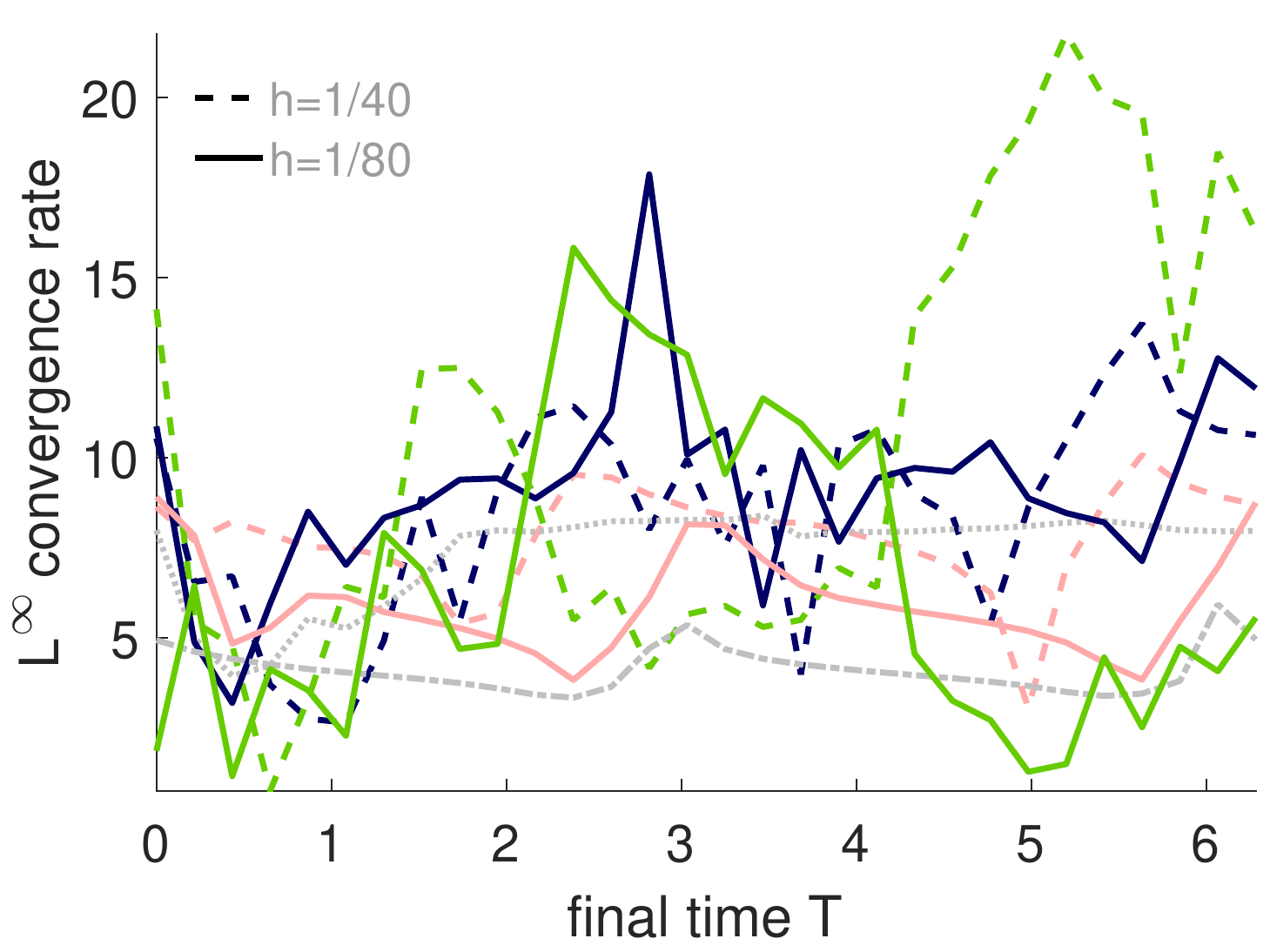} 
\mylabel
\caption{
\probRef{prob:vs-p}{variable}{periodic}
\capRateAll{1,2,3}
}
\label{fig:xt-speed} 
\end{figure}


\end{document}